\documentclass[11pt,reqno,a4paper]{article}
 \usepackage[english]{babel}
\usepackage{times}
\usepackage[left=2cm,right=2cm,top=2cm,bottom=2cm]{geometry}
\usepackage{mathrsfs,stmaryrd}
\usepackage{amssymb,amsthm}
\usepackage[nosumlimits,nonamelimits]{amsmath}
\usepackage{mathtools}
\usepackage{empheq}
\usepackage{enumerate,enumitem}
\usepackage{dsfont}
\usepackage[hidelinks]{hyperref}
\usepackage{xcolor}
\usepackage{cancel}
\usepackage{verbatim}
\usepackage{upgreek}
\usepackage{authblk}
\usepackage{titlefoot} 
\usepackage{graphicx}
\newcommand{\bbGamma}{{\mathpalette\makebbGamma\relax}}
\newcommand{\makebbGamma}[2]{%
	\raisebox{\depth}{\scalebox{1}[-1]{$\mathsurround=0pt#1\mathbb{L}$}}%
}
\renewcommand{\tilde}{\widetilde}

\makeatletter
\@addtoreset{equation}{section}
\makeatother

\theoremstyle{plain}
\newtheorem{theorem}{Theorem}[section]
\newtheorem{proposition}[theorem]{Proposition}
\newtheorem{corollary}[theorem]{Corollary}

\newtheorem{lemma}[theorem]{Lemma}
\theoremstyle{definition}
\newtheorem{remark}[theorem]{Remark}
\newtheorem{example}[theorem]{Example}   

\newtheorem{definition}[theorem]{Definition}

\newtheorem{notation}[theorem]{Notation}

\DeclareMathOperator*{\esssup}{ess\,sup}

\newcommand{\rp}{\ensuremath{\boldsymbol\beta}}
\newcommand{\rpp}{\ensuremath{\beta\!\!\!\beta}}
\newcommand{\T}{\ensuremath{{\mathbb T}}} 
\newcommand{\R}{\ensuremath{\mathbb R}} 
\renewcommand{\tt}{\ensuremath{{\mathcal T}}}
\newcommand{\D}{\ensuremath{\mathcal{D}}}

\newcommand{\V}{\ensuremath{\mathcal V}}

\newcommand{\GG}{\ensuremath{{\mathbb{G}}}}
\newcommand{\G}{\ensuremath{{\mathbf{G}}}}
\newcommand{\A}{\ensuremath{{\mathbf{A}}}}
\renewcommand{\AA}{\ensuremath{{\mathbb{A}}}}
\newcommand{\B}{\ensuremath{{\mathbf{B}}}}
\newcommand{\BB}{\ensuremath{{\mathbb{B}}}}

\newcommand{\ZZ}{\ensuremath{{\mathbb{Z}}}}
\newcommand{\JJ}{\ensuremath{{\mathbb{J}}}}
\newcommand{\II}{\ensuremath{{\mathbb{I}}}}
\newcommand{\LL}{\ensuremath{{\mathbb{L}}}}
\newcommand{\RD}{\ensuremath{{\mathcal{RD}}}}
\newcommand{\dd}{\ensuremath{\partial_x}}

\begin{document}
	
	\title{A pathwise stochastic Landau-Lifshitz-Gilbert equation with application to large deviations}
	
	
	\author{Emanuela Gussetti}
	\affil{\small Bielefeld University, Germany}
	
	\author{Antoine Hocquet}	
	\affil{\small Technische Universit\"at Berlin, Berlin, Germany}

	\maketitle
	
	\unmarkedfntext{\textit{Mathematics Subject Classification (2020) ---} 
		60H10, 
		60H15, 
		60L50, 
		60L90. 
	}
	
	\unmarkedfntext{\textit{Keywords and phrases ---} Stochastic PDEs, Rough paths, Large deviations, Landau-Lifshitz-Gilbert equation}
	
	\unmarkedfntext{\textit{Mail}: \textbullet$\,$ egussetti@math.uni-bielefeld.de $\,$\textbullet$\,$ antoine.hocquet@tu-berlin.de}
	
	\begin{abstract}
		Using a rough path formulation, we investigate existence, uniqueness and regularity for the stochastic Landau-Lifshitz-Gilbert equation
		with Stratonovich noise on the one dimensional torus. As a main result we show the continuity of the so-called It\^o-Lyons map in the energy spaces $L^\infty(0,T;H^k)\cap L^2(0,T;H^{k+1})$ for any \( k\ge1 \).  The proof proceeds in two steps. First, based on an energy estimate in the aforementioned space together with a compactness argument we prove existence of a unique solution, implying the continuous dependence in a weaker norm. This is then strengthened in the second step where the continuity in the optimal norm is established through an application of the rough Gronwall lemma. Our approach is direct and does not rely on any transformation formula, which permits to treat multidimensional noise.
		As an easy consequence we then deduce a  Wong-Zakai type result, a large deviation principle for the solution and a support theorem. 
		\end{abstract}
	
	\tableofcontents
	
	\section{Introduction}
	
	\subsection{Motivation}
	Out of equilibrium, the magnetization of a ferromagnetic material occupying a domain $\mathcal O\subset \R^d$ ($d=1,2,3$) is subject to a time-evolution whose first mathematical description goes back to a short note from Landau and Lifshitz in 1935 \cite{landau1935theory}. In 1954, Gilbert \cite{gilbert1954} used the Lagrangian formalism to clarify the form of the so-called damping term, which accounts for the fact that energy is dissipated along the flow of solutions. 
	In a simplified setting, the latter energy takes the form of a Dirichlet functional
	\[
	\mathcal E(m(t,\cdot)):= \int_{\mathcal O}\frac{A}{m_0^2}|\nabla m(t,x)|^2dx\,,\quad t\geq0\,,
	\]
	where we denote the magnetization by $m\colon \R_+\times\mathcal{O}\to \mathbb R^3$, $m_0>0$ is the constant magnitude and $A>0$ is an exchange constant.
	
	Let $u(t,x):= m_0^{-1} m(t,x) \in\mathbb S^2\subset \R^3$
	and denote by $\times$ the vector product in $\R^3$. Given a sphere-valued initial datum $u(0,\cdot)$, the dynamics is described at a formal level by the geometric partial differential equation
	\begin{equation}\label{llg_original}
	\left \{\begin{aligned}
	\partial_t u &=
	 \mu  u\times \Delta u
	 -\lambda u\times (u\times \Delta u)
	\quad \text{on}\enskip [0,T]\times \mathcal O
	\\
	& |u_t(x)|=1\quad \forall (t,x)\in [0,T]\times \mathcal O
	\end{aligned}\right .
	\end{equation}
	where, ignoring physical units, $\lambda,\mu\geq0$ are given parameters.
	Using the well-known tri-linear formula  $a\times (b\times c)= (a\cdot c) b - (a\cdot b)c,$ $\forall a,b,c\in\R^3,$ and then observing that
	\begin{equation}\label{id:laplace_1}
	0=\Delta (\tfrac12|u|_{\R^3}^2)= u\cdot\Delta u + |\nabla u|^2
	\end{equation}
	(which is a consequence of the fact that $u$ is sphere-valued), one ends up with the formal identity
	\begin{equation}\label{id:main_tt}
	-u\times (u\times \Delta u)=\Delta u + u|\nabla u|_{\R^{3\times d}}^2 = :\tt_u\,.
	\end{equation}
 A useful observation is then that $\tt_u\colon \R_+\times \mathcal O\to \R^3$ is also the vector field that results from differentiation of minus the energy under the unit sphere constraint.
 Namely, if $\langle\cdot,\cdot\rangle$ is the $L^2(\mathcal O;\R^3)$ inner product, we have for every $\varphi \in L^2(\mathcal O;\R^3)$
	\[
	\langle\tt_u,\varphi\rangle := -\lim_{\epsilon\to 0}\frac1\epsilon\left (\mathcal E\Big(\frac{u(\cdot)+\epsilon\varphi(\cdot)}{|u(\cdot)+\varphi(\cdot )|_{\R^3}}\Big)-\mathcal E(u(\cdot))\right )\,.
	\]
	When $\mu =0$, the sytem \eqref{llg_original} therefore reduces to the gradient-type equation $\frac{\partial u}{\partial t}=-\lambda\frac{\partial \mathcal E}{\partial u}$, also known as the \textit{heat flow} of harmonic maps from $\mathcal O$ into the unit sphere $\mathbb S^2$, and was independently investigated by geometers in the 60's to determine which map from a manifold $\mathcal O$ to another (here the sphere) can be continuously deformed to a harmonic map, i.e.\ a minimizer of the Dirichlet energy (\cite{eells1964harmonic}).
	For analysts, this is also a ``toy model'' for \eqref{llg_original}, simplifying the analysis of \eqref{llg_original} but yet capturing some of its key features  such as regularity properties, blow-up phenomenon or large time behaviour (see, e.g., \cite{alouges1992global}).
	For an overview of \eqref{llg_original}, see \cite{lakshmanan2011} and the references therein.\\

	The introduction of a stochastic forcing term to model thermal fluctuations in \eqref{llg_original} is due to Brown \cite{brown1963thermal, brown1978}, and its importance in applications regarding e.g.\ information storage is now widely acknowledged (see for instance \cite{hubert2008}). 
	Roughly speaking, the idea is to add space-time Gaussian white noise $\xi$ in $\R ^{3}$ to the laplacian $\Delta u$. More precisely, the term
	\begin{equation}\label{bilinear_noise}
	\mu u\times \xi - \lambda u\times (u\times \xi)
	\end{equation}
	shall be added to the right hand side of \eqref{llg_original}.
	That being said, there is formal evidence -- at least for dimensions $d\geq2$ -- that the resulting equation is ill-posed (this is a well-known general phenomenon concerning SPDEs, see \cite{hairer2012triviality} for an example). 
	To the best of our knowledge, it remains unclear whether the renormalization techniques recently developed by Hairer and co-authors (\cite{hairer2014theory,bruned2019algebraic}) could give a mathematically consistent meaning to the corresponding evolution, see however \cite{gerenczer2019solution} for the treatment of a quasilinear ansatz akin to \eqref{llg_original}.

	On the other hand, the rigorous mathematical treatment of the corresponding stochastic partial differential equation started only about a decade ago, where aiming at using standard infinite-dimensional It\^o calculus techniques, the authors considered the case of colored noise \cite{brzezniak2013weak,hocquet2015landau}. 
	The corresponding analysis is however  non-trivial since the deterministic PDE arguments for a  nonlinear system such as \eqref{llg_original} do not always blend well with tools from stochastic analysis. A  basic technical obstacle is for instance the anticipative nature of the mild formulation, which in the case of existence and uniqueness of local solutions renders the standard fixed-point argument inoperative, see nonetheless \cite{kuehn2018pathwise}.
	
	As in the deterministic context, a possibility is to work at the level of the corresponding \textit{stochastic} heat flow (i.e.\ with $\mu =0$). 
	This path was followed by one of the present authors in \cite{hocquet2018struwe,hocquet2019finite}, leading to a two-dimensional theory which carries some similarities with the deterministic model studied by Struwe in the 80's (\cite{struwe1985}). 
	In \cite{hocquet2019finite}, it was shown for $d=2$ that equivariant solutions can blow-up in finite time, and that the probability of such an event is positive provided the noise has a certain symmetry properties. 
	This scenario happens no matter how the initial datum is chosen, in contrast with the deterministic scenario where the initial datum determines which solution will explode. Though the generalization to the case $\mu \neq0$ is yet unclear, this particular fact shows that the noisy Landau-Lifshitz-Gilbert equation may give rise to complex phenomena that cannot be reproduced without noise. 
	In particular, it is an open problem to determine whether a suitable non-symmetric noise could lead to the opposite behaviour, namely to \textit{regularize} the solutions.
	
	These questions amount in a way to asking whether such singularities are stable under singular perturbations of the right hand side of \eqref{llg_original}. While some results in that direction are available in a deterministic setting \cite{raphael2013stable,merle2013blowup}, the stochastic counterpart fails to be amenable to the techniques used there to show (in)stability of blow-up for \eqref{llg_original}. This is partially due to the basic fact that the solution map $\xi\mapsto u$ fails to be continuous in general. 
	This well-known drawback of It\^o's calculus was one of Lyons's motivations to introduce the theory of \textit{rough paths} \cite{lyons1998differential,lyons2002system}, where continuity of the solution map is recovered by adding some extra components to the noise, namely,  iterated integrals of the underlying Wiener process with itself. 
	Lyons' theory is particularly versatile in finite dimensions were it was extensively used in the context of stochastic differential equations and stochastic analysis (see \cite{FrizVictoire,FrizHairer} and the references therein). 
	Extending such developments to stochastic \textit{partial} differential equations is a very active direction of research, whether the noise is white (see e.g. \cite{gubinelli2018panorama} for an overview), colored or finite-dimensional (\cite{gess2017long,hofmanova2020navier}).

	In the following, we fix the dimension $d=1$ and let $\mathcal O=\T:=\R/\mathbb Z$. 
	We will let $\lambda =\mu =1$ and fix an
	arbitrary finite time horizon $T>0.$
	We consider a noise term which is white in time and colored in space, given as the time-derivative of a $Q$-Wiener process $w\colon \Omega\times [0,T]\to L^2(\T ,\R^3)$.
	Fixing an orthonormal basis $(e_k)_{k\geq 1}\subset L^2(\T ,\R^3)$, we assume that the latter can be written as the infinite sum
	\begin{equation}
	\label{sum_w}
	w(\omega,t,x):=\sum_{k\in\mathbb{N}}g _k(x)\beta_k(\omega,t),\quad (\beta_k)_{k\in\mathbb{N}}\enskip =\text{an i.i.d.\ family of Brownian motions in }\R\,,
	\end{equation}
	where $g _k=\sqrt{Q} e_k,$ for some non-negative, trace-class operator 
	$Q\in \mathcal L(H^s(\mathcal O;\R^3)),$ $s>\frac 12.$ 
	As alluded to earlier, we want to do the singular substitution $\Delta u \leftarrow \Delta u+\frac{dw}{dt}$ in the right hand side of \eqref{llg_original}.
	For simplicity we only do it for the gyromagnetic term (i.e.\ the first term in the right hand side). The version of the Stochastic Landau-Lifshitz-Gilbert equation that we are concerned with then writes as the \textit{Stratonovich} SPDE
	\begin{equation}
	\label{LLG}
	\tag{LLG}
	\left\{\begin{aligned}
	& d u-(\partial_x^2 u+u|\partial_x u|^2 + u\times\partial_x^2 u) d t=u\times\circ d w \,,
	\quad \text{on}\enskip [0,T]\times\T 
	\\
	&|u_t(x)|=1\,,\quad \forall (t,x)\in [0,T]\times\T 
	\\
	&u_0=u^0 
	\end{aligned}\right.
	\end{equation} 
	where $d\in\{1,2,3\}.$ We note that the last term in the drift may alternatively be written as the vector product of $u$ with $\tt_u$, which is sometimes useful in the computations. Indeed, by orthogonality we have formally $u\times \partial_x^2 u = u\times (\partial_x^2 u + u|\partial_x u|^2) =u\times\tt_u $.
	As for our choice to keep noise only in the gyromagnetic term, we can loosely justify it as follows. 
	First, one can simply argue that in physically relevant contexts, the damping constant $\lambda $ is typically small in comparison with $\mu .$
	If we leave aside mathematical rigor, the second reason why this omission is not completely absurd stems from a geometric argument: replacing the noise in \eqref{LLG} by the full ansatz $\sqrt{2}(u\times \frac{dw}{dt} -u\times(u\times \frac{dw}{dt}))$ leaves invariant the infinitesimal generator of the corresponding (infinite-dimensional) diffusion, meaning in principle that both solutions should have identical laws.
	(For spatially-discretized versions of LLG, this was rigorously established in \cite{neklyudov2013role}, the generator in question being a multi-dimensional variant of the Laplace-Beltrami operator).
	See also the introduction in \cite{hocquet2019finite}, for a related discussion.\\

	In line with the previous discussion on the lack of continuity for the solution map,  another motivation to introduce a \textit{rough path formulation} for \eqref{LLG} is the study of rare events and small probabilities for noises with vanishing amplitude.
	This direction of research is particularly active in the context of micromagnetism where the small noise asymptotics for \eqref{LLG} is needed to study phase transition between different equilibrium states of the magnetization (see for instance \cite{kohn2005}).
	Using Budhiraja and Dupuis' weak convergence method \cite{budhiraja2000variational}, a large deviation principle in 1D was established by Brze\'zniak, Goldys and Jegaraj (see \cite{brzezniak_LDP}), for a noise term of the form \eqref{bilinear_noise}, where $w$ is a Brownian motion in $\R^3$.
	Based on a Doss-Sussmann type transformation introduced by Goldys, Le and Tran \cite{goldys2016finite} for a real valued Wiener process to reduce \eqref{LLG} to a random PDE, Brze\'zniak, Manna and Mukherjee \cite{brzezniak2019wong} showed existence and uniqueness of a pathwise solution in 1D.

	Using this technique, they also derived Wong-Zakai convergence results and showed that the corresponding solution enjoys the maximal regularity property, namely, it belongs to the space $L^\infty(0,T;H^{1})\cap L^2(0,T;H^2)$, for initial conditions in $H^{1}$.
	This transformation technique was further extended in \cite{goldys2020weak} to include space-dependent noise of the form \eqref{sum_w}, where \eqref{sum_w} takes the form of a finite sum.
	The trick is based on solving for each $x\in \mathcal O$ the auxiliary Stratonovich SDE (here $q$ is the largest integer for which $g_k\neq0$)
	\begin{equation}\label{auxiliary_ODE}
	Z(t,x)= \mathbf 1_{3\times3} + \sum_{k=1}^q\int_0^t(\cdot\times g_k(x))Z(s,x)\circ  d\beta_k(s),
	\end{equation}
	where the unknown $Z(t,x)\in \mathcal L(\R^3)$ belongs to a certain Lie group.
	This observation in turn permits to define the new unknown $v(t,x):=Z(t,x)^{-1}u(t,x)$, which satisfies a deterministic PDE with random coefficients. 
	We note that there is, in principle, no essential difficulty to extend this procedure to the case $q=\infty$ as considered in the present paper. However, to the best of our knowledge, it remains unclear how to use the previous transform in order to extend the Wong-Zakai results of \cite{brzezniak2019wong} to our setting.\smallskip
	
	The objective of this paper is twofold. At first, we derive a \textit{robust} pathwise formulation of \eqref{LLG} and obtain energy estimates for it, using tools from the variational approach to rough PDEs (\cite{deya2016priori,hocquet2017energy,hofmanova2019navier,hofmanova2020navier,hocquet2018ito}). The combination of some of the techniques developed in these prior works with a ``rough boostrap'' argument (see Section \ref{subsec:bootstrap}) permit to show that energy estimates of the same kind are true at any level of spatial regularity, as permitted by the regularity of the data. Namely, if the initial datum belongs to $H^k$ with $k\geq1$ and provided the noise is regular enough, we obtain global a priori estimates in the spaces 
	$ L^\infty(0,T;H^k)\cap L^2(0,T;H^{k+1}).$
	We then use these estimates to show Wong-Zakai type convergence statements in these spaces, hence extending the results of \cite{brzezniak2019wong}.
	The proof of our main continuity results (Theorem \ref{thm:wong-zakai} and Corollary \ref{cor:WZ}) consists of two steps: first the proof of existence and uniqueness of a solution by compactness methods, second the continuity of the map that to a noise associates the solution (which implies a Wong-Zakai type convergence). The latter step presents technical difficulties: we need to estimate the remainder of the equation and control the drift in a negative Sobolev norm. 
	In our way to handle this particular step, we introduce a certain pattern of proof referred to as ``rough standard machinery'', which seems to be of independent interest. It is a tool to obtain a priori estimates for rough PDEs, extending to the case of \textit{systems} (and arbitrary high order of spatial regularity) some techniques developed in \cite{deya2016priori,hocquet2017energy}.

	Second, we derive a large deviation principle based on a contraction argument, hence completing the results of \cite{brzezniak_LDP}  (at least when the contribution of the noise in the damping term is neglected). The method we rely on comes as a standard by-product of the rough path formulation and the previous a priori bounds, even though such results for stochastic \textit{partial} differential equations are not so common in the literature (see nevertheless \cite{hairer2015large,fehrman2019large}).
	By following the same approach, we will also derive a support theorem (in the sense of \cite{Strook_Varadhan,Millet_Sanz_Sole}) for the solution to \eqref{rLLG}. To the best of our knowledge, this is the first the time that a support Theorem for a quasi-linear system of SPDEs such as \eqref{LLG} using rough path techniques is derived.  We note that a version of Stroock and Varadhan celebrated theorem has been formulated in the rough path context in \cite{Ledoux_Qian_Zhang_p_var} with respect to the $p$-variation topology and in \cite{FrizVictoire} with respect to the $\alpha$-H\"older topology. 
	Finally, let us point out that our approach based on pathwise energy estimates could lead further to a numerical method, even though it remains unclear at the moment how to simulate the iterated integral needed in our formulation (we leave this question for later investigations).

	\paragraph{Organization of the paper}
	In Section \ref{sec:main} we introduce our notations and present our main results.
	Section \ref{sec:apriori} is devoted to a priori estimates and the preparation of the proof of the main theorem.
	In Section \ref{sec:sewing} we state a multi-dimensional product formula which is at the core of our argument.
	The remainder of the section shows how to obtain the main energy estimates from the product formula. Since the argument seems to generalize well to other contexts, we present a general methodology (refered to as ``rough standard machinery'') which later permits to bootstrap. This method is built upon the formalism of ``rough driver'' developed in \cite{bailleul2017unbounded,deya2016priori} and subsequently analyzed in \cite{hocquet2017energy,hocquet2018ito}.
	In Section \ref{sec:wong-zakai} we use these estimates to prove the continuity of the It\^o-Lyons map, thereby finishing the proof of Theorem \ref{thm:wong-zakai}.
	Finally, Section~\ref{sec:LD} presents the proof of the large deviation principle.
	The appendices are devoted to technical results. Appendix \ref{app:combinatorial} proves some combinatorial lemmas which are in particular needed in the proof of Wong-Zakai convergence. Finally, Appendix \ref{app:perturbed} deals with the well-posedness of LLG for an input with non-rough time regularity.
	
	\subsection*{Acknowledgements}
	Both authors warmly thank M.~Hofmanov\'a for her many advices and constant help to run this project.
	AH was partially funded by Deutsche Forschungsgemeinschaft (DFG) through grant CRC 910 ``Control of self-organizing nonlinear systems: Theoretical methods and concepts of application'', Project (A10) ``Control of stochastic mean-field equations with applications to brain networks.''
	EG gratefully acknowledges the financial support by the German Science Foundation DFG via the Collaborative Research Center SFB1283.

	\subsection{Frequently used notations}
	We denote by $\mathbb{N}$ the space of natural numbers $1,2,\dots$, by $\mathbb{N}_0:=\mathbb{N}\cup\{0\}$.
	Throughout the paper, a finite but arbitrary time horizon $T>0$ is fixed. 
	We denote by $\mathbb{S}^2$ the unit sphere in $\R ^3$ and by $\T=\R/\mathbb Z$ the unit torus.
	By $a\cdot b$, we denote the inner product in $\R^n$ for any $n\in\mathbb N$ and $a,b\in \R^n$, and by $|\cdot|$ the norm inherited from it (we will not distinguish between the different dimensions, it will be clear from the context). If $n=3$ we recall the definition
	$a\times b:=(a_2b_3-a_3b_2,a_3b_1-a_1b_3,a_1b_2-a_2b_1)$. The space of linear operators from a Banach space $E$ to itself is denoted by $\mathcal L(E)$. When $E=\R^n,$ we denote by $\mathbf 1_{n\times n}$ the identity map.

	\paragraph{Paths and controls}
	Let $I$ be a subinterval of $[0,T]$. The shorthand `$\forall s \leq t \in I$' will be used instead of `$\forall (s,t)\in I^2$ such that $s \leq t.$'
	Given a one-index map $g$ defined on $I$, we denote by $\delta g_{s,t}:=g_t-g_s$ for $s \leq t \in I$.
	If $G$ is defined on $I^2$, we denote by $\delta G_{s,u,t}:=G_{s,t}-G_{s,u}-G_{u,t}$ for $s \leq u\leq t\in I$. 
	We call \textit{increment} any two-index map which is given by $\delta g_{s,t}$ for some $g=g_t.$
	As is easy to observe, increments are exactly those 2-parameter elements $G_{s,t}$ for which $\delta G_{s,u,t}\equiv 0$.
	
	We say that a continuous map $\omega\colon \{(s,t)\in I :s\leq t\}\rightarrow [0,+\infty)$ is a \textit{control on $I$} if $\omega$ is continuous, $\omega(t,t)=0$ for any $t\in I$ and if it is superadditive, i.e.\ for all $s \leq u \leq t$
	\begin{align*}
	\omega(s,u)+\omega(u,t)\leq \omega (s,t)\, .
	\end{align*}
	Given a control $\omega$ on $I:=[a,b]$, we also denote $\omega(I):=\omega(a,b)$.
	
	Let $(E,\|\cdot\|_E)$ be a Banach space and $p>0$, we denote by $\mathcal{V}^p_{2}(I;E)$ the set of two-index maps $G\colon\{(s,t)\in I\times I :s\leq t\}\rightarrow E$ that are continuous in both the components such that $G_{t,t}=0$ for all $t\in I$ and there exists a control $\omega$ on $I$ so that
	\begin{equation} \label{condition_control_1}
	\|G_{s,t}\|_E\leq \omega(s,t)^{1/p}
	\end{equation}
	for all $s\leq t \in I$. 
	The space $\mathcal{V}^p_{2}(I;E)$ is equivalently defined as the (usual) space of continuous paths of finite $p$-variation, namely $G_{s,t}$ belongs to $\mathcal{V}^p_{2}(I;E)$ if and only if 
\[
	\|G\|_{\mathcal{V}_{2}^{p}(I;E)}:=\sup_{\pi}\Big (\sum\nolimits_{[u,v]\in\pi}\|G_{u,v}\|_E^p\Big)^\frac{1}{p}<+\infty\, ,
\]
	where the supremum is taken over the set of partitions $\pi=\{[t_0,t_1],\dots [t_{n-1},t_{n}]\}$ of $I$.
	Moreover, the above semi-norm $\|\cdot\|_{\mathcal{V}^p_{2}(I;E)}$ is in fact equal to the infimum of $\omega(I)^p$ over the set of controls $\omega$ such that \eqref{condition_control_1} holds (see \cite{hocquet2017energy} or \cite[Paragraph 8.1.1]{FrizVictoire}).	
	In a similar fashion we define the space $\mathcal{V}^p(I;E)$ of all continuous paths $g\colon I\to E$ such that $\delta g\in\mathcal{V}^p_{2}(I;E)$. It is endowed with the norm $\|g\|_{\V^p(I;E)}=\sup_{t\in I}\|g_t\|_{E}+\|\delta g\|_{\mathcal{V}_{2}^{p}(I;E)}$.

	We will sometimes need to work with local versions of the previous spaces. For that purpose, we define $\mathcal{V}_{2,\mathrm{loc}}^{p}(I;E)$ as the space of two-index maps $G\colon\{(s,t)\in I\times I :s\leq t\}\rightarrow E$ such that there exists a finite covering $(I_k)_{k\in K}$ of $I$, $K\subset \mathbb{N}$, so that $G\in \mathcal V_{2}^{p}(I_k;E)$ for all $k\in K$. 
	We define the linear space 
	\begin{align*}
	\mathcal{V}^{1-}_{2,\mathrm{loc}}(I;E):=\bigcup_{0<p<1} \mathcal{V}_{2,\mathrm{loc}}^{p}(I;E)\, .
	\end{align*}
	We denote by $\mathcal{V}^p(E)=\mathcal{V}^p([0,T];E),$ and for $I=[s,t]\subset[0,T]$ we use the abbreviations $\V^p(s,t;E)=\V^p([s,t];E)$. Likewise \( \mathcal{V}_2^p(E)=\mathcal{V}_2^p([0,T];E) \) and \( \mathcal{V}_2^p(s,t;E)=\mathcal{V}_2^p([s,t];E) \). We denote by $C([0,T];E)$ the space of continuous functions defined on $[0,T]$ with values in $E$.
	\paragraph{Sobolev spaces}
	For $n\ge1 $ we consider the usual Lebesgue spaces $L^p:=L^p(\T ,\R^n)$, for $p\in[1,+\infty]$ endowed with the norm $\|\cdot\|_{L^p}$ and the classical Sobolev spaces $W^{k,q}:=W^{k,q}(\T ,\R ^n)$ for integer $q\in [1,+\infty]$ and $k\in\mathbb{N}$ endowed with the norm $\|\cdot\|_{W^{k,q}}$.
	We also denote by $H^k:=W^{k,2}(\T ,\R ^n)$. 
	Again, our abbreviations do not distinguish between different (finite) dimensions for the target space.
	Sometimes it will be necessary to consider functions taking values in the unit sphere $\mathbb{S}^2\subset\R^3$: for that purpose, we adopt the notation
	\begin{equation}
	H^k(\T ;\mathbb{S}^2):=H^k(\T;\R^3)\cap \{f:\T \rightarrow \R^3 \, \textrm{s.t.}\, |f(x)|=1\text{ a.e.}\}\, ,
	\end{equation}
	for $k\geq 1$. 
	Finally, we will denote by $L^p(W^{k,q}):=L^p([0,T];W^{k,q}(\T ;\R ^n))$ and by $L^p(s,t;W^{k,q}):=L^p([s,t];W^{k,q}(\T ;\R ^n))$.

	\section{Settings and main results}
	\label{sec:main}
	
	We  first need to specify the type of data needed as inputs to our main results, which are not exactly rough paths nor controlled paths, at least in the usual sense \cite{lyons1998differential,gubinelli2004controlling}.
	In the context of a system with linear multiplicative noise such as \eqref{LLG}, it is indeed more convenient to think of the rough input as being the multiplication operator $(\times \dot w_t(x))$ instead of the noise $\dot w_t(x)$ itself.	
	This leads to a slighlty different formalism as what can be encountered in the usual rough path literature (see \cite{lyons2002system,FrizVictoire,FrizVictoire}).
	
	\subsection{Rough drivers}
	\label{sec:RD}

	For the reasons alluded above, we build our setting upon the formalism of (unbounded) rough driver.
	These objects were introduced in \cite{bailleul2017unbounded} and further investigated in \cite{deya2016priori,hocquet2017energy,hocquet2018ito,hofmanova2019navier,hocquet2018quasilinear}. 
	
	\begin{definition}[Rough driver]
		\label{defi-RD}
		Given $n\in\mathbb{N}$ and $p\in[2,3)$
		we say that $\G$ is a $n$-dimensional \textit{rough driver} if it consists in a pair
		\begin{equation}\label{p-var-rp}
		\G :=(G,\GG ) \in \mathcal V^p_2 \left ([0,T];L^2(\T ;\mathcal L(\R^n))\right ) \times \mathcal V^{p/2}_2 \left ([0,T]; L^2(\T ;\mathcal L(\R^n))\right ) 
		\end{equation}
		such that $\delta G_{s,u,t}=0$ (i.e.\ the first component is an increment) while the following Chen type relation is fulfilled
		 in the sense of composition of linear maps in $\R^n$:
		\begin{equation}\label{chen-rela}
		\delta \GG _{s,u,t}(x)=G_{u,t}(x)G_{s,u}(x)\,,
		\quad \text{for all }s<u<t\in [0,T],
		\quad \text{Lebesgue-a.e.\ }x\in\R^d\,.
		\end{equation}
		In addition, we say that
		\begin{enumerate}[label=(\alph*)]
		\item
		$\G $ \textit{geometric} if it can be obtained as the limit, for the inherited $(p,\frac p2)$-variation topology, of a sequence of smooth rough drivers $\G ^\epsilon=(G^{\epsilon},\GG ^{\epsilon}),$ $\epsilon>0$, explicitly defined for $s\le t \in [0,T]$ as
		\begin{equation}\label{relazioni_path}
		G^{\epsilon}_{s,t} := \Gamma^{\epsilon}_{t}-\Gamma^\epsilon_s \ , 
		\qquad \GG ^{\epsilon}_{s,t}:=\int_{s}^td\Gamma^\epsilon_r (\Gamma^\epsilon_{r}-\Gamma^\epsilon_s)\, ,
		\end{equation}
		for some smooth path $\Gamma^\epsilon\colon[0,T] \to L^2(\T;\mathcal L(\R^n))$. 
		\item \label{itm:anti}
		$\G $ is \textit{anti-symmetric} if it is geometric and such that the approximating sequence $(\Gamma_t^\epsilon(x))$ in \eqref{relazioni_path} can be taken with values
		in the space $\mathcal L_a(\R^n)\subset\mathcal L(\R^n)$ of linear maps $\Gamma\colon \R^n\to \R^n$ such that $\Gamma v\cdot w = -v\cdot(\Gamma w)$ $\forall v,w\in\R^n$
		(this implies in particular that the first component $G_{s,t}(x)$ takes values in \( \mathcal L_a(\R^n) \)). 
		\end{enumerate}
		\end{definition}
	
	Since the solution to \eqref{LLG} has to lie on $\mathbb{S}^2$, the rough term in \eqref{LLG} must be (formally) orthogonal to $u_t(x)$, almost everywhere. This certainly imposes that the first component of $\G$ be \( \mathcal L_a (\R^n)\)-valued. However, this property alone does not guarantee that a solution will stay in the sphere for all times (as can be seen by formally applying It\^o's formula to $|u_t(x)|_{\R^3}^2$ for $x$ fixed). 
	For that purpose, we really need $\G$ to be anti-symmetric in the sense of \ref{itm:anti}, which imposes in partivular geometricity. In the stochastic case, this essentially means that our rough enhancement $\G$ of $G$ is consistent with Stratonovich-type integration.

	For $k\in \mathbb{N}$, we denote by $\RD ^p_a(H^k)$ the space of such $3$-dimensional, geometric and anti-symmetric rough drivers
	\begin{align*}
	\G  \in \mathcal V^p_2 \left (0,T;H^k(\T;\mathcal L_a(\R^3))\right ) \times \mathcal V^{p/2}_2 \left (0,T; H^k(\T;\mathcal L(\R^3))\right ) \, ,
	\end{align*}
	whose coordinates belong to the $k$-th order Sobolev space $H^k.$
	We introduce also the following controls for $s\leq t$ and $0\leq \gamma\leq k$
	\begin{align*}
	&\omega_{G,H^\gamma}(s,t):=\|G\|^p_{V^p_2 \left (s,t;H^\gamma(\T;\mathcal L_a(\R^3))\right ) }, \quad \omega_{\GG ,H^\gamma}(s,t):=\|\GG \|^{p/2}_{V^{p/2}_2 \left (s,t;H^\gamma(\T;\mathcal L(\R^3))\right ) }\, ,\\
	&\omega_{\G ,H^\gamma}(s,t):= \|G\|^p_{V^p_2 \left (s,t;H^\gamma(\T;\mathcal L_a(\R^3))\right ) }+\|\GG \|^{p/2}_{V^{p/2}_2 \left (s,t;H^\gamma(\T;\mathcal L(\R^3))\right ) }+\|G\|^{2p}_{V^p_2 \left (s,t;H^\gamma(\T;\mathcal L_a(\R^3))\right ) }\, .
	\end{align*}

	\begin{notation}
		\label{nota:F}
	For $n=3$, we let $\mathcal F\colon \R^3\to \mathcal L_a(\R^3)$ be the map which to each vector $\xi\in \R^3$ assigns the element $(\cdot)\times \xi\in \mathcal L_a(\R^3)$. 
	In the canonical basis of $\R^3$, this means that: 
	\[
	\mathcal F(\xi) 
=	\begin{pmatrix}
		0 & \xi^{3} & -\xi^{2}\\
		-\xi^{3} & 0 & \xi^{1}\\
		\xi^{2} & -\xi^{1} & 0
	\end{pmatrix}\,.
	\]
	\end{notation}
	
	A basic example of a geometric continuous rough path is given by the Stratonovich enhancement of a (three-dimensional)	Brownian motion.
	\begin{example}[Spatially constant rough drivers]
		\label{exa:B}
		Consider a fractional Brownian motion $\beta\colon \Omega\times [0,T]\to\R^3$ with Hurst parameter $H>\frac13$.
		According to \cite{FrizVictoire} (see also \cite{FrizHairer}) there exists a random 
		geometric rough path  $\rp=(\delta\beta_{s,t},\rpp_{s,t})$, which is built over $\beta$ in 
		the sense that the first entry coincides, $\mathbb P$-almost surely, with the increment 
		$\delta\beta^{i}_{s,t}=\beta^i_t-\beta^i_s,\forall s,t\in [0,T],\,i=1,2,3$. 
		This also means that $\rpp_{s,t}$ takes values in $\R^3\otimes \R^3\simeq \R^{3\times3}$ and is subject to the usual Chen's relation
		\begin{equation}\label{chen-rela_classic}
		\delta \rpp^{i,j}_{s,u,t} = \delta\beta^{i}_{s,u}\delta\beta_{u,t}^{j}\,,\quad 1\le i,j\le 3\quad (\mathbb P\text{-a.s.})\,.
		\end{equation}
		In fact, in the case when $H=\frac12$, one has that $\rpp$ is the Stratonovich iterated integral of $\beta$ with itself, i.e.\ $\rpp^{i,j}_{s,t}=\int _s^t(\beta^i_r-\beta^i_s)\circ \beta^j_r$ .
		
	Using the notation \ref{nota:F}, for all $s\leq t\in [0,T]$ we define $G_{s,t}\in \mathcal L_a(\R^3)$ as the vector product against $\delta \beta_{s,t}$, namely
	\setlength{\arraycolsep}{1.5pt}
	\begin{equation}
		\label{G_first}
	G_{s,t}:= \mathcal F(\delta\beta_{s,t}) 
	=
	\begin{pmatrix}
	0 & \delta\beta^{3}_{s,t} & -\delta\beta^{2}_{s,t}\\
	-\delta\beta^{3}_{s,t} & 0 & \delta\beta^{1}_{s,t}\\
	\delta\beta^{2}_{s,t} & -\delta\beta^{1}_{s,t} & 0
	\end{pmatrix}
	\in \V_2^{p}(0,T;\mathcal L_a(\R^3))
	\end{equation}
	which indeed has finite $p$-variation as soon as $p\ge3 >\frac1H$.
	Furthermore, we define $\GG $ as follows:
	for each $s\leq t\in[0,T],$
	\begin{equation}
	\label{murky_formula}
	\GG_{s,t}:=
	\begin{pmatrix}
	-\rpp^{3,3}_{s,t} -\rpp^{2,2}_{s,t} &
	\rpp^{1,2}_{s,t} &
	\rpp^{1,3}_{s,t}\\
	\rpp^{2,1}_{s,t} &
	-\rpp^{3,3}_{s,t} -\rpp^{1,1}_{s,t} &
	\rpp^{2,3}_{s,t}\\
	\rpp^{3,1}_{s,t} &
	\rpp^{3,2}_{s,t} &
	-\rpp^{2,2}_{s,t} -\rpp^{1,1}_{s,t}\\
	\end{pmatrix}
	\,.
	\end{equation}
	With this definition, we find using Chen's relation \eqref{chen-rela_classic} that for each $s\leq u\leq t\in[0,T]$ 
	\begin{equation}
	\label{pre_chen}
	\begin{aligned}
	\delta \GG _{s,u,t}
	&:=
	\begin{pmatrix}
	-\delta\beta^{3}_{s,u}\delta\beta^{3}_{u,t} -\delta\beta^{2}_{s,u}\delta\beta^{2}_{u,t} &
	\delta\beta^{1}_{s,u}\delta\beta^{2}_{u,t} &
	\delta\beta^{1}_{s,u}\delta\beta^{3}_{u,t}
	\\
	\delta\beta^{2}_{s,u}\delta\beta^{1}_{u,t} &
	-\delta\beta^{3}_{s,u}\delta\beta^{3}_{u,t} -\delta\beta^{1}_{s,u}\delta\beta^{1}_{u,t} &
	\delta\beta^{2}_{s,u}\delta\beta^{3}_{u,t}\\
	\delta\beta^{3}_{s,u}\delta\beta^{1}_{u,t} &
	\delta\beta^{3}_{s,u}\delta\beta^{2}_{u,t} &
	-\delta\beta^{2}_{s,u}\delta\beta^{2}_{u,t} -\delta\beta^{1}_{s,u}\delta\beta^{1}_{u,t}\\
	\end{pmatrix}
	\\
	&:=
	\begin{pmatrix}
	0 & \delta\beta^{3}_{u,t} & -\delta\beta^{2}_{u,t}\\
	-\delta\beta^{3}_{u,t} & 0 & \delta\beta^{1}_{u,t}\\
	\delta\beta^{2}_{u,t} & -\delta\beta^{1}_{u,t} & 0
	\end{pmatrix}
	\begin{pmatrix}
	0 & \delta\beta^{3}_{s,u} & -\delta\beta^{2}_{s,u}\\
	-\delta\beta^{3}_{s,u} & 0 & \delta\beta^{1}_{s,u}\\
	\delta\beta^{2}_{s,u} & -\delta\beta^{1}_{s,u} & 0
	\end{pmatrix}
	.
	\end{aligned}
	\end{equation}
	It follows that \eqref{chen-rela} holds and so $\G :=(G,\GG )$ is a 3-dimensional rough driver.
	That $\G$ is geometric follows by the geometricity of the rough path $\rp$. It is also clear 
	that any approximating sequence $\rp^n\to \rp$ for the rough path metric (inherited from 
	$\V^p(0,T;\R^3)\times \V^{p/2}(0,T;\R^{3\times3})$) yields a corresponding anti-symmetric 
	$\G^n\to \G$ for the rough driver metric \eqref{p-var-rp}. Hence the anti-symmetry of $\G$.
	\end{example}

	The formula \eqref{murky_formula} might seem a bit `murky' at first sight. 
	Recalling that we identified $\R^{3\times 3}$ with $\R^3\otimes \R^3$, \eqref{murky_formula} should in fact be understood as the statement that \begin{equation}
		\label{GG_tensor}
		\GG_{s,t}^{i,j}:=(\mathcal F\otimes \mathcal F)[\rpp_{s,t}]^{i,j}=\sum_{k,l}\mathcal F^{i,k}\mathcal F^{j,l}\rpp^{k,l}_{s,t}\,,\quad \forall s\leq t\in [0,T],\quad i,j=1,2,3\, .
	\end{equation}
	 When $H=\frac12$, this guarantees that $\GG_{s,t}$ is indeed the iterated integral of $G$ against itself, namely it holds $\mathbb P$-a.s.\
	\begin{equation}
	\label{second_level_F}
	\GG_{s,t}:=\int_s^t d G \circ (G _r-G_s)
	\end{equation}
	(Stratonovich sense). The relation \eqref{second_level_F} for the second level of $\G$ permits to identify pathwise solutions (in the sense of Definition \ref{def:solution} below) with actual solutions in the usual Stratonovich sense. This fact will be made clear in Proposition \ref{pro:consistency} below.
	
	\begin{notation}
	\label{nota:FF}
	We denote by $\mathcal F(\rp)$ the rough driver defined by the formulas \eqref{G_first}, \eqref{murky_formula}, namely $\mathcal F(\rp)= (\mathcal F(\delta\beta), \mathcal F\otimes \mathcal F(\rpp))$.
	\end{notation}
	
	A way to add spatial dependency to the previous example is to let, for some $g\in H^k(\T;\R)$,
	\begin{equation}
		\label{simple_RD}
		(G,\GG)(x):= \left(g(x)\mathcal F(\delta\beta), g(x)^2 \mathcal F\otimes \mathcal F(\rpp)\right).
	\end{equation}
Equivalently $\G(x):=\Lambda_{g(x)}\mathcal F(\rp)$ where $\Lambda$ is the natural \textit{dilation operator} 
\begin{equation}
	\label{nota:dilation}
	\Lambda_\lambda \G=(\lambda G,\lambda^2\GG)\quad \text{ for } \lambda\ge0\, .
\end{equation}
With this definition the relation \eqref{pre_chen} remains trivially true and it is easy to check that the resulting rough driver is anti-symmetric. 
	When $H=\frac12$, this construction generalises as follows.

	\begin{example}[$Q$-Wiener process]
		\label{exa:WP}
		Fix a stochastic basis $(\Omega,\mathcal F,\mathcal F_t,\mathbb P;(\beta^j))$ where $(\beta^j)_{j\ge0}$ is and i.i.d.\ family of real-valued Brownian motions.
	For any non-negative, trace-class operator $Q\in \mathcal L\left(H^1(\T ;\R^3)\right)$, introduce the  $Q$-Wiener process on $L^2(\T;\R^3)$:
		\[
		w_t(\cdot)=\sum_{j\ge 0}\sqrt Q g_j(\cdot )\beta^j_t\,,
		\]
		where we are assuming that $(g_j(x))_j$ is an orthonormal basis of $L^2(\T ;\R^3)$ (see \cite[p.\ 96]{DPZ} for further details).
		
		As in Example \eqref{exa:B}, we define for a.e.\ $x\in \T$:
		\setlength{\arraycolsep}{1.5pt}
		\[
		G_{s,t}(x):= 
		\mathcal F(\delta w_{s,t}(x))
		=\begin{pmatrix}
		0 & \delta w_{s,t}^3 & -\delta w_{s,t}^2\\
		-\delta w_{s,t}^3 & 0 & \delta w_{s,t}^1\\
		\delta w_{s,t}^2 & -\delta w_{s,t}^1 & 0
		\end{pmatrix}\!\!(x)
		\in \V^{p}(0,T;L^2(\T ;\mathcal L_a(\R^3)))\, .
		\]
		
		As is well-known and easily verified, $H^1(\T ;\R)$ is an algebra under pointwise multiplication and moreover $\|fg\|_{H^1}\lesssim \|f\|_{H^1}\|g\|_{H^{1}}$. In addition $H^1(\T)\hookrightarrow C^0(\T)$ with compact embedding and in particular for each $x$ the evaluation map $g\to g(x)$ is continuous from $H^1\to \R$.
		It follows that one can define the following Stratonovich integrals on a common set of full probability measure
		\[
		\begin{aligned}
		\GG _{s,t}(x)
		&=
		\int\limits_s^t
		\begin{pmatrix}
		-  \delta w_{s,r}^3\circ d w_{r}^3 -  \delta w_{s,r}^2\circ d w_{r}^2 &
		  \delta w_{s,r}^1\circ d w_{r}^2 &
		  \delta w_{s,r}^1\circ d w_{r}^3\\
		  \delta w_{s,r}^2\circ d w_{r}^1 &
		-  \delta w_{s,r}^3\circ d w_{r}^3 -  \delta w_{s,r}^1 \circ d w_{r}^1&
		  \delta w_{s,r}^2\circ d w_{r}^3\\
		  \delta w_{s,r}^3\circ d w_{r}^1 &
		  \delta w_{s,r}^3\circ d w_{r}^2 &
		-  \delta w_{s,r}^2\circ d w_{r}^2 -  \delta w_{s,r}^1\circ d w_{r}^1\\
		\end{pmatrix}\!\!(x)
		\end{aligned}
		\]
		for each $0\le s\le t\le T$ and a.e.\ $x\in \T.$
		With this definition, it is plain to verify that $\G$ is a random rough driver (adapt the computations of Example \ref{exa:B}).
	\end{example}
	
	One might wonder how the second component of an anti-symmetric rough driver is affected by the existence of an anti-symmetric approximating sequence $\Gamma^\epsilon$ as in \eqref{relazioni_path}. The next remark addresses this particular question.
	
		\begin{remark}
		\label{rem:levy}
		The second component $\GG_{s,t}$ of an anti-symmetric rough driver is to be thought of the (generally ill-defined) iterated integral $\int_{s<r_1<r<t} d\Gamma_r d\Gamma_{r_1}$, where $\Gamma_t=G_{0,t}$ is a path taking values in $\mathcal L_a(\R^n).$ As such, it is not anti-symmetric, nor symmetric in general.
		Assuming that $\Gamma$ has finite variation,
		an easy integration by parts argument over $r\in[s,t]$, using the antisymmetry of $\Gamma_t$, implies
		(here $\star$ denotes transposition)
		\[
		\GG^{\star}_{s,t}=\left (\int_s^t d\Gamma_{r}\delta \Gamma_{s,r}\right )^{\star}=
		\int_s^t \delta \Gamma_{s,r}d\Gamma_{r} = \int _s ^td\Gamma_r\delta \Gamma_{r,t}\,.
		\]
		Summing, we obtain that the symmetric part of $\GG$ is determined by the first level $G$ through
		\[
		\frac12\left (\GG _{s,t}+ \GG^\star_{s,t}\right ) = \frac12(G_{s,t})^2\,.
		\]
		Since geometric rough drivers are obtained as limits of such finite-variation lifts, this entails a similar identity for any such $\G\in \mathcal {RD}_a^{p}(H^\sigma)$.
		Namely, there exists an anti-symmetric family $\LL\in\V^{\frac p2}_{2} (H^{1}(\T;\mathcal L_a(\R^n)))$ such that
		\begin{equation}\label{levy}
		\GG_{s,t}(x) = \frac12(G_{s,t}(x))^2 +\LL_{s,t}(x)\,.
		\end{equation}
		
		In the particular case dealt in Example \ref{exa:B}, one sees that $\LL_{s,t}$ is in fact the L\'evy area of $\B$, formally:
		\begin{equation}
		\label{bracket}
		\begin{aligned}
		\LL_{s,t} =
		\iint\limits_{0<r_1<r<t}
		\frac12
		\begin{pmatrix}
		0 &
		d \beta_{r}^2 d \beta_{r_1}^1 -  d \beta_{r}^1 d \beta_{r_1}^2 &
		d \beta_{r}^3 d \beta_{r_1}^1- d \beta_{r}^1 d \beta_{r_1}^3
		\\
		d \beta_{r}^1 d \beta_{r_1}^2 - d \beta_{r}^2 d \beta_{r_1}^1 &
		0 &
		d \beta_{r}^3 d \beta_{r_1}^2- d \beta_{r}^3 d \beta_{r_1}^2
		\\
		d \beta_{r}^1 d \beta_{r_1}^3 - d \beta_{r}^3 d \beta_{r_1}^1&
		d \beta_{r}^2 d \beta_{r_1}^3 - d \beta_{r}^3 d \beta_{r_1}^2&
		0
		\end{pmatrix}.
		\end{aligned}
		\end{equation}
	\end{remark}
	
	\subsection{Notion of solution and main results }
	
	We want to formulate the original problem \eqref{LLG} by giving a meaning to the integral against the noise $ \times d w$ in our rough driver setting. As already underlined in the introduction, two main advantages of this approach are the fact that it allows for a deterministic treatment and, more importantly perhaps, pathwise continuity results are available. The following classical result is at the core of our approach.
	\begin{lemma}[Sewing lemma \cite{gubinelli2004controlling}]
		\label{lemma_sewing}
		Fix an interval $I\subset[0,T]$, a Banach space $(E,|\cdot |)$ and a parameter
		$\zeta > 1$. Consider a map $H \colon I^2 \rightarrow E$ and $C>0$ such that 
		\begin{equation}
		\label{a_gamma}
		\left|\delta H_{s,u, t}\right|\leq C\omega (s,t)^{\zeta }\,
		, \quad s \leq u \leq t \in I,
		\end{equation}
		for some control function $\omega ,$ and denote by $[\delta H]_{\zeta,\omega }$ the smallest possible constant $C$ in the above bound. 
		There exists a unique pair $\mathcal I\colon I \rightarrow E$ and $\mathcal I^{\natural}\colon I^2 \rightarrow E$ such that $\mathcal I_0=0$, $\delta \mathcal I^{\natural}_{s,u,t}=-\delta H_{s,u,t}$, \( \forall s\le u\le t\in [0,T] \)  and
\begin{equation}
	\label{integral_map}
\mathcal I_{t}-\mathcal I_s = H_{s,t} + \mathcal I_{s,t}^{\natural}
\end{equation}
where for $s\le t \in [0,T]$,
\begin{equation}\label{remainder_est_SL}
|\mathcal I_{s,t}^{\natural}| \leq C_\zeta [\delta H]_{\zeta,\omega } \omega (s,t)^{\zeta}\,,
\end{equation}
for some universal constant $C_\zeta>0$.
\end{lemma}

In Example \ref{exa:WP} we defined a random rough driver  $\G (\omega)=(G,\GG )(\omega)$ and our 
	purpose is to look at the solvability of  \eqref{LLG} 
	in the sense of a suitable Taylor-Euler expansion driven by $\G (\omega)$ for $\omega$ from the set of full probability where $\GG $ was constructed. In the sequel, unless otherwise stated,  we keep such an $\omega$ fixed and for notational simplicity we do not write it explicitly. 
	Next, we iterate the equation into itself to obtain, at least formally,
	\begin{align*}
	\delta u_{s,t}
	-\int_s^t(\dd^2 u _r+ u_r|\partial_x u_r|^2 + u_r\times\dd^2 u_r ) d r
	&=\int_s^tu_r\times d w_r\\
	&=u_s\times \delta w_{s,t} + \left (\int _s^t(\cdot \times \delta w_{s,r})\times  d w_r\right )u_s + o(t-s)\,.
	\end{align*}
	We have already discussed in Examples \ref{exa:B}-\ref{exa:WP} how to interpret the stochastic integral, and therefore we look at the equation
	\begin{align}
	\label{rLLG}
	\tag{rLLG}
	\delta u_{s,t}
	-\int_s^t(\partial_x^2 u_r  + u_r|\partial_x u_r|^2 + u_r\times\partial_x^2 u_r ) d r=G_{s,t}u_s + \mathbb G_{s,t}u_s +u^{\natural}_{s,t}\,,
	\end{align}
	where the remainder  $u^{\natural}_{s,t}$ is defined through \eqref{rLLG} and it is required to be small in a certain  sense. 
	Specifically, one needs this term to satisfy an estimate of the form \eqref{remainder_est_SL} in $E:=L^2(\T;\R^3)$, and for this purpose it is enough to check that \(H_{s,t}=G_{s,t}u_s + \mathbb G_{s,t}u_s\) satisfies the assumptions of Lemma \ref{lemma_sewing}.
	As we shall see later on, the following definition of a pathwise solution will be enough to guarantee these properties.	
	\begin{definition}
		\label{def:solution}
		Let $\omega\rightarrow \mathbf{G}(\omega)\equiv(G(\omega),\mathbb{G}(\omega))$ be a pathwise random rough driver with values in \( \RD_a^p(H^2) \) for some \( p\in [2,3) \).
		We say that a stochastic process $u\colon \Omega\times [0,T]\to L^2(\T;\R^3)$  is a \textit{pathwise solution} of the \eqref{rLLG} if it is
		such that
			\begin{enumerate}[label=(\roman*)]
			\item \label{sol_i}
			$u_t(\omega,x)\in\mathbb{S}^2$ for a.e.\ $(\omega,t,x)\in\Omega\times[0,T]\times\T$;
			\item \label{sol_ii} with probability one
			$u\in L^\infty (0,T;H^1)\cap L^2(0,T;H^2)$; 
			\item \label{sol_iii}
			there exists \( q<1 \) and a random variable $u^{\natural}\in L^0(\Omega;\V^{q}_{2,\mathrm{loc}}(0,T;L^2))$ such that
			\begin{equation}
			\label{rLLG_def}
			\delta u_{s,t}-\int_s^t(\dd^2u_r +u_r|\partial_x u_r|^2 +u_r\times\partial_x^2 u_r)  d r
			=G_{s,t}u_s + \mathbb G_{s,t}u_s + u^\natural_{s,t}\,,
			\end{equation}
			as an equality in $L^2(\T ;\R^3)$, for every $s\leq t\in[0,T] $ and $\mathbb P$-a.s.
		\end{enumerate}
	
	Additionally, we say that \textit{$u$ starts at \( u^0\in H^1(\mathbb{T},\mathbb{S}^2) \)} if
	\begin{enumerate}[label=(\roman*)]\setcounter{enumi}{3}
	\item\label{itm:ci} with probability one, \( \lim _{t\downarrow0}u_t \) exists in \( L^2 \)-strong and equals \( u^0 \) .
	\end{enumerate}
	\end{definition}
	
	\begin{remark}
	It will be seen thanks to the a priori estimates of Section \ref{sec:apriori} and a compactness argument that pathwise solutions are in fact supported in \( \V^p([0,T];L^2)\subset C([0,T];L^2) \). Consequently the first part in condition \ref{itm:ci} above is always satisfied.
	\end{remark}
	
	\begin{remark}
	The fact that we require no $\mathcal F_t$-measurability for the unknown $u_t(\omega,x)$ is not a typo.
	Indeed, our treatment of \eqref{LLG} relies on rough paths techniques and hence integration makes sense independently of $\omega\in\Omega$, provided $\G(\omega)\in \mathcal {RD}^{p}_a(H^{1})$.
	Nevertheless, it will be seen that the solutions obtained from the rough driver of Example \ref{exa:WP} (resp.\ Example \ref{exa:B} with $H=\frac12$) are indeed $\mathcal F_t$-adapted, and in both cases these coincide with the usual Stratonovich solutions.
	A counterpart of this fact should be true as well for the fractional Brownian motion case (i.e.\ example \ref{exa:B} with $H\in(\frac13,\frac12)$), but this would require first to make more precise what we mean by a ``solution'' to \eqref{LLG} (for instance via Malliavin calculus techniques). Since this question goes much beyond our main motivation, we chose to leave it for future investigations. 
	\end{remark}

	Our first main result concerns existence and uniqueness  and we prove this statement in Section \ref{sec:apriori}.
	
	\begin{theorem}
		\label{thm:existence}
		Let  $u^0\in H^1(\T ;\mathbb{S}^2)$ and consider a random variable $\G\colon \Omega\to \RD _a^p(H^2)$.
		There exists a  unique pathwise solution $u$ to \eqref{rLLG} in the sense of Definition \ref{def:solution}.
		Moreover, there is a random variable $\ell\colon \Omega\to (0,T]$ such that the remainder $u^{\natural}$ is estimated above as
		\begin{equation}
			\label{remainder_intro}
		\|u^{\natural}_{s,t}\|_{L^2} \lesssim
		\omega_{\G,H^2}(s,t)^{\frac3p}+\omega_{\G,H^2}(s,t)^{\frac1p}\int_s^t\|u_r\times \dd^2u_r\|_{L^2}dr \,,
		\quad \forall 0\leq t-s\leq \ell\,.
		\end{equation}
	\end{theorem}

\begin{remark}[Dynamical description]
	\label{rem:global}
	Let \( u,\G\) be as in Theorem \ref{thm:existence} and for simplicity, assume that the control in the right hand side of \eqref{remainder_intro} is linear, namely \( |u^{\natural}_{s,t}|_{L^2}\lesssim (t-s)^{3/p}\).
	Note that \eqref{remainder_intro} and the uniqueness part of the Sewing Lemma
	 imply that $\mathcal I^{\natural }_{s,t}=u^{\natural}_{s,t}$ in \eqref{integral_map}.
	Letting \( H_{s,t}=(G_{s,t}+\GG_{s,t})u_s \), the identity \( -\delta H_{s,\theta,t} \equiv\delta u^{\natural}_{s,\theta,t}\) that results from \eqref{rLLG_def} ensures that the quantity \( \int_s^t d\G u_r:=H_{s,t}+u^{\natural}_{s,t}\) satisfies the additivity property \( \int_{s}^\theta d\G_r u_r + \int _\theta ^t d\G_r u_r = \int_s^t d\G_r u_r ,\) for each \( s\le \theta\le t\in [0,T] .\)
	
	Consequently, the estimate \eqref{remainder_intro} allows to make sense of \eqref{rLLG} as an \textit{integral} equation in the following way.
	If \( \pi=\{[0=t_0<t_1],\dots, [t_k,t_{l+1}=t]\} \) is any finite partition of \( [0,t] \), summing \eqref{rLLG_def} over \( \pi \) shows that
	\begin{equation}
		\label{integral_description}
		\begin{aligned}
		u_t-u_0- \int_0^t(\partial_x^2 u_r  + u_r|\partial_x u_r|^2 + u_r\times\partial_x^2 u_r ) d r
		&=\int_0^td\G_r u_r.
		\end{aligned}
	\end{equation}
For \(|\pi|:=\max(t_{i+1}-t_i) \le \ell \), we have that
\[
\begin{aligned}
\int_0^td\G_r u_r
&\equiv\sum_{[t_i,t_{i+1}]\in \pi} (H_{t_i,t_{i+1}} + u^{\natural}_{t_i,t_{i+1}})
=\sum_{[t_i,t_{i+1}]\in \pi} (H_{t_i,t_{i+1}}) + O(|\pi|^{3/p-1})
\end{aligned}
\]
showing that \( \int_0^td\G_r u_r \) can be defined as the limit as \( |\pi|\to 0\) of the compensated Riemann sums \(\sum_{\pi}H_{t_i,t_{i+1}} \), \( \mathbb P \)-a.s.\ in \( L^2(\T) \).
A rigorous notion of rough integral requires however to introduce the corresponding space of \( G \)\textit{-controlled paths}, which is omitted here for simplicity of the presentation. 
In the context of unbounded rough drivers, we point out that related questions were addressed in \cite[Sec.~3]{hocquet2018ito}.
\end{remark}

	Similar to what is encountered in (S)PDE theory, it is reasonable to hope that our solutions are arbitrarily regular (in the sense of a suitable scale), as permitted however by the regularity of the data. 
	In our case, provided the initial condition and the noise have finite higher order Sobolev norms, then the solution lives in a stronger space, as illustrated in the next result.
	\begin{theorem}
		Let $u^0\in H^k(\T ;\mathbb{S}^2)$  and $\G \colon \Omega\to \RD _a^p(H^{k+1})$ be as in Example \ref{exa:WP} for some $k\geq 2$ or as in Example \ref{exa:B} with \( H=\frac12 \) and \( g\in H^{k+1} \).
		The unique pathwise solution \( u \) constructed in the previous theorem belongs to $\mathcal X^k:=L^\infty(H^k)\cap L^2(H^{k+1})\cap\mathcal{V}^p(H^{k-1})$, $\mathbb P$-a.s. 
		
		Moreover, \( u \) is \( \mathcal F_t \)-adapted and defines a semi-martingale in $L^2(\T)$.
		It has finite moments of arbitrary order in \( \mathcal X^k\) and satisfies the Stratonovich equation:
		\[
		\begin{aligned}
		du_t 
		=(\dd^2 u_t +u_t|\dd u_t|^2 + u_t\times\dd^2 u_t)dt + u _t\times \circ dw \quad \text{ on }\,\Omega\times[0,T]\times\T,\qquad u_0=u^0.
		\end{aligned}
		\]
	\end{theorem}

	Second, we address the so called Wong-Zakai convergence. Let $(G^n, \GG ^n)$ be an approximation of  $(G, \GG )$ which converges in the corresponding $p$-variation norm. Let $u^n$ be the solution to \eqref{rLLG} driven by $(G^n, \GG ^n)$  and let $u$ be the solution to \eqref{rLLG} driven by $(G, \GG )$. We investigate, whether and  in which sense $u^n$ converges to $u$. This Wong-Zakai convergence is a consequence of the continuity of the It\^o-Lyons map, that we prove in Section \ref{sec:wong-zakai}.
	
	\begin{theorem}[Continuity of the It\^o-Lyons map]
		\label{thm:wong-zakai}
		Let \( \G\colon \Omega\to \RD_a^p(H^2) \) be a random variable and denote by \(u^0\mapsto u(u^0) \) the pathwise solution of \ref{rLLG} supplied by Theorem \ref{thm:existence}.
		There is a continuous, deterministic map
		\begin{align*}
		\pi:H^1(\T ;\mathbb{S}^2)\times\RD _a^p(H^2)  &\longrightarrow L^\infty(H^1)\cap L^2(H^2)\cap \V^p(L^2)\\
		\end{align*}
	 	such that for any \( \omega \) with \( \G(\omega)\in \RD^p_a(H^2)\) and every \( u^0\in H^1 \),
	 	\[
	 	 u(u^0;\omega)= \pi(u^0,\G(\omega))\,.
	 	\]
	\end{theorem}
	As a corollary, we conclude with a Wong-Zakai convergence result.
	\begin{corollary}[Wong-Zakai convergence]\label{cor:WZ}
		Let $u^0\in H^1(\mathbb{T};\mathbb{S}^2)$ and $(G^n,\GG ^n)_n$ be an approximation of  $(G,\GG )$ with respect to the $\RD ^p_a(H^2)$-norm, for $p\in [2,3)$.  Consider $u^n:=\pi(u^0;\G ^n)$ and $u:=\pi(u^0;\G )$. We obtain $\omega$-wise the following Wong-Zakai type convergence result 
		\begin{equation}
		\lim_{n\rightarrow +\infty}\|u^n-u\|_{L^\infty(H^1)\cap L^2(H^2)}^2=0\, ,
		\end{equation}
		with speed of convergence $1$, i.e.  
		\begin{align}
		\|\pi(u^0,\mathbf{G}^n)-\pi(u^0,\mathbf{G})\|_{L^\infty(H^1)\cap L^2(H^2)}\lesssim \|\mathbf{G}^n-\mathbf{G}\|_{\mathcal{RD}^p_a(H^2)}\, .
		\end{align}
	\end{corollary}

\begin{remark}
	In the case of higher order regularity (i.e.\ for $u_0\in H^k$ and $\G\in \RD^p_a(H^{k+1})$ with $k\ge2$),
	similar conclusions as that of Theorem \ref{thm:wong-zakai} and Corollary \ref{cor:WZ} will be shown to hold in the spaces \( L^\infty(0,T;H^k)\cap L^2(0,T;H^{k+1}) \). See Theorem \ref{th:cont_wz_k} and Corollary \ref{cor:wz_k} below.
	\end{remark}
		
	As already discussed in the introduction, the properties described in Theorem \ref{thm:wong-zakai} and Corollary \ref{cor:WZ} constitute a direct motivation for studying a rough paths formulation of \eqref{rLLGn}. 
	We note that similar considerations were made in \cite[Chap.11]{FrizVictoire}, but for a somewhat different class of rough partial differential equations (viscosity solutions). 
	As underlined in \cite[Chap 19]{FrizVictoire}, once the continuity of the It\^o-Lyons map is established, one can simply hinge on the ``contraction principle'' to obtain a large deviation principle for vanishing noise. 
	The following result will indeed be shown in this manner (see Section \ref{sec:LD} for a proof).

	\begin{theorem} (Freidlin-Wentzell large deviations)
		Fix $k\in \mathbb{N}$, $p\in (2,3)$ and $u^0\in H^k(\T ;\mathbb{S}^2)$. Let $\G=\Lambda_{g(\cdot)}\mathcal F(\rp) \colon \Omega\to \RD ^p_a(H^{k+1})$ be a Brownian rough driver of the form \eqref{simple_RD} where $\rp$ is as in Example \ref{exa:B} with \( H=\frac12 \) and $g\in H^{k+1}(\T;\R)$.
		Let $\epsilon>0$ and $u^\epsilon=\pi(u^0; \Lambda_{\sqrt\epsilon}\G )=\pi(u^0; (\sqrt\epsilon G,\epsilon\GG))$ be the pathwise solution to \ref{rLLG} obtained from the deterministic solution map $\pi$ on $\T\times [0,T]$. Let moreover $u=\pi(u^0; 0)$.
		Denote by $u=\pi(u^0, h)$ the solution driven by the Cameron Martin path $h\in\mathcal{H}$ enhanced to a rough driver.
		Then $(\mathbb{P}\circ(u^\epsilon)^{-1})_{\epsilon>0}$ satisfies a large deviations principle in $L^\infty(H^k)\cap L^2(H^{k+1})\cap \V^p(H^{k-1})$ with good rate function
		\begin{align}
		\mathcal{J}(y):= \inf _{h\in\mathcal{H}}\big(\mathcal{I}(h) : \pi(u^0,\Lambda_{g(\cdot)}\mathcal{F}(h))=y\big)\, ,
		\end{align}
		where $\mathcal{I}$ is defined in \eqref{good_rate_fc_BM}.
	\end{theorem}

Another application of our continuity results for the It\^o-Lyons map is the following Stroock-Varadhan like support theorem, the proof of which will be presented in Section \ref{sec:LD}.
	Again, it will be an easy further consequence of Corollary \ref{cor:WZ} and the existence of a ``nice description'' of the support of the enhanced rough Brownian motion in the $p$-variation topology for $p\in (2,3)$ (see \cite[Chap.~15]{FrizVictoire}).

	\begin{theorem}
	Let $p\in(2,3)$, $k\in \mathbb{N}$, $u^0\in H^k(\T ;\mathbb{S}^2)$ and $u=\pi(u^0,\G )\in L^\infty(H^k)\cap L^2(H^{k+1})=:\mathcal{Y}$ be the solution to \eqref{rLLG}, where in the notations of Section \ref{sec:RD}, 
	\[
	\G(\omega;x) =(g(x)\mathcal F(\delta\beta(\omega)),g(x)^2\mathcal F^{\otimes 2}\rpp(\omega))\in \RD ^p_a(H^{k+1})
	\]
	is the Brownian rough driver as defined in \eqref{simple_RD} with $H=\frac12$.
	Then the support of its distribution is the closure $\overline{\pi(u^0;\Lambda_{g(\cdot)}\circ\mathcal F(\mathcal D))}^{p\mathrm{-var};\mathcal{Y}}$ of the image through the solution map $\pi(u^0;\Lambda_{g(\cdot)}\circ\mathcal F(\cdot))  $ of the space of dyadic rough paths with respect to the $\mathcal{Y}$-norm, for $p\in (2,3)$. 
	\end{theorem}
	
	\section{A priori estimates and more}
	\label{sec:apriori}
The purpose of this section is to collect a priori estimates for solutions of a \textit{deterministic} rough LLG. Namely, the sample parameter is considered as fixed and one fixes all throughout a realization $\G:=\G(\omega)$ of the random rough driver considered in Theorem \ref{thm:existence}. 
We will also ignore for the moment the dependency of solutions with respect to that sample.
Later on, these results will be applied in order to obtain existence and uniqueness of pathwise solutions.

	\subsection{Deterministic solutions and product formula}
	\label{sec:sewing}

	Our purpose here is to state an equation on a product of the form $a_t^i(x)b^j_t(x)$ for two unknowns 
	which are solving a system:
	\begin{align}
	\label{deterministic_generic}
	&d a^{i}= f^id t + (d\A a)^i ,\quad 
	\quad \text{on}\enskip  [0,T]\times \T\,,
	\\
	\label{deterministic_generic_b}
	&d b^{i}= g^id t + (d\B b)^i 
	\quad 
	\quad i=1,\dots n,
	\end{align}
	where $\A=(A,\AA)\in \RD^p(H^1)$ is $n$-dimensional (i.e.\ it the first component defines an element of \( \mathcal L(\R^n) \) for almost every \( (t,x)\in [0,T]\times\T \)) and similarly for $\B=(B,\BB).$
We assume for convenience that both $a$ and $b$ are bounded from $[0,T]\to L^2(\T;\R^n)$ and moreover that
	the drift terms $f,g$ are square-integrable in time and space (note that they may depend on the unknowns).	
	
	At first, we need a notion of \textit{deterministic} weak solution for a generic rough sytem of PDEs of the form \eqref{deterministic_generic} with some given initial datum.
	\begin{definition}
		\label{def:sol_deter}
	Let $q\ge1$ and fix $a^0\in L^q(\T;\R^n)$.
	We say that $a\colon [0,T]\times \T\to \R^n$ solves \eqref{deterministic_generic} in $L^q$, with initial condition $a^0$, if 
	\begin{enumerate}[label=(\roman*)]
		\item $t\mapsto a_t(\cdot)$ belongs to $L^\infty(L^q(\T;\R^n))$,
		$f$ belongs to $L^q(L^q(\T;\R^n))$;
		\item 
there exists $a^{\natural}$ in $\V^{1-}_{2,\mathrm{loc}}(L^q)$ such that
		\begin{align}
		\label{rough_generic}
		\delta a_{s,t}
		-\int_s^tfd r=A_{s,t}a_s + \AA_{s,t}a_s +a^{\natural}_{s,t}\,,\quad \forall s\le t \in [0,T]\, .
		\end{align}
		\item $\lim_{t\downarrow 0}\langle a_t,\phi\rangle=\langle a^0,\phi\rangle$ for each $\phi\in L^2 (\T;\R^n).$
	\end{enumerate}
	\end{definition}

	We need a small modification of the product formula in \cite{hocquet2018ito}. Here we focus on the case when the two unknowns $a,b$ take values in $\R^n$ for some $n\ge1$, since it will only be needed in this form throughout the paper. 
We also restrict to the torus $\T$ for the sake of readability, even though generalizations to other domains are available.
Since we shall work with systems all throughout, it is convenient at this stage to introduce some notation on finite-dimensional tensor products. It should be noted that the following is consistent with the notations \eqref{GG_tensor}-\eqref{simple_RD}. 
\begin{notation}
	\label{nota:tensor}
	\begin{itemize}
			\item 
	For two vectors $a,b\in \R^n$, we let $a\otimes b\in\R^{n\times n}$ be the element 
	\begin{equation}\label{tensor_vector}
	(a\otimes b)^{i,j}=a^ib^j\,, \quad 1\le i,j\le n\,.
	\end{equation}
			\item
	For two linear maps $M,N\in\mathcal L(\R^n;\R^n),$ we let $M\otimes N\in \mathcal L(\R^{n\times n};\R^{n\times n})$ be the linear map
	\begin{equation}\label{tensor_matrix}
	M\otimes N\colon \R^{n\times n} \to \R^{n\times n},\quad \quad 
	\Phi \mapsto \left[\sum\nolimits_{1\leq k,l\leq n}M^{ik}N^{jl}\Phi ^{kl}\right]_{1\leq i,j\leq n}\,.
	\end{equation}
	\end{itemize}
\end{notation}
	
With Notation \ref{nota:tensor} at hand, the product formula reads as follows.

	\begin{proposition}[Product formula]
		\label{pro:product}
		Fix an integer $n\ge1$
		and let $a=(a^{i})_{i=1}^n\colon [0,T]\to L^2(\T;\R^n)$ (resp.\ $b=(b^i)_{i=1}^n\colon [0,T]\to L^2(\T;\R^n)$) be a bounded path, given as a weak solution of the system
		\eqref{deterministic_generic} (resp.\ \eqref{deterministic_generic_b}) on $[0,T]\times \T,$
		for some $f\in L^2(L^2)$ (resp.\ $g\in L^2(L^2)$).
		We assume that both  
		\[
		\A=\left (A^{i,j}_{s,t}(x),\AA^{i,j}_{s,t}(x)\right )_{\substack{1\leq i,j\leq n;\\ s\le t \in [0,T];x\in\T }}
		\quad 
		\B=\left (B^{i,j}_{s,t}(x),\BB^{i,j}_{s,t}(x)\right )_{\substack{1\leq i,j\leq n;\\ s\le t \in [0,T];x\in\T }}
		\]
		are $n$-dimensional geometric rough drivers
		of finite $(p,\frac p2)$-variation with $p\in[2,3)$ and with coefficients in $H^{1}(\T)$.
		Then the following holds:
		\begin{enumerate}[label=(\roman*)]
			\item \label{Q_shift}
			The two parameter mapping $\boldsymbol\Gamma^{\A,\B}\equiv(\Gamma^{A,B},\bbGamma^{\A,\B})$ defined for $s\le t \in [0,T] $ as
			\begin{equation}
			\label{tensorized_driver}
			\begin{aligned}
			&\Gamma_{s,t}^{A,B}:= A_{s,t}\otimes \mathbf 1 + \mathbf 1\otimes B_{s,t} \, ,\quad 
			\\
			&\bbGamma_{s,t}^{\A,\B}:=\AA_{s,t}\otimes \mathbf 1+ A_{s,t}\otimes B_{s,t}+ \mathbf 1\otimes \BB_{s,t}\, ,
			\end{aligned}
			\end{equation}
			where $\mathbf 1\equiv \mathbf 1_{n\times n}\in \mathcal L(\R^n)$ is the identity, is a $n^2$-dimensional rough driver (in the sense of Definition \ref{defi-RD}),
			\item \label{prod_uv}
			The product $v^{\otimes 2}_t(x)=(a^i_t(x)b^j_t(x))_{1\leq i, j\leq n}$ is bounded as a path in $L^1(\T;\R^{n\times n}).$ Moreover, it is a weak solution, in $L^1,$ of the system
			\begin{equation}
			\label{concl:prod}
			d (a\otimes b)= (a\otimes g +f \otimes b )d t + d \boldsymbol\Gamma^{\A,\B} [a\otimes b]\,.
			\end{equation}
		\end{enumerate}
	\end{proposition}
	
	\begin{proof}
		Proceeding component-wise the proof is carried out by the same arguments as that of \cite{hocquet2018ito},
		hence we only sketch it (see also \cite{hofmanova2020navier}).
		At first we observe that the new unknown
		$v_t(x,y)=a_t(x)\otimes b_t(y)$ satisfies an equation similar to \eqref{concl:prod}, but where this time the variables are doubled. 
		The aim is then to show that the corresponding system can be restricted to the diagonal 
		$\{(x,x)\in\T\times\T\}$, which amounts to integrating it against the singular test 
		function $\delta_x (y)\phi(x),$ for any smooth $\phi\colon \T\to \R$. 
		This is done by considering the family $\Phi_\epsilon(x,y)=(2\epsilon)^{-1}\psi(\frac{x-y}{2\epsilon})\phi(\frac{x+y}{2})$ for some $\psi\geq0$ regular enough which integrates to one, and any $\phi$ as above. One needs then to justify that the remainder $\langle a^{\otimes2,\natural}_{s,t},\Phi_\epsilon\rangle$ is uniformly bounded in $\frac{p}{3}$-variation norm, which allows to conclude by a variant of Banach-Alaoglu Theorem. 
		We omit the details.
	\end{proof}
	
	\begin{remark}[symmetric tensor product in $\R^n$]
	Let us introduce the following notation for two vectors $a,b\in \R^n$:
	\[
	a\odot b= \frac12(a\otimes b + b\otimes a)\, .
	\]
	Then for $a=b$, formula \eqref{concl:prod} can be rephrased in a convenient way as
	\begin{equation}\label{product_square}
	\delta a_{s,t}^{\otimes 2}
	=
	2\int_s^ta\odot fd r
	+ 2a_s\odot[(A_{s,t} +\AA_{s,t})a_s] + (A_{s,t}a)^{\otimes2} + (a^{\otimes 2})^{\natural}_{s,t}
	\end{equation}
	for every $ s\le t \in [0,T].$
	\end{remark}
	
	\subsection{The $L^\infty(0,T;H^1)\cap L^2(0,T;H^2)\cap 
	\mathcal{V}^p(0,T;L^2)$-estimate }
	\label{subsec:bootstrap}
In this section we prove that deterministic solutions are bounded in $L^\infty(H^1)\cap L^2(H^2)\cap \mathcal{V}^p(L^2)$ and that this bound only depends on the size of the initial datum and the rough driver $\G\in \RD^p_a(H^2)$.

	The first paragraph is devoted to the proof that any solution starting in the sphere $\mathbb S^2$ will remain sphere-valued for all times. This will result from the standard Gronwall Lemma and the simple observation that, loosely speaking, because of geometricity and Remark \ref{rem:levy}, the contributions of the rough input vanish almost everywhere when the solution is tested against itself (in the sense of a product formula on $u$ similar to \eqref{product_square}). 
	This is however a very specific case of product formula, in that it cannot be used to evaluate other types of test functions in general.
	For that purpose, the second paragraph introduces a fundamental tool in the rough integration, that is how to estimate the remainder in a sytem of the generic form \eqref{rough_generic}: the presented construction will be useful in other parts of the paper and could also be applied to other equations. 
	In the third paragraph we apply these ideas to prove the desired a priori estimate.

\subsubsection{The norm constraint}
	We need to prove that any solution $u$ which is smooth enough satisfies the spherical constraint, that is the first requirement in Definition \ref{def:solution}. This is a consequence of the following result.
\begin{lemma}[Norm constraint]
	\label{lem:constraint}
	Let $u^0\in H^1(\T ;\mathbb{S}^2)$ and $\G \in \RD ^p_a(H^1)$. Let $u\in L^\infty(L^2)\cap L^2(H^1)$ be so that it satisfies \eqref{rLLG} and $\int_{0}^{T}\|\mathcal{T}_{u_r}\|^2_{L^2}dr<+\infty$ holds. Assume further that
	$u_0(x)\in\mathbb{S}^2$ while 
	\[
	\|\partial_x u\|_{L^1(0,T;L^\infty)}<\infty\,.
	\]
	Then, we have $u_t(x)\in \mathbb{S}^2$ a.e.
\end{lemma}

\begin{proof}[Proof of Lemma \ref{lem:constraint}]
	From the product formula, we can consider
	\begin{align}
	\label{delta_u2}
	\delta( \|u\|_{L^2}^2)_{s,t}- 2\int_{s}^t \int_{\T}u_r\cdot \tt_{u_r} d x d r
	- (\|u^{\natural,2}\|_{L^2}^2)_{s,t}
	= \int_{\T} [2G_{s,t} u_s\cdot u_s+
	2\GG _{s,t}u_s\cdot u_s
	+|G_{s,t}u_s|^2]  d x\,  .
	\end{align}
	Because of anti-symmetry, one has $G_{s,t}u_s\cdot u_s = -u_s\cdot G_{s,t}u_s=0$ a.e., moreover Remark \ref{rem:levy} implies
	the identity
		\begin{align*}
		2\GG _{s,t} u_s\cdot u_s+G_{s,t}u_s\cdot G_{s,t}u_s
		&=2(\tfrac12G_{s,t}^2 + \LL_{s,t} u_s)\cdot u_s+G_{s,t}u_s\cdot G_{s,t}u_s
		\\&
		=-G_{s,t}u_s\cdot G_{s,t}u_s+G_{s,t}u_s\cdot G_{s,t}u_s
		=0\,.
		\end{align*}
	Hence the rough term in \eqref{delta_u2} vanishes.
	Going back to \eqref{delta_u2}, this leads to
	\[
	\delta (\|u\|_{L^2}^2)_{s,t} - 2\int_{s}^t \int_{\T}u_r\cdot \tt_{u_r} d x d r
	=
	 (\|u^{\natural,2}\|_{L^2}^2)_{s,t}\,.
	\]
	This means in particular that the term $(\|u^{\natural,2}\|_{L^2}^2)_{s,t}$ is an increment. Being also an element of $\V^{1-}(0,T;\R)$ by definition, this implies by a classical result that
	\[
	(\|u^{\natural,2}\|_{L^2}^2)_{s,t}= 0,\quad \text{for each}\enskip s\leq t\in[0, T] \, .
	\]
	Consequently, we only need to evaluate the drift term.
	Observing that $\Delta |u|^2=2u\cdot \Delta  u + 2|\nabla u|^2$, we can rewrite the latter as
	\[
	-2\int_{s}^t \int_{\T}u_r\cdot \tt_{u_r} d x d r
	=-\int_{s}^t \int_{\T}\partial_x ^{2} (|u_r|^2) d x d r+ 2\int_{s}^t \int_{\T}|\partial_x 
	u_r|^2(1-|u_r|^2) d x d r\,,
	\]
	which leads to the following equation on $\rho _t :=1/2(|u_t|_{\R^3}^2-1):$
	\begin{equation}
	\label{eq:rho_1}
	\begin{aligned}
	\partial _t\rho -\partial_x ^{2} \rho -|\partial_x u|^2\rho =0\,,
	\quad \text{on}\enskip [0,T]\times\T\, ,
	\\
	\rho _0=0\,.
	\end{aligned}
	\end{equation} 
	Testing \eqref{eq:rho_1} against $\rho $, we end up with the inequality
	\[
	\|\rho _t\|^2_{L^2} + \int_0^t\|\partial _x \rho _r\|_{L^2}^2 d r
	\leq \int_0^t\|\partial_x u_r\|^2_{L^\infty}\|\rho _r\|_{L^2}^2 d r\,,
	\]
	and the fact that $\rho =0$ a.e.\ follows by Gronwall.
\end{proof}

	\subsubsection{Remainder estimate and rough standard machinery}
	 \label{Section_remainder_existence}
	The computation that we perform here can be generalised to different settings. For that reason we prefer a more general exposition by working at the level of a generic equation. Up to some index change and different definition of the coefficient of the noise, the computations below hold also for classical rough paths (instead of rough drivers).

	Let $f\in L^2(H^k) $,  $a_0\in H^k$ and $\mathbf{A}=(A,\AA )\in \mathcal{RD}^p_a(H^{k+1})$ for $p\in [2,3)$ and $k\geq 1$. Assume that there exists a solution $a\in L^\infty(H^1)\cap L^2(H^2)$ to the abstract equation\footnote{Note that the integrand $f$ may depend on the solution $a$.}
\begin{equation}\label{eq;abstract_eq_existence}
\delta a_{s,t}=\int_{s}^{t}f_rdr+A_{s,t}a_s+\AA _{s,t}a_s+a^\natural_{s,t}\,,
\end{equation}
by which we mean that $a^\natural$ (defined implicitly by the above equation) belongs to $\mathcal{V}^{1-}_{2,\mathrm{loc}}([0,T];L^2)$. 
By applying $\partial_x$ to both sides, one sees that a similar equation is satisfied by $\partial_x a$.
In particular, we can think of
\eqref{eq;abstract_eq_existence} and the equation on its derivative as an $n+n$-dimensional system of the form \eqref{deterministic_generic}, with unknown $v=(a,\partial_x a)$. 
In that case, Proposition \ref{pro:product} implies the system
\begin{equation}
\label{eq;square_dx_a}
\left \{
\begin{aligned}
&\delta (a^{\otimes 2})_{s,t}
=\int_{s}^{t}2a\odot f_rdr
+ (J+ \JJ+ \tilde\JJ)(0,0)_{s,t}
+(a^{\otimes 2})^\natural_{s,t}
\\
&\delta (a\otimes \partial_xa)_{s,t}
=\int_{s}^{t}(a\otimes \partial_x f + f\otimes \partial_x a)dr
+ (J+ \JJ+ \tilde\JJ)(0,1)_{s,t}
+(a\otimes \partial _x a)^\natural_{s,t}
\\
&\delta (\partial_xa\otimes a)_{s,t}
=\int_{s}^{t}(\partial_xa\otimes f + \partial_xf\otimes a)dr
+ (J+ \JJ+ \tilde\JJ)(1,0)_{s,t}
+(\partial_xa\otimes a)^\natural_{s,t}
\\
&\delta (\partial_xa^{\otimes 2})_{s,t}
=\int_{s}^{t}2\partial_x a\odot \partial_xf_rdr
+ (J+ \JJ+ \tilde\JJ)(1,1)_{s,t}
+(\partial_xa^{\otimes 2})^\natural_{s,t}\,.
\end{aligned}\right .
\end{equation}
The terms $J,\JJ,\tilde\JJ$ are explicitly determined by the formula \eqref{product_square}.
For the last line, we have for instance
\begin{equation}
\label{nota:HJK_1}
\begin{aligned}
J^{i,j}_{s,t}(1,1)&=(A_{s,t}\partial_x a_s + \partial_x A_{s,t}a_s)^i \partial_xa_s^j + (A_{s,t}\partial_x a_s +\partial_x A_{s,t}a_s)^j \partial_xa_s^i,
\\
\JJ^{i,j}_{s,t}(1,1)&=(\AA _{s,t}\partial_x a_s + \partial_x\AA _{s,t} a_s)^i \partial_x a_s^j + (\AA _{s,t}\partial_x a_s+\partial_x\AA _{s,t} a_s)^j \partial_x a_s^i,
\\
\tilde\JJ^{i,j}_{s,t}(1,1)&=(A_{s,t}\partial_x a_s + \partial_xA _{s,t} a_s)^i(A_{s,t}\partial_x a_s+\partial_xA _{s,t} a_s)^j\,,
\end{aligned}
\end{equation}
for each $1\leq i,j\leq 3$.
The precise computations will be exposed in the next part of this paragraph, 
here we illustrate the strategy.
Most of the time in the following, our goal will be to integrate in space the last line of \eqref{eq;abstract_eq_existence}, with the purpose to derive some Gronwall type estimate.
 It is useful to think of integration as `testing the equation against the constant map $(\mathbf 1^{i,j})=(1\text{ if }i=j;\enskip0\text{ otherwise})\in \R^{n\times n}$'. In particular, we aim to find a $p$-variation estimate on $(s,t)\mapsto\langle (\partial_xa)_{s,t}^{\otimes 2;\natural},\mathbf 1\rangle$.
 Since $\mathbf 1$ certainly belongs to $H^1(\T;\R^{n\times n})$, we can use the (obvious) inequality $\langle (\partial_xa)_{s,t}^{\otimes 2;\natural},\mathbf 1\rangle\lesssim \|(\partial_xa)_{s,t}^{\otimes 2;\natural}\|_{H^{-1}}$ and, as was done in \cite{deya2016priori}, try to show an upper bound on the latter quantity. For that purpose we aim to expand 
\[
\delta [(\partial_xa ^i\partial_xa^j)^{\natural}]_{s,u,t}\equiv(\partial_xa ^i\partial_xa^j)^{\natural}_{s,t}-(\partial_xa ^i\partial_xa^j)^{\natural}_{s,u}-(\partial_xa ^i\partial_xa^j)^{\natural}_{u,t},\qquad s<u<t \in[0,T],
\]
and use the Sewing Lemma (Lemma \ref{lemma_sewing}).
A caveat lies in the fact that while computing $\delta [\cdot]^\natural_{s,u,t}$, the remainder term $(\partial_x a ^{\otimes2})_{s,u}^\natural$ appears again in the corresponding expansion. The estimate on the former depends, through \eqref{remainder_est_SL}, on an estimate on the latter, hence leading to a circular spot.
We deal with this issue in pretty much the same way as what was done in \cite{deya2016priori}, noticing that the remainder appearing on the right hand side is multiplied a `small' factor, hence it can be absorbed to the left (in a certain sense).
An additional hurdle that we have to overcome here is the appearance of mixed terms like $(\dd^n a 
\otimes  \dd^m a)^{\natural}$ with $0\leq n,m\leq N$ but such that $n+m<2N$. 
When $N=1$ for instance, one sees that the $H^{-1}$ bound on $(\partial_x a^{\otimes 2})^{\natural}$ needs an estimate on the two other terms $(a\otimes \partial_x a)^\natural,$ $(a\otimes \partial_x a)^{\natural}$, which themselves require an estimate on $(a^{\otimes2})^{\natural}$. 
Fortunately, the equation on $a^{\otimes2}$ has now a closed form, and so its remainder can be estimated independently of other quantities, based on the main argument of \cite{deya2016priori}.

Obviously, some computational difficulties will arise in the case of arbitrary $N\in\mathbb N$, and this justifies the introduction of a certain proof pattern which will be frequently used in the sequel.  
We call  it \textit{rough standard machinery}.
It consists of some standard steps that we need in order to estimate the remainder: 
\begin{enumerate}[label=RM\arabic*]
	\item\label{RM1}\!\!: \textbf{(the a priori estimate)} we find a suitable estimate for $\delta 
	(\partial_x ^N a^i \partial_x ^M a^j)^{\natural}_{s,u,t}$ that depends on the $p$-variation 
	norm of $(\partial_x ^N a^i \partial_x ^M a^j)^{\natural}$ by means of combinatorial formulas 
	based on Chen's relations (these are discussed in full generality in Appendix 
	\ref{app:combinatorial}).
	In the proof of the continuity of the It\^o-Lyons map in Section~\ref{sec:wong-zakai}, this 
	step consists of obtaining a special form of $\delta (\partial_x ^N (a-b)^i \partial_x ^M 
	(a-b)^j)^{\natural}_{s,u,t}$ as a consequence Theorem \ref{teo:rem_derivatives} and then 
	substituting the equation or the equation for the derivatives, where needed.
	\item\label{RM2}\!\!: \textbf{(the sewing Lemma)} from \ref{RM1} the hypothesis of Lemma 
	\ref{lemma_sewing} are fulfilled and we apply it. This leads to an estimate that presents the 
	same norm of  $(\partial_x ^N a^i \partial_x ^M a^j)^{\natural}$ both on the left and on the 
	right hand side of the inequality.
	\item\label{RM3}\!\!: \textbf{(absorb and conclude)} since the element in the right hand side 
	of the inequality in \ref{RM2} is multiplied by a control, it can be absorbed to the left hand 
	side up to the choice of a suitably fine partition. This leads to an estimate of  $(\partial_x 
	^N a^i \partial_x ^M a^j)^{\natural}$.
\end{enumerate}

It is useful to also introduce some notation for the drift terms in \eqref{eq;abstract_eq_existence}.
\begin{notation}[Drift]
	 We define for each $s\le t\in [0,T]$
	 \[
	 \D(1,1)_{s,t}:= \int _s^t (\partial_xa\otimes \partial_xf + \partial_xf\otimes \partial_x a)dr
	 \]
in the sense of a Bochner integral, in $H^{-1}$.
	 It is obvious that the map $(s,t)\mapsto\D_{s,t}(1,1)\in H^{-1}$ has finite variation
	 since 
	 \[
	 \|\D(1,1)_{s,t}\|_{H^{-1}}\leq \int _s^t \|\partial_xa\otimes \partial_xf + \partial_xf\otimes \partial_x a\|_{H^{-1}}dr =: \omega_{\D(1,1)}(s,t)\,
	 \]
	 and the right hand side is finite according to our hypotheses on $a,f$ and H\"older Inequality.
	 Similarly, we define $\D(1,0)_{s,t},$ $\D(0,1)_{s,t}$, $\D(0,0)_{s,t}$ as the drift terms appearing in each of the lines of \eqref{eq;abstract_eq_existence} and we denote their respective $1$-variation in $H^{-1}$ on any interval $[s,t]$ by $\omega_{\D(1,0)}(s,t),$ $\omega_{\D(0,1)}(s,t),$ $\omega_{\D(0,0)}(s,t)$.
\end{notation}

We now apply in detail the rough standard machinery in the first step of the proof of Theorem \ref{th:remainder_interate}.
\begin{theorem}\label{th:remainder_interate}
	In the notations of this section, we can estimate $(\partial_x  a^i \partial_x  a^j)^{\natural}$ locally  in time by
	\begin{multline}\label{eq:estimate_rem_existence}
	\|(\partial_x  a^{\otimes2})^{\natural}_{s,t}\|_{H^{-1}}
	\lesssim 
	\omega^{1/p}(s,t) \big[\omega_{\D (1,1)}+\omega_{\D (1,0)}+\omega_{\D (0,1)}+\omega_{\D (0,0)}\big](s,t)
	+\omega^{3/p}(s,t)\|a\|^2_{L^{\infty}(H^1)}\,,
	\end{multline}
	where $\omega:=\omega_{\mathbf{A}, H^2}$. 
	
	In addition, if $\A$ is anti-symmetric, then one has the simplified estimate
		\begin{equation}\label{eq:estimate_rem_simplified}
	\langle(\partial_x  a^{\otimes2})^{\natural}_{s,t},\mathbf 1\rangle
	\lesssim 
	\omega^{1/p}(s,t) \big[\omega_{\D (1,0)}+\omega_{\D (0,1)}+\omega_{\D (0,0)}\big](s,t)
	+\omega^{3/p}(s,t)\|a\|^2_{L^{\infty}(H^1)}\,,
	\end{equation}
	where we recall that $\mathbf 1^{i,j}:=1$ if $i=j$ and $0$ otherwise. 
\end{theorem}

\begin{proof}
	We note that in view of the previous discussion,
	in order to estimate $\partial_x  a ^{\otimes 2;\natural}$, we need to estimate also $(\partial_x  a \otimes a)^{\natural}$,  $( a \otimes \partial_x a)^{\natural}$, $a ^{\otimes 2;\natural}$. We start therefore from $N=0,M=0$.
	
	Recall that each remainder term is defined locally on a covering of $[0,T]$, that we denote by $(P_k)_k$. What follows holds for every $s\leq t \in P_k$ for every fixed $k$. Each control appearing in this proof is evaluated in such interval $(s,t)$ (we do not specify it for simplicity of notations), unless otherwise stated. In order to simplify the notations, we introduce a general control $\omega=\omega_{\mathbf{A};H^2}$ on $[s,t]$ as above. We also use the following notation: $\omega_A:=\omega_{A;H^1}$, $\omega_{ \AA }:=\omega_{ \AA ; H^1}$.
	
	\item{\textbf{\indent Step 1: explicit computations for $a^{\otimes 2;\natural}_{s,t}$.}}\\
	Assume $N=0, M=0$. 
	Recalling the notation \eqref{tensorized_driver} for the rough driver $\boldsymbol\Gamma^{\A,\A}$, then the equation looks like
	\begin{equation}
	\label{eq;square_a}
	\begin{aligned}
	\delta (a^ia^j)_{s,t}
	&=\D ^{i,j}(0,0)+\Gamma^{A,A}[a^{\otimes 2}_s]^{i,j}+\bbGamma^{\A,\A}[a^{\otimes 2}_s]^{i,j}+(a^ia^j)^{\natural}_{s,t}.
	\\
	&=\D ^{i,j}(0,0)+(A_{s,t}a_s)^ia^j_s+a^i_s(A_{s,t}a_s)^j+(\AA _{s,t}a_s)^ia^j_s\\
	&\quad\quad\quad\quad\quad\;+a^i_s(\AA _{s,t}a_s)^j+(A_{s,t}a_s)^i(A_{s,t}a_s)^j+(a^ia^j)^{\natural}_{s,t}\,.
	\end{aligned}
	\end{equation}
	Here we start to implement the steps described previously. 
	Applying $\delta$ to both sides of \eqref{eq;square_a} and using Chen's relation \eqref{chen-rela}, we see that the quantity $\delta (a^{\otimes 2,\natural})_{s,u,t}$ has the form 
	\[
	\delta (a^{\otimes 2,\natural})_{s,u,t}=
	\Gamma^{A,A}_{u,t}[\delta a^{\otimes 2}_{s,u}- J_{s,u}(0,0)] + \bbGamma^{\A,\A}_{u,t}\delta a^{\otimes 2}_{s,u}\, .
	\]
	Substituting the equation on $\delta a^{\otimes 2}_{s,u}$
	(see also Proposition \ref{pro:remainder_existence} with $N=M=0$),
	we thus find for $i,j=1,\dots n$
	\begin{equation}\label{first_delta_a}
	\begin{aligned}
	\delta (a^ia^j)^{\natural}_{s,u,t}
	&=\sum\nolimits_{1\leq l\leq 3} A^{i,l}_{u,t}[(\D + \JJ + \tilde\JJ)^{l,j}_{s, u}(0,0)+a^{2,\natural,l,j}_{s,u}] \\
	&\quad \sum\nolimits_{1\leq l\leq3} A^{j,l}_{u,t}[(\D + \JJ + \tilde\JJ)^{l,i}_{s, u}(0,0)+a^{2,\natural,l,i}_{s,u}]\\
	&\quad \sum\nolimits_{1\leq l\leq3} \AA ^{i,l}[(\D +J+ \JJ + \tilde\JJ)^{l,j}_{s, u}(0,0) +a^{2,\natural,l,j}_{s,u}]\\
	&\quad \sum\nolimits_{1\leq l\leq3} \AA ^{j,l}[(\D +J+ \JJ + \tilde\JJ)^{l,i}_{s,u}(0,0)+a^{2,\natural,l,i}_{s,u}]\\
	&\quad{  \sum\nolimits_{1\leq l,m\leq3} A^{i,l}_{u,t} A^{j,m}_{u,t}[(\D +J+ \JJ + \tilde\JJ)^{l,m}_{s,u}(0,0)+a^{2,\natural,l,m}_{s,u}]}\, ,
	\end{aligned}
	\end{equation}
	where
	\[
	\begin{aligned}
	J_{s,u}^{i,j}(0,0)&
	\equiv(\Gamma_{s,u}^{A,A}a_s^{\otimes 2})^{i,j}
	= (A_{s,u}a_s )^i a_s^j + a_s^i(A_{s,t}a_s)^j\, ,
	\\
	\JJ_{s,u}^{i,j}(0,0)&= (\AA_{s,u}a_s )^ia_s^j + a_s^i(\AA_{s,t}a_s)^j \, ,
	\\
	\tilde\JJ_{s,u}(0,0)&= (A_{s,u}a_s )^i(A_{s,u}a_s )^j\,.
	\end{aligned}
	\]
	Using several times the inequality
	 \[
	\|ab\|_{H^{-1}}\lesssim \|a\|_{H^{-1}}\|b\|_{H^1},\quad \forall (a,b)\in H^{-1}\times H^1
	\]
	(which is an easy consequence of the embedding $H^1\hookrightarrow L^{\infty}$),
	we observe that
	\[ \| (\D + \JJ + \tilde \JJ)(0,0)_{s, u} \|_{H^{-1}} \lesssim \| \D (0, 0)_{s, u} \|_{H^{-1}} + \| \AA _{s, u}\|_{L^{\infty}} \| a_s\|_{L^2}^2 + \| A_{s, u} \|_{L^{\infty}}^2 \| a_s \|_{L^2}^2, 
	\]
	\[\| J(0,0)_{s,u} \|_{H^{-1}} \lesssim \| A_{s, u} \|_{H^1} \|a_s\|_{L^2}^2, \]
	hence taking the $H^{-1}$-norm in \eqref{first_delta_a} yields
	\begin{multline*}
	\|\delta (a^{2, \natural})_{s, u, t} \|_{H^{-1}} 
	\lesssim 
	\| A_{u, t}\|_{H^1}\left [\|\D  (0, 0)_{s, u}\|_{H^{-1}} + \|\AA _{s,u}\|_{H^1}\|a_s\|_{L^2}^2 
	+ \| A_{s, u} \|_{H^1}^2 \| a_s\|_{L^2}^2 
	+ \| a^{\otimes 2, \natural}_{s, u} \|_{H^{-1}}\right ]
	\\
	+ (\| \AA _{u, t} \|_{H^1} + \| A_{u, t} \|_{H^1}^2) 
	\\
	\cdot \left[\| \D  (0, 0)_{s, u} \|_{H^{-1}} + \| A_{s, u}\|_{H^1}\| a_s \|_{L^2}^2 
	+ \| \AA _{s, u} \|_{H^1}
	\left\|a_s \right\|_{L^2}^2 + \| A_{s, u} \|_{H^1}^2 \|a_s\|_{L^2}^2 
	+ \| a^{\otimes 2, \natural}_{s, u} \|_{H^{-1}} \right] \,.
	\end{multline*}
	Then,
	\begin{multline}\label{eq:M1}
	\| \delta (a^{2, \natural})_{s, u, t} \|_{H^{-1}} 
	\lesssim \omega_A^{1 / p}(u,t) \left [\omega_{\D (0,0)}(s,u) + (\omega_{\AA }^{2/ p} +
	\omega_A^{2 / p})(s,u) \| a \|_{L^{\infty}(L^2)}^2 + \omega_{a^{2,\natural};H^{-1}}^{3/p}(s,u)\right ]
	\\
	 + (\omega_{\AA }^{2/ p} + \omega_A^{2 / p})(u,t) \left [\omega_{\D (0,0)}(s,u) +
	(\omega_A^{1 / p} + \omega_{\AA }^{2/p} + \omega_A^{2 / p})(s,u) \|a\|_{L^{\infty} (L^2)}^2 
	+ \omega_{a^{2, \natural};H^{-1}}^{3/p}(s,u)\right ]\,.
	\end{multline}
	
	The estimate \eqref{eq:M1} concludes the first part \ref{RM1} of the rough standard machinery. As noted before, the estimate of the $H^{-1}$-norm of $\delta a^{2,\natural}_{s,t}$ depends itself on the control $\omega^{3/p}_{a^{2,\natural}}$, which eventually has to be absorbed into the left hand side.
	To this end, in \ref{RM2}, we  want to apply the Sewing Lemma (Lemma \ref{lemma_sewing}): we can apply it because in \eqref{eq:M1} the right hand side is a sum of controls of  $p/3$ bounded variation, and therefore a control itself of  $p/3$ bounded variation. Hence the application of Lemma \ref{lemma_sewing} leads to 
	\begin{align}\label{M2}
	\begin{aligned}
	\|a^{2,\natural}_{s,t}\|_{L^1}
	&\lesssim\omega_A^{1 / p}\left[\omega_{\D (0,0)} + \omega_{\AA ;H^{-1}}^{3/p}\right](s,t) \\
	&\qquad + (\omega_{\AA }^{2/ p} + \omega_A^{2 / p}) \left[\omega_{\D (0,0)} +
	(\omega_A^{1 / p} + \omega_{\AA }^{2 / p} + \omega_A^{2 / p}) \|a\|_{L^{\infty} (L^2)}^2 
	+ \omega_{a^{2, \natural};H^{-1}}^{3 / p}\right](s,t)\,,
	\end{aligned}
	\end{align}
	and this concludes the step  \ref{RM2}. Now, take the power $p/3$, use the super-additivity property of the controls, pass to the supremum on the partitions and take the power $3/p$. 
	In this way the $p$-variation norm $\omega^{3/p}_{a^{2,\natural};H^{-1}}$ appears both on the left and on the right hand side in \eqref{M2}, i.e.\ there is a constant $C_p>0$ such that
	\begin{align}\label{eq:M111}
	\begin{aligned}
	\omega^{3/p}_{a^{2,\natural};H^{-1}}(s,t)
	&\leq 
	C_p\bigg[\omega_A^{1 / p}  \left[\omega_{\D (0,0)} + (\omega_{\AA }^{2/ p} +
	\omega_A^{2 / p}) \| a \|_{L^{\infty} (L^2)}^2 + \omega_{a^{2, \natural};H^{-1}}^{3/p} \right](s,t) \\
	&\qquad + (\omega_{\AA }^{2/ p} + \omega_A^{2 / p}) \left [\omega_{\D (0,0)} +
	(\omega_A^{1 / p} + \omega_{\AA }^{2 / p} + \omega_A^{2 / p}) \|a\|_{L^{\infty} (L^2)}^2 + \omega_{a^{2, \natural};H^{-1}}^{3/p}\right ](s,t)\bigg]\,.
	\end{aligned}
	\end{align}
	By choosing a covering $(\tilde{P}_k)_k$ of $[0,T]$ which is finer than $(P_k)$, we see that for all $s\leq t\in \tilde{P}_k$ we have
	\[
	\omega_A^{1/p}(\tilde{P}_k),\omega_{\mathbb A}^{2/p}(\tilde{P}_k)<\min\{1/(3C_p),1\}\,,
	\]
	we can then absorb the terms $\omega_{A;H^1}^{1/p}\omega^{3/p}_{a^{2,\natural};H^{-1}}$ and $\omega_{A;H^1}^{1/p}\omega^{3/p}_{a^{2,\natural};H^{-1}}$ in the left hand side,
	leading to an estimate on $\omega^{3/p}_{a^{2,\natural};H^{-1}}.$
	Since $\|a^{2,\natural}_{s,t}\|_{L^1}$ can be estimated above by its $p/3$-variation, we conclude that for $s\leq t\in \tilde{P}_k$
	\begin{align}\label{eq:N=0,M=0}
	\begin{aligned}
	\|a^{2,\natural}_{s,t}\|_{L^1}&
	\lesssim\omega_A^{1 / p}(s,t) \left [\omega_{\D (0,0)}(s,t) + (\omega_{\AA }^{2/ p} +
	\omega_A^{2 / p})(s,t) \| a \|_{L^{\infty} (L^2)}^2 \right ]
	\\
	&\qquad + (\omega_{\AA }^{2/ p} + \omega_A^{2 / p})(s,t) \left [\omega_{\D (0,0)} +
	(\omega_A^{1 / p} + \omega_{\AA }^{2 / p} + \omega_A^{2 / p}) \|a\|_{L^{\infty} (L^2)}^2 \right ](s,t)\, ,
	\end{aligned}
	\end{align}
	which closes the step \ref{RM3} and the estimate of Step 1.
	
	\item{\textbf{\indent Step 2: estimates on $(\partial_x  a \otimes  a)^{\natural}$ and $( a \otimes \partial_x  a)^{\natural}$.}} 
	We want to proceed as in Step 1.
	According to Proposition~\ref{pro:remainder_existence}, the new expression of the remainder $(\partial_x  a \otimes  a)^{\natural}$ depends on both $\delta (\partial_x  a  \otimes a)$  and $\delta ( a \otimes  a)$.
	From the same argument as before, the remainder is therefore estimated by 
	\begin{align*}
	\|(\partial_x  a\otimes a)^{\natural}_{s,t}\|_{L^1}
	&\lesssim\omega_{\mathbf{A},H^1}^{1/p}\omega^{2/p}(s,t)\|a\|_{L^\infty (H^1)}^2+\omega^{1/p}[\omega_{\D (1,0)}+\omega_{\D (0,0)}](s,t)+\omega_{\partial_x  A;H^1}^{1/p}\omega^{3/p}_{a^{2,\natural}}(s,t)\,.
	\end{align*}
	The above estimate depends on $\omega^{3/p}_{a^{\natural,2}}$ which we estimated  in \eqref{eq:N=0,M=0}. This leads to the final estimate of this step
	\begin{align}\label{N=1,M=0}
	\|(\partial_x  a \otimes a)^{\natural}_{s,t}\|_{L^1}&
	\lesssim\omega_{\mathbf{A},H^1}^{1/p}\omega^{2/p}(s,t)\|a\|_{L^\infty(H^1)}^2+\omega^{1/p}[\omega_{\D (1,0)}+\omega_{\D (0,0)}](s,t)\,.
	\end{align}

	\item{\textbf{\indent Step 3: \textbf{estimate on $(\partial_x  a \otimes 
	\partial_x  a)^{\natural}$} in Sobolev norm.}}  From the rough standard 
	machinery, the estimate becomes
	\begin{align*}
	\|(\partial_x  a \otimes \partial_x  a)^{\natural}_{s,t}\|_{L^1}&\lesssim 
	\omega^{1/p} \big[\omega_{\D (1,1)}+\omega_{\D (0,1)}+\omega_{\D 
	(1,0)}]+\omega^{2/p}\omega_{\partial_x 
	\mathbf{A},H^1}^{1/p}\|a\|^2_{L^\infty(H^1)}+\omega^{1/p}[\omega^{3/p}_{(\partial_x
	  a a)^{\natural}}+\omega^{3/p}_{a^{2,\natural}}]\\
	&\quad +\omega_{\mathbf{A},H^1}^{1/p} [\omega_{\D (1,1)}+\|\partial_x a\|^2_{L^\infty(L^2)}]+\omega_{\partial_x\mathbf{A},H^1}^{1/p}\|a\|^2_{L^\infty(L^2) }.
	\end{align*}
	Since we can deal with $\omega^{3/p}_{(\partial_x  a a)^{\natural}}+\omega^{3/p}_{a^{2,\natural}}$ by means of \eqref{eq:N=0,M=0} and \eqref{N=1,M=0},  we infer the first estimate \eqref{eq:estimate_rem_existence}.

\item{\textbf{\indent Step 4: refined estimate on $\langle\partial_x  a^{\otimes2 ;\natural},\mathbf 1\rangle$}.}
 We now address \eqref{eq:estimate_rem_simplified}, assuming that $\A$ is anti-symmetric.
 Applying $\delta$ to both sides of \eqref{eq;square_dx_a}, using Chen's relation \eqref{chen-rela}, then testing against $\mathbf 1\equiv\mathbf 1_{n\times n}$, we find with $\boldsymbol\Gamma=\boldsymbol\Gamma^{\A,\A}$:
 \[
 \begin{aligned}
 \langle\delta (\partial_xa^{\otimes 2})^{\natural}_{s,u,t},\mathbf1\rangle
 &=\big\langle \Gamma_{u,t}[\delta \partial_xa_{s,u}^{\otimes 2} - J_{s,u}(1,1)] 
 + \bbGamma_{u,t}\delta \partial_xa_{s,u}^{\otimes 2},\mathbf 1\big\rangle
 \\
 &\quad 
 +\Big\langle (\partial_x A_{u,t}\otimes \mathbf 1 + \mathbf 1\otimes A_{u,t})[\delta (a\otimes \partial_xa)_{s,u} - J_{s,u}(0,1)] 
 \\&\quad \quad \quad 
 +(A_{u,t}\otimes \mathbf 1 + \mathbf 1\otimes \partial_x A_{u,t})[\delta (\partial_x a\otimes a)_{s,u} - J_{s,u}(1,0)] 
 \\&\quad \quad \quad 
 + (\partial_x\AA_{u,t}\otimes \mathbf 1 + \mathbf 1\otimes \AA_{u,t}+\partial_xA_{u,t}\otimes A_{u,t})\delta (a\otimes \partial_x a)_{s,u}
 \\&\quad \quad \quad 
 + (\AA_{u,t}\otimes \mathbf 1 + \mathbf 1\otimes \partial_x\AA_{u,t}+A_{u,t}\otimes \partial_xA_{u,t})\delta (\partial_xa\otimes a)_{s,u}\enskip,\enskip\mathbf 1\Big\rangle
 \,.
 \end{aligned}
 \]
 It will be enough to show that the term $\langle \Gamma_{u,t}[\delta \partial_xa_{s,u}^{\otimes 2} - J_{s,u}(1,1)] 
 + \bbGamma_{u,t}\delta \partial_xa_{s,u}^{\otimes 2},\mathbf 1\big\rangle$ vanishes, since for the remaining terms we can simply rely on the $H^{-1}$ estimates obtained in the previous steps, together with the fact that $\|\mathbf 1\|_{H^1}\lesssim1$.
To this aim we note that since $\A$ is anti-symmetric, we have $A_{u,t}(x)\in\mathcal L_a(\R^n)$, and hence for $i,j=1,\dots,n$
 \[
 [\Gamma_{u,t}^\star\mathbf 1]^{i,j} = -[(A_{u,t}\otimes \mathbf 1 + \mathbf 1 \otimes A_{u,t})\mathbf 1]^{i,j}
 = - \sum_{k,l}(A^{i,k}_{u,t}1_{j=l}+ 1_{i=k}A^{j,l}_{u,t})1_{k=l} = -A_{u,t}^{i,j}-A_{u,t}^{j,i}=0\,.
 \]
 Moreover, thanks to Remark \ref{rem:levy} we can write $\AA_{s,t}=\frac12(A_{s,t})^2+\LL_{s,t}$ where $\LL_{s,t}(x)$ is again some anti-symmetric matrix. 
 But then it is also true that
 $(\LL_{u,t}\otimes \mathbf 1 + \mathbf 1 \otimes \LL_{u,t})^\star\mathbf 1 = 0$
 and one finds 
 \[
 \bbGamma^{\star}_{u,t}\mathbf 1=(\tfrac12A_{u,t}^2\otimes \mathbf 1 + \mathbf 1 \otimes \tfrac12A_{u,t}^2+A_{u,t}\otimes A_{u,t})^\star\mathbf 1 = 
 \frac12(A_{u,t}\otimes \mathbf 1 + \mathbf 1 \otimes A_{u,t})^2\mathbf 1 =0\,.
 \]
Finally
\[
\langle \Gamma_{u,t}[\delta \partial_xa_{s,u}^{\otimes 2} - J_{s,u}(1,1)] 
+ \bbGamma_{u,t}\delta \partial_xa_{s,u}^{\otimes 2},\mathbf 1\rangle
=\langle \delta \partial_xa_{s,u}^{\otimes 2} - J_{s,u}(1,1) ,\Gamma_{u,t}^{\star}\mathbf 1\rangle 
+ \langle \delta \partial_xa_{s,u}^{\otimes 2},\bbGamma_{u,t}^{\star}\mathbf 1\rangle
=0\,,
\]
hence our claim.
\end{proof}

	\subsubsection{A priori estimate}
\label{ss:apriori}

This paragraph is devoted to the derivation of an a priori bound: by means of it, we will prove that any solution to the perturbed equation \eqref{rLLGxi} is bounded in terms of a constant depending only on the rough path norm of the perturbation.
From now to the end of this paragraph, we thus fix a deterministic rough driver $\G\in\RD^p_a(H^2)$ and assume that $u\in L^\infty(H^1)\cap L^2(H^2)\cap \V^{p}(L^2)$ is a solution of \eqref{LLG}, in the sense of Definition \ref{def:sol_deter} for $f=u\times \partial_xu- u\times (u\times\partial_x^2u).$
For the reader's convenience, recall that what this means is that for any $s<t\in[0,T]$
\begin{equation}\label{LLG_apriori}
\delta u_{s,t}
= \int _s^t (u\times \partial_x^2 u-u\times(u\times \partial_x^2 u))dr
+G_{s,t}u_s+\GG_{s,t}u_s+ u_{s,t}^{\natural}\, ,
\end{equation}
for some $u^{\natural}\in \V_{\mathrm{loc}}^{p/3}(L^2)$.

We recall the statement of the rough Gronwall lemma (\cite{deya2016priori}), which can be considered as counterpart of the classical statement in the $p$-variation setting and will be needed in the sequel.
\begin{lemma}[Rough Gronwall lemma]
	\label{lem:gronwall}
	Let $E\colon[0,T]\to \R_+$ be a path such that there exist constants $\kappa,\ell>0,$ a super-additive map $\varphi $ and a control $\omega $ such that:
	\begin{equation}\label{rel:gron}
	\delta E_{s,t}\leq \left(\sup_{s\leq r\leq t} E_r\right)\omega (s,t)^{\kappa }+\varphi (s,t)\, ,
	\end{equation}
	for every $s\le t \in [0,T]$ under the smallness condition $\omega (s,t)\leq \ell$.
	
	Then, there exists a constant $\tau _{\kappa ,\ell}>0$ such that
	\begin{equation}
	\label{concl:gron}
	\sup_{0\leq t\leq T}E_t\leq \exp\left(\frac{\omega (0,T)}{\tau _{\kappa ,\ell}}\right)\left[E_0+\sup_{0\leq t\leq T}\left|\varphi (0,t)\right|\right].
	\end{equation}
\end{lemma}

Our objective now is to obtain a relation of the form \eqref{rel:gron} for the map $E_t= \|\partial_xu_t\|^2_{L^2}+ \int_0^t\|\tt_{u_r}\|^2_{L^2}dr$ and then conclude that the latter is bounded in terms of controls depending only on the quantities $\|u_0\|_{H^1}$, $\|G\|_{\V^p(H^2)}$ and $\|\GG\|_{\V^{p/2}(H^2)}$.
We use the equation on $\partial_x u$, which thanks to \eqref{eq;square_dx_a} writes as
\begin{equation}
\label{eq:energy_equation_complete}
\begin{aligned}
\delta (\partial_x  u^{\otimes 2})_{s,t}
&=
2\int _s^t\partial_x u\odot (u\times\partial_x^2 u -u\times(u\times \partial_x^2u))dr
\\&\quad \quad \quad \quad \quad \quad 
+ 2\partial_x  u_s\odot ( G_{s,t} \partial_x  u_s + \partial _xG_{s,t} u_s)+ 2(\GG _{s,t} 
\partial_x  u_s +\partial _x\GG _{s,t} u_s)\odot\partial_x  u_s 
\\&\quad\quad \quad \quad \quad \quad \quad 
+ (G_{s,t}\partial_x  u_s + \partial_x  G_{s,t}u_s)^{\otimes 2}+ (\partial_x  u^{\otimes 
2})_{s,t}^{\natural}\,,
\end{aligned}
\end{equation}
where we recall that $a\odot b=\frac12(a\otimes b+b\otimes a).$
The strategy we use relies on the remainder estimate, Theorem \ref{th:remainder_interate}, which provides an a priori bound on $\langle (\partial_xu)^{\otimes 2;\natural},\mathbf 1\rangle=\sum_{i}\int_{\T}(\partial_x u ^i\partial_x u^i)^{\natural}_{s,t}(x)dx.$ 
We have the following result.

\begin{proposition}
	\label{pro:energy}
	Let $u^0\in H^1(\T ;\mathbb{S}^2)$, $\G \in \RD_a ^p(H^2)$ and let $u\in L^\infty(H^1)\cap L^2(H^2)$ be a solution to \eqref{LLG_apriori} such that $|u_t(x)|_{\R^3}=1$ a.e.
	Then there exists a universal constant $C_p>0$ 
	\[
	\sup_{t\in[0,T]}\|\partial_x u_t\|_{L^2}^2 + \int_0^T \|\tt _{u_t}\|_{L^2}^2dt\leq 
	C_p\|\G \|_{\V^{p}\times \V^{p/2}(H^2)}\|u_0\|_{H^1}^2 
	\,,
	\]
	where $\tt_u:=\partial_x ^{2} u+ u|\partial_x  u|^2=-u\times(u\times \partial_x^2u)$.
\end{proposition}

\begin{proof}
Testing \eqref{eq:energy_equation_complete} against $\mathbf 1=(1_{i=j})_{1\le i,j\le3}\in\R^{3\times3}$ yields by anti-symmetry of $\G$
\[
\begin{aligned}
\delta (\|\partial_x  u\|_{L^2}^2)_{s,t}
&=
2\int _s^t\int_{\T}\partial_x u\cdot (u\times\partial_x^2 u -u\times(u\times \partial_x^2u))dxdr
\\&\quad \quad \quad \quad 
+ 2\int_{\T}[\partial_x  u_s\cdot \partial _xG_{s,t} u_s
+ 2\partial_x  u_s\cdot \partial _x\GG _{s,t} u_s
\\&\quad \quad \quad \quad \quad \quad 
+ 2G_{s,t}u_s\cdot \partial_x G u_s + |\partial_x  G_{s,t}u_s|^{2}]dx+ \int _{\T}(|\partial_x  
u|^2)^{\natural}_{s,t}dx\,.
\end{aligned}
\]
	To estimate the remainder term, we need to evaluate the following drift terms
	in $H^{-1}$
	\[
	\begin{aligned}
	&\D_{s,t}(1,0)= 2\int _s^t\partial_xu\odot (\tt_u + u\times \tt_u)dr\, ,
	\\&
	\D_{s,t}(0,1)= 2\int _s^tu\odot \partial_x(\tt_u + u\times \tt_u)dr\, ,
	\\&
	\D_{s,t}(0,0)= 2\int _s^tu\odot (\tt_u + u\times \tt_u)dr\,.
	\end{aligned}
	\]
	 We start with $\D(0,0)$. Observing that $L^1(\T)\hookrightarrow H^{-1}(\T),$ we can estimate the $H^{-1}$-norm by the $L^1$-norm, yielding 
	\begin{equation}
	\label{est:D00}
	\begin{aligned}
	\|\D_{s,t}(0,0)\|_{H^{-1}} 
	&\lesssim \int_{s}^{t}\|u_r\odot(\tt _{u_r} + u_r\times\tt_{u_r})\|_{L^1} d r
	\lesssim \int_s^t \|\tt_{u_r}\|^2_{L^2} d r\,,
	\end{aligned}
	\end{equation}
since $\|u\|_{L^\infty}=1$ and $\tt_u$, $u\times \tt_u$ are orthogonal.
	As for $\D(0,1),$ we can write for any test function $\Phi\in H^{1}(\T;\R^{3\times 3})$:
	\begin{equation}
	\label{est:D01}
	\begin{aligned}
	\langle \D(0,1),\Phi\rangle 
	&= 2\sum_{i,j}\int_s^t\int _{\T} u^i\partial_x(\tt_u + u\times\tt_u)^j\Phi^{i,j}(x)dx\, dr
	\\
	&= -2\sum_{i,j}\int_s^t\int _{\T} \big[\partial_xu^i(\tt_u + u\times\tt_u)^j\Phi^{i,j}
	+u^i(\tt_u + u\times\tt_u)^j\partial_x\Phi^{i,j}\big]dx\,dr
	\\
	&\lesssim
	\int_s^t\big[\|\partial_xu\|_{L^2}\|\tt_u + u\times\tt_u\|_{L^2} \|\Phi\|_{L^\infty}
	+\|(\tt_u + u\times\tt_u)\|_{L^2}\|\partial_x\Phi\|_{L^2}\big]dr
	\\
	&\lesssim
	\|\Phi\|_{H^1}\int_s^t\big[(1+\|\partial_xu\|_{L^2})\|\tt_u + u\times\tt_u\|_{L^2}\big]dr\,.
	\end{aligned}
	\end{equation}
	But $\|\tt_u + u\times\tt_u\|_{L^2}=\sqrt{2}\|\tt_u \|_{L^2}$ and therefore
	\[
	\begin{aligned}
	\|\D(0,1)\|_{H^{-1}}
	&=\sup_{\|\Phi\|_{H^1}\leq 1}\langle \D(0,1),\Phi\rangle 
	\\&
	\lesssim
	[1+\sup_{r\in[s,t]}\|\partial_xu_r\|_{L^2}](t-s)^{\frac12}
	\Big(\int_s^t \|\tt_{u_r} \|^2_{L^2}dr\Big)^{\frac12}
	\\&
	\lesssim
	(t-s)\Big[(1+\sup_{r\in[s,t]}\|\partial_xu_r\|_{L^2}^2)
	+
	\int_s^t \|\tt_{u_r} \|^2_{L^2}dr\Big]\,.
	\end{aligned}
	\]
	The estimate on $\D(1,0)$ is hardly different and therefore omitted.

	From these computations, we can conclude thanks to the remainder estimate \eqref{eq:estimate_rem_simplified} that
	\begin{align*}
	&\delta \left(\|\partial_x u\|_{L^2}^2\right)_{s,t}
	+ \int_s^t\|\tt _{u_r}\|^2_{L^2} d r
	\leq C_\epsilon \left (\omega_{\G ;H^2}^{1/p}(s,t)+(t-s)\right )
	\Big[\sup_{r\in[0,t]}\|\partial_x u\|_{L^2}^2 + \int_0^t\|\tt_{u_r}\|_{L^2}^2dr\Big]\, .
	\end{align*}
	From Lemma \ref{lem:gronwall} with $\omega(s,t):=(t-s)+\omega_{\G ;H^2}(s,t),\kappa=1/p$ and $\varphi(s,t)=0$, we get
	\[
	\sup_{t\in[0,T]}\|\partial _x u_t\|^2_{L^2}+\int_{0}^{T}\|\tt _{u_r}\|_{L^2}^2 d r
	\lesssim\exp \left(\omega(0,T)\right)\Big[\|\partial_x u_0\|^2_{L^2}+ \omega_{\G ;H^2}^{1/p}(0,T)\|u_0\|_{L^2}^2\Big] \, ,
	\]
yielding the claimed estimate.	
  \end{proof}
The following result is a consequence of Proposition \ref{pro:energy}. 
It shows that $u$ is also uniformly bounded in $L^2(H^2)$.
\begin{corollary} \label{cor:L^2(L^2)}
	Let $u_0\in H^1(\T ;\mathbb{S}^2)$, $\G \in \RD_a ^p(H^2)$ and let $u\in L^\infty(H^1)\cap L^2(H^2)$ be a solution to \eqref{LLG_apriori}. Then there exists a universal constant $C$ such that
	\begin{equation}
	\label{L2H2}
	\int_{s}^{t}\|\partial_x ^{2}u_r\|^2_{L^2} d r\leq C\| u_0\|_{H^1}^{6} T +C\| u_0\|^2_{H^1}.
	\end{equation}
		Moreover, $u$ belongs to $\mathcal{V}^p(L^2)$ and satisfies the estimate
	\begin{align}
	\label{u_p-var_L2}
	\|\delta u_{s,t}\|_{L^2}
	\lesssim \int_s^t\|\tt_{u_r}\|_{L^2}dr 
	\,.
	\end{align}
\end{corollary}
\begin{proof}
	Since the solution $u_t(x)$ takes values in the sphere $\mathbb{S}^2\subset \R ^3$, then a.e.\ 
	$x\in \T $ it holds that $0=\partial_x  |u|^2=2 \partial _x  u\cdot u $ and $\dd^2 u\cdot 
	u=-|\partial_x u|^2 $. Therefore for all fixed $t\in [0,T]$,
	\begin{align}\label{eq:step_L^2L^2_1}
	\int_{\T }\tt_{u_t}\cdot\partial_x ^{2} u_t   d x=
	&\int_{\T }\big(|\partial_x ^{2} u_t|^2+(\partial_x ^{2} u_t\cdot u_t)|\partial_x  u_t|^2 
	\big)  d  x
	=\int_{\T }\big(|\partial_x ^{2} u_t|^2-
	|\partial_x  u_t|^4\big)  d x.
	\end{align}
	Recall also that $\tt_u\cdot u|\partial_x u|^2=0$, hence $\tt_u\cdot \tt_u=\tt_u\cdot \partial_x^2 u$. By substituting \eqref{eq:step_L^2L^2_1} in the a priori estimate in Proposition \ref{pro:energy}, we get
	\begin{align*}
	\int_{0}^{t} \|\partial_x ^{2} u_r \|^2_{L^2}  d r
	&\lesssim \int_{0}^{t} \|\partial_x  u_r \|^4_{L^4}  d r +\| u_0\|^2_{H^1} 
	 \\&
	 \lesssim \int_{0}^{t} \|\partial_x  u_r \|^ {3}_{L^2}\|\partial_x ^{2} u_r\|_{L^2}  d 
	 r+\|u_0\|^2_{H^1},
	\end{align*}
	where we used the following Gagliardo-Nirenberg-Sobolev inequality in one dimension
	\begin{equation}\label{ineq:GN_L4}
	\|\partial_x u \|_{L^4}^2\lesssim \|\partial_x u \|_{L^2}^{3/4}\|\partial_x^2 u\|^{1/4}_{L^2}\,.
	\end{equation}
	From Young Inequality, we have for any $\epsilon>0$
	\begin{align*}
	\int_{0}^{t} \|\partial_x ^{2} u_r \|^2_{L^2}  d r
	 \leq C\Big[\int_{0}^{t}(\tfrac{1}{2\epsilon} \|\partial_x  u_r \|^ 
	 {6}_{L^2}+\tfrac{\epsilon}{2}\|\partial_x ^{2} u_r\|^{2}_{L^2} ) d r+\|u_0\|^2_{H^1}\Big]\,.
	\end{align*}
	From the estimate on $\|\partial_x u\|_{L^\infty(L^2)}$, we see by choosing $\epsilon\in(0,1/(2C))$ that there is another constant $\tilde C>0$ such that 
	\[
	\frac12\int_{0}^{T} \|\partial_x ^{2} u^n_r \|^2_{L^2}  d r \leq \tilde C\|u_0\|_{H^1}^{6} T 
	+\| u_0\|^2_{H^1}\,
	\]
	and this concludes our first estimate \eqref{L2H2}.

To show \eqref{u_p-var_L2}, we first note from orthogonality that
$|u_t- u_s|^2= |u_t|^2- |u_s|^2 - 2 u_s\cdot\delta u_{s,t}=- 2 u_s\cdot\delta u_{s,t}$,
but from the equation satisfied by $u$ we know that:
\begin{equation}
\label{u_delta_u}
\delta u_{s,t}\cdot u_s + \int_s^t(\tt_u + u\times\tt_u)\cdot\delta u_{s,r}dr = u^{\natural}_{s,t}\cdot u_s \,.
\end{equation}
To evaluate the right hand side, we can rely on the Sewing Lemma: we compute
\[
\|\delta(u^{\natural}\cdot u)_{s,\theta,t}\|_{L^1}= 
\|\delta u_{s,\theta}\cdot \int _\theta ^t(\tt_u + u\times \tt_u)dr \|_{L^1}
\leq \|\delta u\|_{\V^p(s,\theta;L^2)}\int_\theta ^t \|\tt_u+u\times\tt_u\|_{L^2}dr
\,,
\]
which has finite $p/3$-variation, in particular Lemma \ref{lemma_sewing} asserts that the right hand side of \eqref{u_delta_u} can be estimated in $L^1$ by the above quantity.
Putting all together, we see by using Cauchy-Schwarz and Young inequalities that
\[
\begin{aligned}
\|\delta u\|_{\V^p(s,t;L^2)}^2
&\lesssim 
\|\delta u\|_{\V^p(s,t;L^2)}\int_s^t\|\tt_u+u\times\tt_u\|_{L^2}dr
+\|u^{\natural }_{s,t}\cdot u_s\|_{L^1}
\\
&\lesssim 
\|\delta u\|_{\V^p(s,t;L^2)}\int_s^t\|\tt_u\|_{L^2}dr
\lesssim 
\frac{\epsilon}{2}\|\delta u\|^2_{\V^p(s,t;L^2)} + \frac{1}{2\epsilon}\int_s^t\|\tt_u\|_{L^2}dr
\end{aligned}
\]
for any $\epsilon>0.$ 
Choosing the latter parameter small enough gives the desired bound.
\end{proof}

\subsection{Generalization to higher orders}
Assume that $a\in L^\infty(H^k)\cap L^2(H^{k+1})$ satisfies $da= fdt+d\A a$ for some $f\in L^2(H^{k-1})$.
It is intuitively clear from the two previous subsections that, in order to obtain an a priori estimate for $\partial_x^ka^{\otimes2}$ in $L^\infty(L^2)\cap L^2(H^1)$ within our rough driver framework, an apparently inescapable step is to derive a uniform estimate on each of the remainders
\begin{align*}
(\partial_x ^N a\otimes \partial_x ^M a)_{s,t}^{\natural}:=\delta (\partial_x ^N a\otimes 
\partial_x ^M a)_{s,t}-(\D +J+\JJ +\tilde\JJ )(N,M)_{s,t}\,  ,
\end{align*}
for $i,j=1,\dots n$ and each $N,M=0,\dots,k$, 
where here $\D(N,M)$ denotes the drift term
\begin{equation}\label{drift_squared_equation}
\D ^{i,j}(N,M)_{s,t}:=\int_s^t \partial_x ^N f_r^i\partial_x ^Ma_r ^jdr+\int_s^t \partial_x 
^Na_r^i\partial_x ^M f_r^jdr\, ,
\end{equation}
and we introduce the following notation for the rough terms
\begin{equation}
J(N,M)_{s,t}
:=\sum_{n=0}^{N} \binom{N}{n} (\partial_x ^{N-n}A_{s,t}\partial_x ^{n} a_s)\otimes \partial_x ^Ma_s
+\sum_{n=0}^{M} \binom{M}{n} \partial_x ^Na_s\otimes (\partial_x ^{M-n}A_{s,t}\partial_x ^{n} 
a_s)\, ,
\end{equation}
\begin{equation}
\JJ(N,M)_{s,t}
:=\sum_{n=0}^{N} \binom{N}{n} (\partial_x ^{N-n}\AA_{s,t}\partial_x ^{n} a_s)\otimes \partial_x 
^Ma_s
+\sum_{n=0}^{M} \binom{M}{n} \partial_x ^Na_s\otimes (\partial_x ^{M-n}\AA_{s,t}\partial_x ^{n} 
a_s)\, ,
\end{equation}
\begin{equation}
\tilde\JJ(N,M)_{s,t}
:=\sum_{n=0}^{N} \binom{N}{n} (\partial_x ^{N-n}A_{s,t}\partial_x ^{n} a_s)\otimes \sum_{m=0}^{M} 
\binom{M}{m}(\partial_x ^{M-m}A_{s,t}\partial_x ^{m} a_s)\,.
\end{equation}
We introduce also the following controls
\begin{align*}
&\omega_{\D ^{i,j}(N,M)}(s,t):=\int_s^t \|\partial_x ^N f_r^i \partial_x ^M a^j_r
+\partial_x ^N f_r^i \partial_x ^M  a^j_r\|_{H^{-1}} dr\, , \\
&\omega_{ \D (N,M)}(s,t):=\sum_{i,j\in \{0,1,\dots,k\}} \omega_{ \D ^{i,j}(N,M)}(s,t)\, .
\end{align*}
The remainder $(\partial_x ^N a \otimes \partial_x ^M a)_{s,t}^{\natural}$ can be estimated using 
the rough standard machinery, as described in Section \ref{Section_remainder_existence}.
Since $N,M$ are arbitrary now, the first step \ref{RM1} leads to somewhat involved algebraic computations.
The following Proposition \ref{pro:remainder_existence} gives an explicit formula for 
$\delta(\partial_x ^Na \otimes \partial_x ^M a)_{s,u,t}^{\natural}$.

\begin{proposition}\label{pro:remainder_existence}
	Let $\mathbf 1=\mathbf 1_{n\times n}$. In the notations of this section, we can rewrite $\delta (\partial_x ^N a \otimes 
	\partial_x ^M a)^{\natural}_{s,u,t}$ as
	\begin{align*}
	\delta (\partial_x ^N a &\otimes \partial_x ^M a)^{\natural}_{s,u,t}
	=\sum_{n=0}^{N} \binom{N}{n} \partial_x ^{N-n}A_{u,t}\otimes \mathbf 1 \big[ (\D +\JJ 
	+\tilde\JJ)(N,M)_{s,u}+(\partial_x ^{n} a\otimes \partial_x ^M v)^{\natural}_{s,u} \big]
	\\
	&\quad
	+\sum_{m=0}^{M} \binom{M}{m} \mathbf 1\otimes \partial_x ^{M-m}A_{u,t} \big[ (\D +\JJ +\tilde\JJ 
	)(N,m)_{s,u}+(\partial_x ^N a\otimes \partial_x ^{m} a)^{\natural}_{s,u} \big]\\
	&\quad
	+\sum_{n=0}^{N} \binom{N}{n}\partial_x ^{N-n}\AA _{u,t}\otimes \mathbf 1 \big[ (\D +J+\JJ 
	+\tilde\JJ)(N,M)_{s,u}+(\partial_x ^{n} a\otimes \partial_x ^M a)^{\natural}_{s,u}\big]\\
	&\quad
	+\sum_{m=0}^{M} \binom{M}{m} \mathbf 1\otimes \partial_x ^{M-m}\AA _{u,t} \big[ (\D +J+\JJ +\tilde\JJ 
	)(N,M)_{s,u}+(\partial_x ^N a\otimes \partial_x ^{m} a)^{\natural}_{s,u}\big]\\
	&
	+\sum_{n=0}^{N} \sum_{m=0}^{M}  \binom{N}{n}  \binom{M}{m}  \partial_x ^{N-n}A_{u,t}\otimes 
	\partial_x ^{M-m}A_{u,t}\big[(\D +J+\JJ +\tilde\JJ )(n,m)+(\partial_x ^{n} a\otimes \partial_x 
	^{m} a)^{\natural}_{s,u}\big].
	\end{align*} 
\end{proposition}
\begin{proof}
	This is a consequence of Theorem \ref{teo:rem_derivatives}. Using \eqref{eq;abstract_eq} with 
	$b=0$ and $(B,\mathbb{B})=0$, we can indeed represent the remainder $\delta (\partial_x ^N a^i 
	\partial_x ^M a^j)^{\natural}_{s,u,t}$ as above. 
\end{proof}

	The same considerations as in Section \ref{Section_remainder_existence} can be iterated and the outcome is the following corollary.
	\begin{corollary}\label{cor:existence_rem_k}
		In the notation of this section,  we can estimate the remainder term $(\partial_x ^N a^i 
		\partial_x ^N a^j)^{\natural}$ as follows:
		\begin{align*}
		\|\partial_x^N a ^{\otimes 2;\natural}_{s,t}\|_{H^{-1}}
		&\lesssim_N \sum_{m=0}^{N}\sum_{k,\ell=0;(k,l)\neq (N,N)}^{m} \omega^{1/p}\omega_{\D (k,\ell)}(s,t)+\omega^{3/p}(s,t)\|a\|^2_{L^\infty (H^N)}
		\\
		&\quad +\omega_{\A,H^1}^{1/p}[\|\partial_x^N a\|^2_{L^\infty(L^2)}+\omega_{\D (N,N)}(s,t)]\, ,
		\end{align*}
		where $\omega=\omega_{\A,H^{N+1}}$.
	\end{corollary}
	
	\begin{proof}
		From Theorem \ref{th:remainder_interate} we know how to estimate $(\partial_x  a^{ \otimes2 })^{\natural}$. We assume as inductive step that the following estimate holds
		\begin{align*}
		\|\partial_x^{N-1} a ^{\otimes 2;\natural}_{s,t}\|_{H^{-1}}
		&\lesssim \sum_{m=0}^{N}\sum_{k,\ell=0}^{m}\big[\omega^{1/p}\omega_{\D 
		(k,\ell)}(s,t)+\omega^{2/p}\omega_{\partial_x^{\ell}A;H^1}^{1/p}(s,t)\|\partial_x^{m-\ell} 
		a\|^2_{L^\infty(L^2)}\big]\, ,
		\end{align*} 
		and we want to prove that the claim holds. As in Theorem \ref{th:remainder_interate} the 
		estimate to $\|(\partial_x^N a )^{\otimes 2;\natural}_{s,t}\|_{H^{-1}}$ depends on 
		$(\partial_x^N a\otimes \partial_x^{N-1} a)^\natural$ and $(\partial_x^k a \otimes 
		\partial_x^k a)^\natural$ for all $k=0,\dots,N-1$. 
		The latter are estimated thanks to the inductive hypothesis. As for $(\partial_x^N a 
		\otimes \partial_x^{N-1} a)^\natural$,
 as in Step 2 of Theorem  \ref{th:remainder_interate}, the mixed term depends on the previous 
 drifts and also on $\|\partial_x^N a \|_{L^\infty(L^2)}$. 
		All the contributions of the estimate of the mixed term are contained in the estimate on 
		$(\partial_x^{N-1} a \otimes \partial_x^{N-1} a)^{\natural}$ plus a term proportional to 
		$\|\partial_x^N a\|_{L^\infty(L^2)}$ and the $1$-variation of the drift $\D (N,N-1)$.
		
		Now we move to the estimate of  $\|(\partial_x^N a\otimes \partial_x^N 
		a)^{\natural}_{s,t}\|_{H^{-1}}$: from the rough standard machinery, the inductive step and 
		the considerations on $(\partial_x^N a\otimes \partial_x^{N-1} a)^{\natural}$ the claim 
		follows.
	\end{proof}

		\begin{remark}\label{example:time_indep_noise}
		It is worth noticing that in the case of space independent noise,
		the computations are considerably simplified, starting from the equality in Proposition \ref{pro:remainder_existence}. In the same spirit, it is also possible to conclude that, again in the case of space independent noise,
		\begin{align}\label{eq:time_indep_noise_N}
		\|(\partial_x^N a \otimes \partial_x^N a)^{\natural}\|_{H^{-1}}\lesssim 
		\omega^{1/p}_{\A;L^\infty}\omega_{\D (N,N)}+\omega_{ \A; L^\infty}^{3/p}\|\partial_x^N 
		a\|_{L^\infty(L^2)}^2\, .
		\end{align}
		\end{remark}

	\section{Application to existence and uniqueness}
	\label{sec:existence_uniqueness}
	
	The aim of this section is to make use of the a priori estimates in order to show existence, higher order regularity of the solution to the stochastic \eqref{rLLG}. Uniqueness is also shown for solutions in $L^\infty(H^1)\cap L^2(H^2)$.
		
	The strategy of the existence proof is to construct a sequence of solutions to \eqref{rLLG} with smooth noise, denoted by $(u^n)_n$, that is uniformly bounded in $L^\infty(H^1)\cap L^2(H^2)\cap\mathcal{V}^p(L^2)$.
	From a compactness argument, there exists a limit $u$ of a subsequence of  $(u^n)_n$, that we identify as a solution to \eqref{rLLG}.  
	Then we show that the solution is unique and, if the initial condition and the noise are more regular,  the solution has a better regularity as well.
	As a consequence of uniqueness, every subsequence of $(u^n)_n$  converges to the same limit $u$ and therefore in the stochastic case the pathwise solution can also be interpreted as a random variable adapted with respect to the natural filtration.

	\begin{theorem}
		Let $u^0\in H^1(\T ;\mathbb{S}^2)$ and $\G \in \RD _a^p(H^2)$. Then there exists a solution to \eqref{rLLG} in the sense of Definition \ref{def:solution} and it is a unique pathwise solution (in the sense that if $u_1, u_2$ are solutions to \eqref{rLLG} with same initial condition $u^0$, then $u_1=u_2$ a.e. in $[0,T]\times \T $).
	\end{theorem}

We recall that \eqref{rLLG} is driven by a rough driver $\G $ which is 
constructed from a  geometric lift of the linear map $(\cdot)\times w$. 
Therefore, there exists a sequence of smooth paths $(\Gamma^n)_n$ so that the 
associated driver $\G ^{n}$ constructed through the same procedure converges to 
$\G $ in the rough driver metric. 
From Appendix \ref{app:perturbed}, for each $n\in\mathbb N$ there is a unique solution $u^n=u^{\dot\Gamma^{n}}$ to the system
\begin{equation} \label{rLLGn}\tag{rLLGn}
\left\{\begin{aligned}
&du^n= [\partial_x ^{2} u^n+ u^n|\partial _x  u^n|^2 + u^n\times \partial_x ^{2} u^n]dt +  \dot 
\Gamma ^nu^n,\quad \text{on}\enskip [0,T]\times \T ,
\\
&u^n(0)=u^{0}\in H^1(\T ;\mathbb{S}^2)\,.
\end{aligned}\right.
\end{equation}

	\subsection{Proof of existence }\label{sub_section:proof_existence}
	From Lemma \ref{lem:constraint}, $\|u^n\|_{L^2}=\|u_0\|_{L^2}$.
	Because of the a priori estimate in Proposition \ref{pro:energy},
	\begin{align}\label{eq:existence1}
	\sup_{r\in[0,T]}\| \partial_x  u^n_r\|_{L^2}^2+\int_{0}^{T} \|u^n_r \times \partial_x ^{2} 
	u^n_r\|^2_{L^2} d r+ \int_{0}^{T} \|\partial_x ^{2} u^n_r \|^2_{L^2}  d r\lesssim\|\partial_x  
	u_0\|_{L^2}^2+\|\partial_x  u_0\|_{L^2}^{6} T \, ,
	\end{align}
	we conclude that the sequence $(u^n)_{n\in\mathbb{N}}$ is uniformly bounded in $L^\infty(0,T;H^1)\cap L^2(0,T;H^2)\cap \mathcal{V}^p(L^2)$.  
	
	The following Lemma is a small modification of a version of the Aubin-Lions lemma adapted to the rough path setting, that was proved in \cite[Appendix A]{hofmanova2019navier}.
	\begin{lemma}
		\label{lem:aubin}
		Let $\omega$ and $\bar{\omega}$ be controls on $[0,T]$ and let $L,\kappa>0$. Define
		\begin{align*}
		X:=L^2(H^2)\cap \{f\in C([0,T];L^2): \|\delta f_{s,t}\|_{L^2}\leq \omega (s,t)^\kappa, \;\forall s\leq t \in[0,T]\; \mathrm{s.t.}\;\bar{\omega}(s,t)\leq L \}
		\end{align*}
		endowed with the norm
		\begin{align*}
		\|f\|_X:= \|f\|_{L^2(H^2)}+\|f\|_{L^\infty(L^2)}+\sup\Big\{\frac{\|\delta g_{s,t}\|_{L^2}}{\omega(s,t)^\kappa}: s\leq t \in [0,T] \; \mathrm{s.t.}\; \bar{\omega}(s,t)\leq L\Big\}\, .
		\end{align*}
		Then $X$ is compactly embedded in $C([0,T];L^2)$ and $L^2(H^1)$.
	\end{lemma}
	We can apply this result to our situation and conclude that there exists a 
	subsequence, denoted again by $(u^n)_n$ and $u\in L^\infty(H^1)\cap 
	L^2(H^2)$, such that $u^n\rightarrow u$ strongly in $L^2(0,T;H^1)\cap C([0,T];L^2)$.	
	From \eqref{eq:existence1} and from Banach Alaoglu, there exists $F\in L^2(L^2)$ such that the following converge weakly in $L^2(L^2)$
	\begin{equation}\label{eq:conv_existence_1d}
	u^n \times \partial_x ^{2} u^n \rightharpoonup F, \quad \quad \partial_x ^{2} u^n 
	\rightharpoonup \partial_x^2u\,.
	\end{equation}
To identify $F$, note that
\[
u^n\times \partial_x^2u^n- u \times \partial_x^2u
= (u^n-u)\times \partial_x^2u^n + u\times (\partial_x^2u^n-\partial_x^2u)\,.
\]
The first term converges to $0$ in $L^1(L^2)$ since by H\"older inequality
\[
\int_0^T \|(u^n-u)\times \partial_x^2u^n\|_{L^2} dr 
\leq \left (\int_0^T\|u^n-u\|_{L^\infty}^2 dr\right )^{\frac12}\left (\int_0^T\|\partial_x^2u\|_{L^2}^2 dr\right )^{\frac12}
\lesssim\left (\int_0^T\|u^n-u\|_{L^\infty}^2dr\right )^{\frac12}\to 0
\]
thanks to the fact that $H^1(\T)\hookrightarrow L^\infty(\T)$ and $u^n\to u$ in $L^2(H^1).$
For the second term, we note that since $|u^n(x,t)|=1$ for a.e.\ $(t,x)\in [0,T]\times\T$, for every $\phi \in L^2(L^2)$ the function $-u\times \phi$ belongs to $L^2(L^2)$. By the identity $a\times b\cdot c= - b\cdot a\times c,\forall a,b,c\in \R^3$,
we infer from this and \eqref{eq:conv_existence_1d} that
\[
\int _0^T\langle u\times (\partial_x^2u^n-\partial_x^2 u),\phi\rangle dr
=-\int _0^T\langle \partial_x^2u^n-\partial_x^2 u,u\times\phi\rangle dr \to 0\,.
\]
Thus, $F=u\times \partial_x^2 u.$

	Then, from \eqref{eq:existence1} also 
	$$\int_{s}^{t}\|u^n_r\times (u^n_r\times \partial_x ^{2} u^n_r)\|^2_{L^2}  d r \leq 
	\int_{s}^{t}\| u^n_r\times \partial_x ^{2} u^n_r\|^2_{L^2}  d r \leq C\, . $$ 
	Again from Banach Alaoglu there exist $F^1\in L^2(L^2)$ such that $u^n\times(u^n\times 
	\partial_x ^{2} u^n) \rightharpoonup F^1$ weakly in $L^2(L^2)$ for $n\rightarrow +\infty$. To 
	see why $F^1=u\times(u\times \partial_x ^{2} u)$  let  $\phi \in L^4(L^4)$ be arbitrary and 
	note that
	\begin{align*}
	&\lim_{n\rightarrow +\infty} \int_{s}^{t}\langle F^1_r- u^n_r\times (u^n_r\times \partial_x 
	^{2} u^n_r), \phi \rangle_{L^2}  d r +\int_{s}^{t}\langle u^n_r\times (u^n_r\times \partial_x 
	^{2} u^n_r)-u_r\times (u_r\times \partial_x ^{2} u_r), \phi \rangle_{L^2}  d r\\
	&=  \lim_{n\rightarrow +\infty} \int_{s}^{t}\langle (u^n_r-u_r)\times (u^n_r\times \partial_x 
	^{2} u^n_r)+u_r\times (u^n_r\times \partial_x ^{2} u^n_r-u_r\times \partial_x ^{2} u_r), \phi 
	\rangle_{L^2}  d r=0
	\end{align*}
	with computations analogous to the ones above. 
	
	We now address the convergence of the noise terms. For all $\phi \in L^4(L^4)$
	\begin{align*}
	\lim_{n\rightarrow +\infty}	\langle G^{n}_{s,t} u^n_s-G_{s,t}u_s,\phi \rangle_{L^2}  
	= & \lim_{n\rightarrow +\infty}	\langle (G^{n}_{s,t} -G_{s,t})u^n_s+G_{s,t}(u^n_s-u_s),\phi\rangle_{L^2} =0\, .
	\end{align*}
	From H\"older's inequality 
	\begin{align*}
	| \langle (G^{n}_{s,t} -G_{s,t})u^n_s+G_{s,t}(u^n_s-u_s),\phi\rangle_{L^2}| 	&\leq \|G^{n}_{s,t} -G_{s,t}\|_{L^\infty}\|u^n_s\|_{L^2}\|\phi\|_{L^2}+\|G_{s,t}\|_{L^\infty}\|u^n_s-u_s\|_{L^2}\|\phi\|_{L^2}\, ,
	\end{align*}
	the strong convergence of the solution in $C([0,T];L^2)$ and the convergence of the noise, we identified the limit.
	
	Let $J=(J^i)_{i\in I}$ be a finite partition of $[0,T]$ such that for all $n\in \mathbb{N}$ the remainders $u^{n,\natural}\in \mathcal{V}^p_2(J^i;L^2)$ for all $i\in I$. This partition exists since the shrinking performed on the partitions in the previous proofs are independent on $n$ and in particular they depend only on the path $(G,\GG )$ .
	
	Let us fix $i\in I$. For all $s\leq t \in J^i$ and for all $\psi \in L^4$ we can define the remainder $u^\natural$  as the limit 
	\begin{align*}
	u^\natural_{s,t}:= \lim_{n\rightarrow +\infty}& \langle \delta 
	u^n_{s,t},\psi\rangle_{L^2}-\int_{s}^{t} \langle u^n_r\times (u^n_r\times \partial_x ^{2} 
	u^n_r), \psi \rangle_{L^2} d r -\int_{s}^{t} \langle u^n_r\times \partial_x ^{2} u^n_r, \psi 
	\rangle_{L^2}  d r\\
	&-\langle G^n_{s,t}u^n_s, \psi\rangle_{L^2}-\langle\GG ^n_{s,t}u^n_s, \psi\rangle_{L^2}\\
	=&\langle \delta u_{s,t},\psi\rangle_{L^2}-\int_{s}^{t} \langle u_r\times (u_r\times \partial_x 
	^{2} u_r), \psi \rangle_{L^2} d r -\int_{s}^{t} \langle u_r\times \partial_x ^{2} u_r, \psi 
	\rangle_{L^2} d r\\
	& -\langle G_{s,t}u_s, \psi\rangle_{L^2}-\langle\GG _{s,t}u_s, \psi\rangle_{L^2}\, .
	\end{align*}
	Hence there exists a map $u^\natural\in \mathcal{V}^p_{2,\mathrm{loc}}(J^i;L^2)$ such that the equation holds and the third requirement of Definition \ref{def:solution} is fulfilled. 
	This finishes the proof of existence.\hfill\qed

	\subsection{Proof of uniqueness}
	
	Let $a,b$ be two solutions to \eqref{rLLG} with the same initial condition $a(0)=b(0)=u^0$ .
	We denote by $z:=a-b$ and note that we have the vector identity 
	\begin{equation}\label{eq:div_laplacian}
	u\times \partial_x ^{2} u = \partial  _x (u\times \partial_x  u)\, .
	\end{equation}
	The equation of the difference of the two solutions has the form
	\[
	dz = \big[\partial_x ^{2} z + z|\partial _x  a|^2 + b\partial _x  z\cdot\partial _x  (a+b) + 
	\partial_x(z\times \partial _x  a )+ \partial_x(a\times \partial _x  z)\big]dt +  d \G  z\,.
	\]
	Applying the product formula on $z\otimes z$, and then testing against the constant function $\mathbf 1=(1 _{i=j})_{1,\le i,j\le3},$ we observe as before that the noise vanishes.
	Integrating by parts, this gives
	\[\begin{aligned}
	\delta (\|z\|_{L^2}^2)_{s,t} + 2\int_s^t\|\partial _x  z\|_{L^2}^2 d r 
	&=\int_{s}^t \int_{\T}|z_r|^2|\partial _x  a_r|^2 d x d r +\int_{s}^t \int_{\T} z_r\cdot 
	b_r\partial _x  z_r:\partial _x  (a_r+b_r) d x d r
	\\
	&\quad \quad 
	-\int_{s}^t \int_{\T}\partial _x  z _r\cdot (z_r\times\partial _x  a_r+ a_r\times \partial _x  
	z_r) d x d r=:\mathrm{I}+\mathrm{II} +\mathrm{III}\,.
	\end{aligned}
	\]
	For the first term, we have
	\[
	\mathrm{I}\leq \int_s^t\|z_r\|_{L^\infty}^2\|\partial _x  a_r\|_{L^2}^2 d r\,.
	\]
	Thanks to the well-known 1-dimensional interpolation inequality
	\begin{equation}\label{eq:interpolation_1D}
	\|a\|_{L^\infty}^2 \lesssim\|a\|_{L^2}\|a\|_{H^1}\, ,
	\end{equation} 
	together with Young inequality, we get for any $\epsilon >0:$
	\begin{equation}
	\label{est:I}
	\begin{aligned}
	\mathrm{I}
	&\lesssim \int_s^t \|z_r\|_{L^2}[\|z_r\|_{L^2} + \|\partial _x  z_r\|_{L^2}]\|\partial_x  
	a_r\|_{L^2}^2 d r
	\\
	&\lesssim\int_s^t\|\partial _x  a_r\|_{L^2}^2\|z_r\|^2_{L^2}  d r+ \frac{1}{2\epsilon 
	}\int_s^t\|\partial _x  a_r\|^4_{L^2}\|z_r\|^2_{L^2} d r + \frac{\epsilon 
	}{2}\int_s^t\|\partial _x  z_r\|^2_{L^2} d r\,.
	\end{aligned}
	\end{equation} 
	For the second term we have, denoting $U:=a+b:$
	\begin{equation}
	\label{est:II}
	\begin{aligned}
	\mathrm{II}
	&\leq 
	\int_s^t\|z_r\|_{L^\infty}\|\partial _x  U_r\|_{L^2}\|\partial _x  z_r\|_{L^2} d r
	\leq \frac{1}{2\epsilon }\int_s^t\|\partial _x  U_r\|^2_{L^2}\|z_r\|_{L^\infty}^2 d r+ 
	\frac{\epsilon }{2}\int_s^t\|\partial _x  z_r\|^2_{L^2} d r
	\\
	&\lesssim
	\frac{1}{2\epsilon }\int_s^t\left[\|\partial _x  U_r\|^2_{L^2}\|z_r\|_{L^2}^2  d r+ \|\partial 
	_x  U_r\|^2_{L^2}\|z_r\|_{L^2}\|\partial _x  z_r\|_{L^2} \right] d r + \frac{\epsilon 
	}{2}\int_s^t\|\partial _x  z_r\|^2_{L^2} d r
	\\
	&\lesssim
	\frac{1}{2\epsilon }\int_s^t[\|\partial _x  U_r\|^2_{L^2}\|z_r\|_{L^2}^2 +  \frac{1}{4\epsilon 
	^2}\|\partial _x  U_r\|^4_{L^2}\|z_r\|^2_{L^2} + \epsilon ^2\|\partial _x  z_r\|_{L^2}^2] d r + 
	\frac{\epsilon }{2}\int_s^t\|\partial _x  z_r\|_{L^2}^2 d r
	\\
	&\lesssim  C(\epsilon )\int_s^t[1+ \|\partial _x  a_r\|^4_{L^2} + \|\partial _x  b_r\|^4_{L^2}]
	d r + \epsilon \int_s^t\|\partial _x  z_r\|^2_{L^2} d r\, .
	\end{aligned}
	\end{equation} 
	For the third term, we have by orthogonality
	\[\begin{aligned}
	\mathrm{III} 
	&= 
	-\int_{s}^{t}\int_{\T}\partial _x  z_r \cdot (z_r\times\partial _x  a_r) d x\, d r
	\lesssim \int_s^t \|\partial _x  z_r\|_{L^2}\|z_r\|_{L^\infty}\|\partial _x  a_r\|_{L^2}  d r\, 
	,
	\end{aligned}\]
	and a similar analysis as for $\mathrm{II}$ gives
	\begin{equation}
	\label{est:III}
	\mathrm{III}\lesssim  C(\epsilon )\int_s^t[1+ \|\partial _x  a_r\|^4_{L^2}] d r + \epsilon 
	\int_s^t\|\partial _x  z_r\|^2_{L^2} d r\, .
	\end{equation} 
	Choosing $\epsilon >0$ small enough, we have by summation of \eqref{est:I}, \eqref{est:II} and\eqref{est:III}:
	\begin{align}\label{eq:drift_uniqueness}
	\delta \left(\| z\|_{L^2}^2\right)_{s,t} + \int_s^t\|\partial _x  z_r\|^2_{L^2} d r 
	\lesssim\int_s^t [1+ \|\partial _x  a_r\|_{L^2}^4+ \|\partial _x  b_r\|_{L^2}^4]\|z_r\|^2_{L^2} 
	d r \, ,
	\end{align}
	which implies by the classical Gronwall Lemma that $z=0,$ and hence $a=b$ for a.e. $(t,x)\in [0,T]\times \T $.
	
	This finishes the proof of uniqueness.\hfill\qed

	\subsection{Higher order regularity}
	\label{ssec:higher}
	As will be shown in the present subsection, an improved regularity on the 
	initial condition and on the noise leads to an improved regularity of the 
	solution.
	In the next proposition, uniform bounds on the non-linearities associated to $v=\partial_x^k u\otimes \partial_x^{k}u$ are proved by induction over $k$.
	Note that Proposition \ref{pro:energy} corresponds to the case $k=1$. 
	The drifts of higher order that we are concerned with are $\D (k,k-1)$ and $\D (k-1,k)$.
	\begin{proposition}\label{pro:existence_drift_k}
		Fix $k\geq 2$. Let $u\in L^\infty(H^k)\cap L^2(H^{k+1})$ so that $u(t,x)\in \mathbb{S}^2$ for a.e. $(t,x)\in [0,T]\times\T $. Then the following estimates hold
		\[
		\|\D(k,k-1)\|_{1\mathrm{-var};[s,t];H^{-1}}\enskip,\enskip \|\D(k-1,k)\|_{1\mathrm{-var};[s,t];H^{-1}}
		\lesssim
		\|u\|_{L^\infty(s,t;H^k)}^2+\|u\|_{L^2(s,t;H^{k+1})}^2\,.
		\]
	\end{proposition}
	\begin{proof}
		By an easy integration by parts argument, if the bound on $\D(k-1,k)$ 
		is proved, the same bound will be true for $\D(k,k-1)$ as well.
		To evaluate $\D(k-1,k)$, we follow the computations in \cite{Melcher}.
		For all $\Phi\in H^1(\mathbb{T},\mathbb{R}^{3\times3})$, consider
		 \[
		 \begin{aligned}
		 \langle \D(k-1,k)_{s,t},\Phi\rangle 
		 &=\sum_{i,j}\int _s^t\int _{\T} \Big[\partial_x^{k-1}(\partial_x^2u_r + 
		 u_r|\partial_x u_r|^2 +u_r\times \partial_x^2 u_r
		 )^i\partial_x^{k}u^j_r
		 \\&\quad \quad \quad 
		 +\partial_x^{k-1}u^i_r\partial_x^{k}(\partial_x^2u_r + 
		 u_r|\partial_x u_r|^2 +u_r\times \partial_x^2 u_r 
		 )^j
		 \Big]\Phi^{i,j}dx\,dr
		 \\&
		 \sum_{i,j}\int _s^t\int _{\T} \Big[\partial_x^{k-1}(\partial_x^2u_r + 
		 u_r|\partial_x u_r|^2 +u_r\times \partial_x^2 u_r 
		 )^i\partial_x^{k}u_r^j\Phi^{i,j}
		 \\&\quad \quad \quad 
		 -\partial_x^{k}u_r^i\partial_x^{k-1}(\partial_x^2u_r + 
		 u_r|\partial_x u_r|^2 +u_r\times \partial_x^2 u_r 
		 )^j\Phi^{i,j}
		\\&\quad \quad \quad 
		-\partial_x^{k-1}u_r^i\partial_x^{k-1}(\partial_x^2u_r + 
		u_r|\partial_x u_r|^2 +u_r\times \partial_x^2 u_r 
		)^j\partial_x\Phi^{i,j}
		\Big]dx\,dr\,.
	 \end{aligned}
		 \]
		 Swapping the indices in the sum $\sum_{i,j}$ then gives with 
		 $\Psi:=\Phi-\Phi^{\star}$
		 \[
		 \begin{aligned}
		 &\langle \D(k-1,k)_{s,t},\Phi\rangle 
		 \\
		 &=
		 \int_s^t \int _{\T}\Big[
		 \dd^{k+1}u^i_r\dd^{k}u^j_r\Psi^{i,j} 
		 + \partial_x ^{k-1}(u_r|\partial u_r|^2)\partial_x^ku_r^j\Psi^{i,j}
		 +\partial_x ^{k-1}(u_r\times\partial_x ^{2}u_r)^i\partial_x ^{k} 
		 u^j_r\Psi^{i,j} 
		 \\&\quad \quad \quad 
		 -\partial_x^{k-1}u_r^i\partial_x^{k-1}(\partial_x^2u_r + 
		 u_r|\partial_x u_r|^2 +u_r\times \partial_x^2 u_r 
		 )^j\partial_x\Phi^{i,j}
		 \Big]dx\,dr
		 \\&
		 =\mathrm I + \mathrm{II} + \mathrm{III} + \mathrm{IV} 
		 \,.
		 \end{aligned}
		 \]
		 For $\|\Phi\|_{H^1}\leq 1$, we evaluate
		 \[
		 |\mathrm{I}|
		 \lesssim \int_{s}^t\|\dd^{k+1}u_r\|_{L^2}\|\dd^k 
		 u_r\|_{L^2}dr
		 \lesssim (t-s)\|u\|_{L^\infty(s,t;H^k)}^2 + 
		 \|u\|_{L^2(s,t;H^{k+1})}^2 \,.
		 \]
		 For the second and third terms, we rely on the following interpolation 
		 inequality for $l=|\beta|+|\gamma|$  (see \cite{Melcher}) 
		 \begin{align}\label{eq:interp_iterative_1}
		 	\|\partial_\beta f\partial_\gamma g\|_{L^2} 
		 	\lesssim\|f\|_{L^\infty}\| g\|_{H^l}+\|g\|_{L^\infty}\| f\|_{H^l}\, .
		 \end{align}
		Applying \eqref{eq:interp_iterative_1} to  
		$f=u\times\cdot$ and $g=\partial_x u$, then using the continuous 
		embedding of $H^1$ into $L^\infty$ leads to
	\begin{equation}
	\label{eq:first_part_drift_k}
	\begin{aligned}
		|\mathrm{III}|
		&\lesssim \int_s^t\|\partial_x ^{k-1}(u_r\times\partial_x 
		u_r)\|_{L^2}\|\partial_x ^{k} u_r^j\|_{L^2}dr
		\\&
		\lesssim\int_s^t\Big[
		\|u_r\|_{L^\infty}\|\partial_x  
		u_r\|_{H^{k-1}}+
		\|\partial_x u_r\|_{L^\infty}\|u_r\|_{H^{k-1}}
		\Big] \|\partial_x ^{k} u_r\|_{L^2}dr
		\\&
		\lesssim\big[\|\partial_x  u\|_{L^\infty (H^1)}\|u\|_{L^\infty 
		(s,t;H^k)}+\|u\|_{L^\infty(H^1)}\|\partial_x  
		u\|_{L^2(s,t;H^k)}\big]\|\partial_x ^{k} u^j\|_{L^2(s,t;L^2)}
		\\&\quad  
		+\big[\|\partial_x  
		u\|_{L^\infty(H^1)}\|u\|_{L^\infty(s,t;H^k)}+\|u\|_{L^\infty(H^1)}\|\partial_x
		  u\|_{L^2(s,t;H^k)}\big]\|\partial_x ^{k-1} u^j\|_{L^\infty(s,t;H^1)}\, .
		\end{aligned}
	\end{equation}
	    Now we turn to the second term:
		applying \eqref{eq:interp_iterative_1} twice, first to $f=u \otimes 
		\partial_x u$ and $g=\cdot \,\partial_xu$, then to $f=u$ and $g=\otimes 
		\partial_x u$, yields:
	\begin{equation}
		\label{eq:drift_2_k_a}
		\begin{aligned}
		|\mathrm{II}|
		&\leq
		\int_s^t\|\partial_x ^{k-1} (u |\partial_x  
		u_r|^2)\|_{L^2}\|\partial_x^ku_r\|_{L^2}dr
		\\
		&\lesssim \int _s^t\big[\|u_r \otimes \partial_x u_r\|_{L^\infty}\| 
		\partial_xu_r\|_{H^{k-1}}+\|\partial_xu_r\|_{L^\infty}\| u_r \otimes 
		\partial_x 
		u_r\|_{H^{k-1}}\big]\|u_r\|_{H^k}dr
		\\&
		\lesssim \int_s^t\Big[\|u_r\|_{H^2}\|u_r\|_{H^k}^2 + 
		\|u_r\|_{H^2}^2\|u_r\|_{H^{k-1}}\|u_r\|_{H^{k-1}}\| u_r\|_{H^k} \Big]dr\, ,
		\end{aligned}	
	\end{equation}
since $\|u\|_{L^\infty}=1$. This shows the corresponding bound for 
$\mathrm{II}.$ 
Similar arguments are applicable for the fourth term, observing that the latter 
has lower order. We leave the details to the reader.
\end{proof}

	\begin{proposition}\label{pro:existence_remainder_k}
		In the notations of Section \ref{Section_remainder_existence}, the following estimate on 
		the remainder $(\partial_x^Nu^{\otimes2})^\natural$ holds
		\begin{align*}
	    \langle (\partial_x^N u^{\otimes 2})^\natural_{s,t},\mathbf 1\rangle
		&\lesssim
		\sum_{m=0}^{N}\sum_{k,\ell=0;(k,l)\neq (N,N)}^m\omega^{1/p}\omega_{\D (k,\ell)}(s,t)+\omega^{1/p}(s,t)\| u\|^2_{L^\infty(s,t;H^N)}\, ,
		\end{align*}
		where $\omega:=\omega_{\mathbf{A};H^{N+1}}$.
	\end{proposition}
	\begin{proof}
		The claim follows from Proposition \ref{pro:existence_drift_k}, from Corollary \ref{cor:existence_rem_k} and from the anti-symmetry of the noise. Indeed the term $\omega_{ \mathbf{A}; H^{N+1}}^{1/p}\omega_{\D (N,N)}(s,t)$ does not appear, as explained in Step 4 of Theorem \ref{th:remainder_interate}.
	\end{proof}
	\begin{theorem}\label{teo:higher_order_regularity}(Higher order regularity)
		Fix $k\geq 1$. Let $u^0\in H^k(\T ;\mathbb{S}^2)$ and $\G \in \RD ^p_a(H^{k+1})$, then the solution to \eqref{rLLG} belongs to $u\in L^\infty (H^k)\cap L^2(H^{k+1})\cap\mathcal{V}^p(H^{k-1})$.
	\end{theorem}
	\begin{proof}
		The statement holds for $k=1$ from Theorem \ref{thm:existence}. We address $k\geq 2$ by induction: we assume that if $u^0\in H^{k-1}(\mathbb{T},\mathbb{S}^2)$ then the solution lies in $L^\infty(H^{k-1})\cap L^2(H^k)\cap \mathcal{V}^p(H^{k-2})$. 
		Now, let $u^0\in H^k(\mathbb{T},\mathbb{S}^2)$. To show that $u\in L^\infty(H^k)\cap 
		L^2(H^{k+1})\cap \mathcal{V}^p(H^{k-1})$, we consider the equation for $\delta(\dd^k u \otimes \dd^k 
		u)_{s,t}$ tested by $\mathbf 1$. 
		\smallskip
		
		\noindent \textbf{Claim.} 
			The integral $\langle\D(k,k),\mathbf 1\rangle$ satisfies the estimate
			\begin{equation}
				\label{content:claim}
				\langle \D(k,k),\mathbf 1\rangle 
				\leq - \int_s^t \|\dd^{k+1}u_r\|_{L^2}^2dr + C
				\|u\|_{L^\infty(H^{k-1})\cap 
					L^2(H^k)}(t-s)\|u\|_{L^\infty(s,t;H^k)}^2\,.
			\end{equation}
		
		\noindent Towards \eqref{content:claim} we write
		\begin{equation}
			\label{decomp:D_kk}
		\begin{aligned}
			\langle \D(k,k),\mathbf 1\rangle 
			&=2\sum _i\int_s^t\int_{\T}\partial_x^k(\partial_x^2u_r + u_r|\partial_x 
			u_r|^2 + u_r\times \partial_x^2u_r)^i\partial_x ^ku_r^i dx\, dr
			\\&
			=-2\int_s^t\|\partial_x^{k+1} u\|_{L^2}^2 dr + 
			2\int_s^t\int_{\T}\dd^{k}(u|\partial u|^2)\cdot\dd^kudx\,dr
			+2\int_s^t\int _{\T}\dd^k(u\times \dd^2 u)\cdot 
			\partial_x^k u
			dx\,dr\,.
		\end{aligned}
		\end{equation}
		We evaluate
		\[
		\begin{aligned}
			&2\int_s^t\int_{\T}\dd^{k}(u|\partial u|^2)\cdot\dd^kudx\,dr  
			\\
			&=
			-2\int_s^t\int_{\T}\dd^{k-1}(u|\partial u|^2)\cdot \dd^{k+1} u dx\,dr
			\\&
			\lesssim
			\int_s^t\Big[\|u\|_{H^2}\|u\|_{H^{k}} +\|\partial 
			u\|_{L^\infty}^2\|u\|_{H^{k-1}}\Big]\|\dd^{k+1}u\| _{L^2}dr
			\\&
			\leq
			\frac12 \int_s^t\|\dd^{k+1}u\|_{L^2}^2dr + 
			C\|u\|_{L^\infty(s,t;H^k)}^2\int_s^t\|u_r\|_{H^2}^2dr 
			+\int _{s}^t\|\partial_x u\|_{L^2}\|\partial_x u 
			\|_{H^1}\|u\|_{H^{k-1}}\|\dd^{k+1}u\|_{L^2}
		\end{aligned}
		\]
		for some universal constant $C>0$,
		where we have used \eqref{eq:interp_iterative_1} twice between the second 
		and third lines, and the interpolation inequality \eqref{eq:interpolation_1D} to obtain the 
		last term in the fourth line. 
		Distinguishing the cases when $k=2$ and $k>2$, we observe that the latter can be estimated 
		above by
		\[
		C\sup_{r\in [0,T]}\|u_r\|_{H^{k-1}}^2\int _s^t\|u_r\|_{H^k}\|\dd^{k+1}u_r\|_{L^2} dr\, ,
		\]
		which in turn is bounded thanks to Young Inequality by 
		\[
		\frac{C}{2}\sup_{r\in [0,T]}\|u_r\|_{H^{k-1}}^4(t-s)\sup_{r\in[s,t]}\|u_r\|_{H^{k}}^2 + 
		\frac12\int_s^t\|\dd^{k+1}u_r\|_{L^2}^2 dr \,.
		\]
		
		We now proceed to the evaluation of the third term in \eqref{decomp:D_kk} 
		By Proposition 
		\ref{pro:existence_drift_k} and Young's inequality, we estimate
		\begin{align*}
			\int_{s}^{t}\int_{\T } \partial_x^k (u_r\times\partial_x^2 u_r)\cdot \dd^k u_r dx dr
			&=-\int_{s}^{t}\int_{\T } \partial_x^{k-1} (u_r\times\partial_x^2 u_r)\cdot \dd^{k+1} 
			u_r dx \, dr
			\\
			&=-\sum_{j=1}^{k-1}\binom{k-1}{j}\int_{s}^{t}\int_{\T } (\partial_x ^{k-1-j} 
			u_r\times\partial_x ^{j}\partial_x^2 u_r)\cdot \dd^{k+1} u_r dx\, dr
		\end{align*}
		and in particular $ (u_r\times\partial_x^{k+1} u_r)\cdot \dd^{k+1} u_r =0$, therefore from 
		the embedding of $H^1(\T)$ into $L^\infty(\T)$,
		\begin{align}\label{eq:a}
		\|\partial_x^{k-1} (u_r\times\partial_x^2 u_r)\cdot \dd^{k+1} u_r \|_{L^1(s,t;L^1)}\leq 
		\frac{1}{2}\|\dd^{k+1} u_r 
		\|^2_{L^2(s,t;L^2)}+C_k\|\partial_x^{k-1}u\|^2_{L^2(L^\infty)}\|u\|_{L^\infty(s,t;H^{k})}^2
		 \, ,
		\end{align}
		where $C_k>0$ is a constant. Note then that if $u\in L^\infty(H^{k-1})\cap L^2(H^k)$, 
		then from the the embedding of $H^1(\T)$ into $L^\infty(\T)$,
		\begin{align}\label{eq:consider_L_infty}
		\int_{s}^{t}	\|\partial_x^{k-1} u_r\|^2_{L^\infty} dr\leq \|u\|_{L^\infty(s,t; H^{k-1})}\|u\|_{L^1(s,t;H^k)}\leq \|u\|_{L^\infty(s,t; H^{k-1})}\|u\|_{L^2(s,t;H^k)}(t-s)^{1/2}\, ,
		\end{align}
		that gives the time regularity.  This finishes to establish \eqref{content:claim}, hence 
		our claim.\smallskip

Now, testing the equation on $(\dd^ku^{\otimes 2})$ against $\mathbf 1$, then using \eqref{content:claim}
 and our induction hypothesis, we obtain the inequality
		\begin{align*}
		\sup\limits_{r\in[s,t]}\|\partial_x^k u_t\|^2_{L^2}+\int_{s}^{t}\|\dd^{k+1} 
		u_r\|_{L^2}^2 d 
		r&\lesssim_{\|u^0\|_{H^{k}}}[\omega_{G,H^{k+1}}^{1/p}(s,t)+(t-s)^{1/2}]\|u\|^2_{L^\infty(s,t;H^k)}+\langle
		 (\partial_x^N u^{\otimes 2})^\natural_{s,t},\mathbf 1\rangle\, .
		\end{align*}
		 By combining Proposition \ref{pro:existence_drift_k} and Proposition \ref{pro:existence_remainder_k} and techniques similar to Proposition \ref{pro:energy} (Young's inequality, sewing Lemma \ref{lemma_sewing} and Gronwall's Lemma \ref{lem:gronwall}), the following inequality holds
		\begin{align}\label{eq:a_priori_k}
		\sup\limits_{t\in[0,T]}\|\partial_x^k u_t\|^2_{L^2}+\int_{0}^{T}\|\dd^{k+1} u_r\|_{L^2}^2 d 
		r\lesssim\exp \left(\omega(0,T)\right)[\|\dd^k u_0\|^2_{L^2}+ \omega_{\G 
		;H^{k+1}}^{1/p}(0,T)\|u_0\|_{H^k}^2]\, ,
		\end{align}
		where $\omega(s,t):=(t-s)+\omega_{\G ;H^{k+1}}(s,t)$. From the a priori estimate \eqref{eq:a_priori_k},
		 we deduce that the sequence of the approximants $(u^n)_n$ (see Proposition \ref{pro:B1})  is uniformly bounded in $L^\infty(H^k)\cap L^2(H^{k+1})\cap\mathcal{V}^p(H^{k-1})$. The bound in $\mathcal{V}^p(H^{k-1})$ follows from argument similar to the proof of Corollary \ref{L2H2}. The same compactness argument as in Section \ref{sub_section:proof_existence} applies. From \eqref{eq:a_priori_k} and the lower semi-continuity of the norm also the limiting solution $u$ satisfies the inequality \eqref{eq:a_priori_k} and thus the claim follows. 
	\end{proof}

	\subsection{Solution to the LLG as  a stochastic process}
Let $(\Omega,\mathcal{G}, (\mathcal{G}_t)_t,\mathbb{P})$ be a filtered probability space. We are interested in showing that the solution $u$ viewed as 
\begin{equation}\label{eq:u_RV}
	u:(\Omega,\mathcal{G},\mathbb{P})\rightarrow L^\infty(H^1)\cap L^2(H^2)\cap \mathcal{V}^p(L^2)
\end{equation}
is a random variable and it is adapted to the natural filtration. 
Consider a random variable $w$ on the probability space $(\Omega,\mathcal{G},\mathbb{P})$ with values in $ \mathcal{V}^p( L^2)$ and let $\tilde{\Omega}\subset \Omega$ be the set of full measure where the path $w(\omega)$ can be lifted to a rough driver $(G(\omega),\GG (\omega))$ for all $\omega\in \tilde{\Omega}$.
Consider also a random variable $u^0$ defined on $(\Omega,\mathcal{G} ,\mathbb{P})$ and so that $u^0(\omega)\in H^1(\T ;\mathbb{S}^2)$ for all $\omega\in \tilde{\Omega} $.

Let us fix $\omega\in \tilde{\Omega} $. 
In the previous sections we constructed a sequence $(u^n(\omega))_n$, where each $u^n(\omega)$ is a solution to \eqref{rLLGn} with initial condition $u^n(\omega)$ driven by $(G^n(\omega),\GG ^n(\omega))$.  From the compactness argument, there exists a subsequence depending on $\omega$ so that $(u^{n_{j_\omega}}(\omega))_{j_\omega}$ converges in $L^2(H^1)\cap C([0,T];L^2)$ to a limit $u(\omega)\in L^\infty(H^1)\cap L^2(H^2)\cap \mathcal{V}^p(L^2)$. From uniqueness in Theorem \ref{thm:existence} $u(\omega)$ is the only solution to \eqref{rLLG}. 
Since $u(\omega)$ is unique then every converging subsequence of $(u^n(\omega))_n$ converges to $u(\omega)$ and therefore the whole sequence converges to $u(\omega)$ for all $\omega\in \tilde{\Omega}$. Moreover, from the continuity of the It\^o-Lyons map, the convergence happens in $L^\infty(H^1)\cap L^2(H^2)\cap\mathcal{V}^p(L^2)$ for a.e.\ $\omega\in \tilde{\Omega}$.  
In Theorem \ref{th:cont_wz_2}, we have proved that there exists a continuous map independent on $\omega$ such that $u=\Psi^1(u^0,\G )$, hence the solution is a continuous function of the path $w$ and thus it is adapted with respect to the natural filtration, thus with respect to $(\mathcal{G}_t)_t$ as well.

To complete the proof of Theorem \ref{thm:existence}, we need the following result.

\begin{proposition}[consistency]
	\label{pro:consistency}
	Fix $k\ge1$, $u^0\in H^k(\T;\mathbb S^2)$ 
	and let $\G=\Lambda_{g(\cdot)}\mathcal F(\rp)$ be as in \eqref{simple_RD} with $H=\frac12$ and \( g\in H^{k+1} .\)
	Then, the stochastic process $u=\pi(u^0,\G)\colon \Omega \times[0,T]\to H^k(\T)$ is \( (\mathcal G_t)_t \)-adapted and has its trajectories supported in $\mathcal X^k:=L^\infty(H^k)\cap L^2(H^{k+1})\cap \V^p(H^{k-1})$.
	Moreover, the following assertions hold.
	\begin{enumerate}[label=(C\arabic*)]
	  \item \label{C1} Pathwise solutions have finite moments of all order, in \( \mathcal X^k \);
	  
	  \item \label{C2} Pathwise solutions are semi-martingales in \( L^2(\T) \); moreover they are solutions in the usual Stratonovich sense, i.e.\ for every $t\in[0,T]$ and any $\varphi \in C^\infty(\T;\R^3)$
	\begin{multline}
		\label{weak_Cc}
		\int_{\T}(u _t(x)-u _0(x))\cdot \varphi (x)dx + \int_0^t\int_{\T} 
		(\mathbf 1_{3\times3}-u_t(x)\times)\partial_x^2 u _s(x)\cdot \varphi(x) dx dr
		\\
		=\int_0^t\int_{\T}u _r|\partial_x u_r|^2(x)\cdot\varphi (x)dxdr +  \int_{0}^t \int_{\T} 
		g(x)(u _r(x)\times \circ d\beta_r) \cdot \varphi (x)dx\,,\quad \mathbb P\text{-a.s.}
	\end{multline}
		
		\item \label{C3} If \( k\ge2 \), then any Stratonovich solution to \eqref{weak_Cc} is also a pathwise solution. 
	\end{enumerate}
\end{proposition}

\begin{proof}
		\textit{Proof of \ref{C1}.}
	The moment estimates are a consequence of the estimate obtained in Theorem 
	\ref{teo:higher_order_regularity}
	Indeed, \eqref{eq:a_priori_k} shows that
	\[
	\esssup_{t\in[0,T]}\|u_t(\cdot)\|_{H^k}^2
	\lesssim K \|u_0\|_{H^k}
	\]
	for some random variable $K\lesssim 1+\exp(\|\G\|_{\mathcal{RD}^p_a(H^{k+1})})$. If we can show that 
	$\|\G\|_{\mathcal{RD}^p_a(H^{k+1})}$ has exponential tail, the result will follow.
	From our definition of $\G$ and the rough driver metric,  we have (using that $\mathcal F\colon 
	\R^3\to \mathcal L_a(\R^3)$ is continuous):
	\[
	\|\G\|_{\mathcal{RD}^p_a(H^{k+1})}\lesssim_{\|g\|_{H^{k+1}}} \|\mathcal 
	F(\rp)\|_{\mathcal{RD}^p_a(\R^3)}\lesssim
	\|\rp\|_{\mathscr C^{1/p}_g(\R^3)}\, ,
	\]
	where \( \mathscr C_g^{1/p} \)-denotes the usual space geometric, 2-step rough paths which are \( 1/p \)-H\"older regular (see \cite{FrizHairer} for a definition), and we recall that $\rp$ is the Stratonovich enhancement of the Brownian 
	motion $\beta\colon \Omega\times [0,T]\to \R^3$.
	It follows from \cite[Theorem A.12]{FrizVictoire} that $\mathbb 
	E[\exp(m\|\G\|_{\RD^p})]\lesssim 
	\mathbb E[\exp(m\|\rp\|_{\mathscr C^{1/p}})]<\infty$ for every $m>0$, hence our 
	conclusion.

	\item[\textit{\indent Proof of \ref{C2}}.]
	Introduce the pair of stochastic processes \( (Y,Y')\colon \Omega\times[0,T]\to L^2(\T;\mathcal L_a(\R^3))\times L^2(\T;\mathcal L_a(\R^{3\times3})) \) such that for each \( a,b\in \R^3 \) and \( t\in [0,T] \)
	\[
	Y_t(\omega;x)[a]:=g(x)u_t(\omega;x)\times a,\qquad
	Y'_t(\omega;x)[a\otimes b]:=g(x)^2 (u_t(\omega;x)\times a)\times b,
	\quad \mathbb P\otimes dx\text{-almost everywhere.}
	\]
	Since we already established that \( u_t \) is \( \mathcal F_t\)-adapted, it remains to show that the Stratonovich integral, when integrated against any test function $\varphi\in C^{\infty}(\T)$ equals the pathwise rough integral.
	In fact, we are going to show the slightly stronger statement that the two integrals are equal as elements of the functional space $L^2(\T;\R^3)$. 
	For that purpose, we need to show that on a set of full probability measure $\Omega_0\subset \Omega$,
	in the notations of Remark \ref{rem:global},
	\begin{equation}
		\label{rough_equals_strato}
	\int_0^t d\G_r u_r = \int_0^t gu\times \circ d\beta_t\,.
	\end{equation}
	To show \eqref{rough_equals_strato}, we first observe thanks to \eqref{murky_formula} that \( \mathbb P \)-a.s.,
	\begin{align}
	\int_0^t d\G_r u_r
	&=\lim_{\pi}\sum_{[t_i,t_{i+1}]\in \pi}(G_{t_i,t_{i+1}}u_{t_i}+\GG_{t_i,t_{i+1}}u_{t_i})
	\nonumber\\
	&=\lim_{\pi}\sum_{[t_i,t_{i+1}]\in \pi}(Y_{t_i}[\delta\beta_{t_i,t_{i+1}}]+Y'_{t_i}[\rpp^{\mathrm{Ito}}_{t_i,t_{i+1}}]) + \frac12 \int _0^tY'_r[\mathbf{1}_{3\times 3}] dr
	\label{integral_corrected}
\end{align}
	where the limit is taken along any sequence of partitions \( (\pi_n) \) of \( [0,t] \) such that $|\pi_n|\equiv\max(t_{i+1}-t_i)\to 0$, and we set \( \rpp^{\mathrm{Ito}}_{s,t}:=\rpp_{s,t}-\frac12(t-s)\mathbf{1}_{3\times 3} .\)
	It is immediate to verify that the last term in \eqref{integral_corrected} corresponds to the usual Stratonovich/It\^o corrections, namely
	\[
	\frac12 \int _0^tY'_r[\mathbf{1}_{3\times 3}] dr = \frac12\sum_{i=1}^3\int_0^t g^2(u_r\times e_i)\times e_idr
	\]
	(as Bochner integrals in \( L^2(\T;\R^3) \))
	where \( (e_1,e_2,e_3) \) is the canonical basis of \( \R^3 \).
	On the other hand, the right hand side in \eqref{rough_equals_strato} is equal to \( \lim_{\pi}\sum_{[t_i,t_{i+1}]\in \pi}Y_{t_i}\delta\beta_{t_i,t_{i+1}}\), by definition of an It\^o integral.
	In particular, the conclusion follows from the estimate
	\begin{equation}
		\label{ito_vanish}
		\lim_{\pi} \mathbb E[\|\sum_{[t_i,t_{i+1}]\in \pi}Y'_{t_i}\rpp^{\mathrm{Ito}}_{t_i,t_{i+1}}\|_{L^2}^2] \lesssim |\pi|\equiv\max_{t_i\in \pi}(t_{i+1}-t_i)\,,
	\end{equation}
 the proof of which is identical to that of \cite[eq.\ (5.1) p.~68]{FrizHairer} and therefore omitted. Hence the identity \eqref{rough_equals_strato}, and therefore \eqref{weak_Cc} is proved.

	\item[\textit{\indent Proof of \ref{C3}}.] 
		Since \( k\ge2>\frac{11}{6} \), the proof of \ref{C3} is carried out by the same arguments as in \cite[Theorem 2.19]{gerasimovics2020non}, 
	the only difference being that the equation considered there is scalar. 
	Observing that the linear map $a_t(x)=(\mathbf{1}_{3\times 3}-u_t(x)\times)\in \mathcal L(\R^3)$ is uniformly elliptic, it is easily seen that the estimates used inside the proof are unchanged. Therefore, we only need to check the condition
	\begin{equation}
		\mathbb E[\esssup_{t\in [0,T]}|a_t(\cdot)|_{C^{k-1}}^m]< \infty\,,
	\end{equation}
	for every $m\in[1,\infty)$, which is a consequence of \ref{C1} and the embedding $H^k\hookrightarrow C^{k-1}$.	
\end{proof}

	\section{Continuity of the It\^o-Lyons map}
	\label{sec:wong-zakai}
	We are going to state that the map that associates a realization of the noise with the solution to \eqref{rLLG} is continuous, i.e. 
	\begin{equation}\label{intro:wz_map}
	\mathbf{A}\longrightarrow \pi(a_0;\mathbf{A})\, ,
	\end{equation}
	where by $\pi(a_0;\mathbf{A})$ we mean the solution to \eqref{LLG_apriori} with initial condition $a_0$ and driven by the rough driver $\mathbf{A}$. 
	Let us stress that the continuity of the map \eqref{intro:wz_map} is a peculiarity of the rough path theory and it is not true in the classical It\^o calculus. 
	
	The first direct consequence of the continuity of \eqref{intro:wz_map}, is the so-called Wong-Zakai convergence: if $(\G ^n)_n$ is an approximations to the noise $\G $, then the sequence $(\pi(a_0;\G ^n))_n$ converges to the solution $\pi(a_0;\G )$. This convergence occurs in a stronger topology compared to the convergence obtained through the compactness method  in the existence proof.  Indeed, the compactness in the existence proof  implies that the approximants converge to the solution in the $L^2(H^1)\cap C([0,T];L^2)$-norm, whereas the following procedure ensures that the convergence occurs in $L^\infty(H^1)\cap L^2(H^2)\cap\mathcal{V}^p(L^2)$.
	
	A further direct application is the $\omega$-wise convergence  to the deterministic solution for vanishing noise.
	In the first subsection we investigate an iterative estimate for the remainder: the proof  is quite technical and we state it in a more general framework. In the second section we state the continuity of the It\^o-Lyons map.
		We introduce some useful identities
		\begin{align}\label{eq:identities_convergence}
			(Aa-Bb)^i(c-d)^j&=((A-B)a(c-d)^j+B(a-b)(c-d)^j)^i \, ,
			\\
			(Aa-Bb)^i(Cc-Dd)^j&=((A-B)a)^i((C-D)c)^j+(A-B)a)^i(D(c-d))^j
			\\
			\nonumber
			&\quad+(B(a-b))^i((C-D)c)^j+(B(a-b))^i(D(c-d))^j \, .
		\end{align}
	
	\subsection{Bilinear estimates}
Let $\A,\B\in \RD^{p}_a(H^\sigma)$ and assume that one has $a,b:[0,T]\to L^2(\T;\R^n)$ such that
\begin{equation}\label{eq:a_b}
\begin{aligned}
da = fdt + d\A a
\\
db = gdt + d\B b
\end{aligned}
\end{equation}
for some $f,g\in L^2(L^2).$
It is worth noticing that the difference $(A-B,\AA-\BB)$ is not a rough driver in general,
which is a relic of the fact that the space of rough paths of finite $p$-variation with $p>2$ is 
not linear (this fact is underlined for instance in 
\cite{FrizVictoire}). Consequently, we cannot apply the estimates of Section \ref{sec:apriori} 
to estimate remainders of the form $(\partial_x ^N (a-b)^i \partial_x ^N (a-b))^{\natural}$.
	We will nonetheless follow a similar procedure, but this time different factors than in Theorem \ref{th:remainder_interate} appear.

	In Theorem \ref{th:remainder_interate} the computations are performed at the level of the approximations of the solutions and the aim is to use that result to prove an uniform bound. 
	Differently the aim of Theorem \ref{th:wz_remainder} is to write an estimate of the remainder in such a way that, under the knowledge of some weaker convergences, we get a stronger convergence. 
	For this reason we need to highlight some `good terms' that will make this convergence happen under our hypothesis, i.e.\ terms of the form 
	\begin{align}\label{eq:good_elements}
	\partial_x ^{n} (A-B), \quad \partial_x ^{m}(\AA -\mathbb{B}), \quad \delta (\partial_x ^{n} 
	(a-b) \partial_x ^{m} (a-b))\, ,
	\end{align}
	for $n\in \{1,\dots,N\}$, $m\in \{1,\dots,M\}$. 
		
	In this section we assume that there are $a,b\in L^\infty (H^k)\cap L^2(H^{k+1})$ as in \eqref{eq:a_b} and from the previous Theorem \ref{th:remainder_interate} we know how to estimate the corresponding remainders for all the derivatives.	
From the product formula, Proposition \ref{pro:existence_remainder_k}, we have for $z=a-b$
	\begin{equation}\label{eq:general_eq_der_wz}
	\begin{aligned}
	\delta z^{\otimes 2}_{s,t}
	&=2\int_s^tz\odot(f-g)dr 
	+ (I+\II+\tilde\II)(N,N)+(z^{\otimes 2})^{\natural}\, ,
	\end{aligned}
	\end{equation}
	where, letting $(Z,\ZZ)=(A-B,\AA-\BB)$
	\[
	\begin{aligned}
	I_{s,t}
	&=(A_{s,t}a_s - B_{s,t}b_s)\otimes z + z\otimes (A_{s,t}a_s - B_{s,t}b_s)
	\\
	&=(A_{s,t}z_s + Z_{s,t}a_s)\otimes z_s + z_s\otimes (A_{s,t}z_s + Z_{s,t}a_s)\, ,
	\\[0.3em]
	\II_{s,t}
	&=(\AA_{s,t}a_s - \BB_{s,t}b_s)\otimes z + z\otimes (\AA_{s,t}a_s - \BB_{s,t}b_s)
	\\
	&=(\AA_{s,t}z_s + \ZZ_{s,t}a_s)\otimes z_s + z_s\otimes (\AA_{s,t}z_s + \ZZ_{s,t}a_s)\, ,
	\\[0.3em]
	\tilde\II_{s,t}
	&=(A_{s,t}a_s - B_{s,t}b_s)^{\otimes2}
	\\
	&=(A_{s,t}z_s + Z_{s,t}a_s)^{\otimes2}\, .
	\end{aligned}
	\]
	
	\begin{theorem}\label{th:wz_remainder}
		Fix $k\geq 1$, consider $a_0,b_0\in H^{k}(\T ;\mathbb{S}^2)$ and $\mathbf{A}, \mathbf{B}\in \RD ^p_a(H^{k+1})$. Let $a=\pi(a_0;\mathbf{A})$ and $b=\pi(b_0;\mathbf{B})$ be the solutions to \eqref{rLLG}. Denote $z:=a-b$, $Z:=A-B$, $\mathbb{Z}:=\AA -\mathbb{B}$ and consider the following controls,
		\begin{align*}
		&\omega_{\D (i,j)}(s,t):=\int_s^t \int_\T \big( |\partial_x ^i (f_r -g_r)\partial_x ^j  
		z_r|+|\partial_x ^j (f_r -g_r)\partial_x ^i  z_r| \big)  d x \, d r\, , \\
		&\omega_{ \D }(s,t):=\sum_{i,j\in \{0,1,\dots,k\}} \omega_{ \D (i,j)}(s,t)\, .
		\end{align*}
		Then, for all $i,j\in\{0,1,\dots,k\}$, the following estimate holds locally, i.e. consider a partition $(P_k)_k$ of $[0,T]$, then for every $k$ and for all $s\leq t\in P_k$
		\begin{align}\label{eq:remainder_estimate_k}
		\omega_{ (\partial_x ^kz)^{2,\natural};H^{-1}}^{3/p}(s,t)&\lesssim \; \omega^{2/p}(s,t) 
		\omega_{ \mathbf{Z};H^{k+1}}^{1/p}(s,t)+\omega^{1/p}\omega_{ \D }(s,t) 
		+\omega^{3/p}(s,t)\|z\|_{L^\infty (H^k)}^2\, ,
		\end{align}
		where $\omega:=\omega_{\mathbf{A};H^{k+1}}+\omega_{\mathbf{B};H^{k+1}}$. 
	\end{theorem}

	\begin{proof} 
		We proceed as in Theorem \ref{th:remainder_interate}: the main ingredients are Theorem \ref{teo:rem_derivatives} and the steps \ref{RM1}--\ref{RM3} described in Section \ref{Section_remainder_existence}. We show first how the remainder can be estimated in the case $N=0$. Second, we show the case $N=1$: in this way we see how the derivation impacts the estimate. In conclusion we use the case $N=k-1$ as inductive hypothesis and we conclude the estimate up to $N=k$ for induction. 	We denote by $\omega_{A}:=\omega_{A;H^1}$, $\omega_{\AA }:=\omega_{\AA ;H^1}$, 
		$\omega_{\mathbf{A}}:=\omega_{\mathbf{A};H^1}$.
		We adopt the same convention for the controls of $\mathbf{B}$ and $\mathbf{Z}$. 
		
		\item{\textbf{\indent Step 1: N=0.}}  As for the proof of Theorem \ref{th:remainder_interate}, we hinge on the rough standard machinery and we start with the first step \ref{RM1}. 
		We start by considering $\delta( z^iz^j)_{s,t}$ in the form \eqref{eq:general_eq_der_wz} and we estimate
		$$z^{2,\natural,i,j}_{s,t}:=\delta (z^iz^j)_{s,t}-2\int_s^tz\odot(f-g)^{i,j}dr-(I+\II+\tilde\II)_{s,t}^{i,j}(0,0)\, .$$
		We note that
		(as a consequence of Theorem \ref{th:remainder_interate})
		\begin{equation}
			\label{delta_I}
			\begin{aligned}
		\textcolor{black}{ - }\delta I_{s,u,t}
		&=\sum_{l} Z^{i,l}_{u,t}[\delta (a^lz^j)_{s,u}-T_{a^lz^j}]-B^{i,l}_{u,t}[\delta (z^lz^j)_{s,u}-T_{z^lz^j}]\\
		&\quad 
		+\sum_{l} Z^{j,l}_{u,t}[\delta (a^lz^i)_{s,u}-T_{a^lz^i}]-B^{j,l}_{u,t}[\delta (z^lz^i)_{s,u}-T_{z^lz^i}]\, ,
		\\
		-\delta\II^{i,j}_{s,u,t}
		&=\sum_{l} [\mathbb{Z}^{i,l}_{u,t}\delta (a^lz^j)_{s,u}-\mathbb{B}^{i,l}_{u,t}\delta (z^lz^j)_{s,u}]
		+\sum_{l} [\mathbb{Z}^{j,l}_{u,t}\delta (a^lz^i)_{s,u}-\mathbb{B}^{j,l}_{u,t}\delta (z^lz^i)_{s,u}]\, ,
		\end{aligned}
		\end{equation}
		while
		\begin{align*}
		&-\delta \tilde\II_{s,u,t}^{i,j}
		=\sum_{l,m} Z^{i,l}_{u,t}[Z^{j,m}_{u,t}\delta(a^lb^m)_{s,u}-A^{j,m}_{u,t}\delta(a^lz^m)_{s,u}]+\sum_{l,m}B^{i,l}_{u,t}[Z^{j,m}_{u,t}\delta(z^lb^m)_{s,u}-A^{j,m}_{u,t}\delta(z^lz^m)_{s,u}]\, .
		\end{align*}	
Substituting $(\delta z^iz^j)_{s,t}$ with \eqref{eq:general_eq_der_wz}, for all test functions $\Phi\in H^1$ so that $\|\Phi\|_{H^1}\leq 1$ we can estimate these terms as
		\begin{align*}
		\quad&[\omega_{\delta aa}^{1/p}+\omega_{\delta a b}^{1/p}]\omega_{\mathbb{Z}}^{2/p}\|\Phi\|_{H^1} +\omega_{\mathbb{B}}^{2/p}\big[2\omega_{Z}^{1/p}+2\omega_{\mathbb{Z}}^{1/p}+\omega_B^{1/p}\|z\|^2_{L^\infty (L^2)}+\omega_B^{1/p}\|z_s\|^2_{L^\infty (L^2)} \big]\|\Phi\|_{H^1}\\
		&+ C\omega_{\mathbb{B}}^{2/p} \big[\omega_{Z}^{1/p}(\omega_A^{1/p}+\omega_B^{1/p})+  \omega^{2/p}_{B}(\|z\|_{L^\infty(L^2)}^2+\|z\|_{L^\infty(L^2)}^2)\big]\|\Phi\|_{H^1}\\
		&+C\omega_\mathbb{Z}^{1/p}[\omega_{a^{2,\natural};H^{-1}}^{3/p}+\omega_{(ab)^{\natural};H^{-1}}^{3/p}+\omega_{b^{\natural,2};H^{-1}}^{3/p}]\|\Phi\|_{H^1}+\omega_{\mathbb{B}}^{1/p}\omega_{z^{2,\natural};H^{-1}}^{3/p}\|\Phi\|_{H^1}\, ,
		\end{align*}
		where we used that $\|a\|_{L^\infty}=\|b\|_{L^\infty}=1$.
		The control $\omega_{\delta aa}^{1/p}$ is bounded since we assume $a$ is a solution (it is a consequence of the estimate of $a^{2,\natural}$).  The same conclusion can be achieved for $\omega_{\delta ab}^{1/p}$: this is also bounded from similar steps and the knowledge that $a,b$ are solutions. 
		The mixed term $\omega_{(ab)^{\natural};H^{-1}}^{3/p}$ is bounded for similar reasons. 
		
		In conclusion, by collecting all the previous contributions,
		\begin{align*}
		\|\delta z^{\natural,2}_{s,u,t}\|_{H^{-1}}
		&\leq\omega^{1/p}\omega_{ \D (0,0)}+[\omega_{\mathbb{Z}}^{2/p} +\omega_{Z}^{1/p} ]\omega^{2/p}+\omega^{3/p}\|z\|_{L^\infty (L^2)}^2\, ,
		\end{align*}
		and we can apply the sewing Lemma \ref{lemma_sewing} and conclude the estimate of the remainder. 
		
		\item{\textbf{\indent Step 2: N=1.}} Now we want to iterate the procedure and we consider also the case $N=1$, to see how the taking the derivative of the equation  impacts on the estimate. As a consequence of  Theorem \ref{th:remainder_interate} and with similar computations as above
		\begin{align*}
		\| \delta (\partial_x z)^{\natural,2,i,j}_{s,u,t}\|_{H^{-1}} &\lesssim\;  \omega^{2/p} \big[\omega_{Z}^{1/p}+\omega_{\partial_x Z}^{1/p}+\omega_{\mathbb{Z}}^{1/p}+\omega_{\partial_x \mathbb{Z}}^{1/p} \big]+\omega^{3/p}\big[\|\partial_x z\|_{L^\infty(L^2)}^2+\|z\|_{L^\infty(L^2)}^2\big]\\
		&\quad+\omega^{1/p}\big[\omega_{\D (0,0)}+\omega_{\D (0,1)}+\omega_{\D (1,0)}+\omega_{\D (1,1)}\big]\\
		&\quad+\omega^{1/p}\big[\omega^{3/p}_{(\partial_x z)^{\natural,2};H^{-1}}+\omega^{3/p}_{(z\partial_x z)^{\natural};H^{-1}}+\omega^{3/p}_{z^{\natural,2};H^{-1}}\big]\, .
		\end{align*}
		Hence step \ref{RM1} of the rough standard machinery is complete. Note that the  amount $\omega^{3/p}_{(z\partial_x z)^{\natural};H^{-1}}$ can be estimated by
		\begin{align*}
		\omega^{3/p}_{(z\partial_x z)^{\natural};H^{-1}} \lesssim \omega^{3/p}_{z^{2,\natural};H^{-1}}+\|z\|_{L^\infty(L^2)}^2+\|\partial_x z\|_{L^\infty(L^2)}^2+\omega_{\D (0,1)}\, ,
		\end{align*}
		and the amount $\omega^{3/p}_{z^{2,\natural};H^{-1}}$ can be estimated by the step before.
		We can then apply the sewing Lemma \ref{lemma_sewing} and show that
		\begin{align*}
		\| (\partial_x z)^{\natural,2}_{s,t}\|_{H^{-1}} &\lesssim \; \omega^{2/p} [\omega_{Z}^{1/p}+\omega_{\partial_x Z}^{1/p}+\omega_{\mathbb{Z}}^{1/p}+\omega_{\partial_x \mathbb{Z}}^{1/p} ]+\omega^{3/p}[\|\partial_x z\|^2_{L^\infty(L^2)}+\|z\|^2_{L^\infty(L^2)}]\\
		&\quad+\omega^{1/p}[\omega_{\D (0,0)}+\omega_{\D (0,1)}+\omega_{\D (1,0)}+\omega_{\D (1,1)}+\omega^{3/p}_{(\partial_x z)^{\natural,2};H^{-1}}]\,.
		\end{align*}
		This concludes \ref{RM2}. By implementing in a standard way \ref{RM3}, we conclude that
		\begin{align*}
		\| (\partial_x z^{\otimes 2})^{\natural}\|_{H^{-1}}&\lesssim \omega^{2/p} [\omega_{Z}^{1/p}+\omega_{\partial_x Z}^{1/p}+\omega_{\mathbb{Z}}^{1/p}+\omega_{\partial_x \mathbb{Z}}^{1/p} ]+\omega^{3/p}[\|\partial_x z\|^2_{L^\infty(L^2)}+\|z\|^2_{L^\infty(L^2)}]\\
		&\quad+\omega^{1/p}[\omega_{\D (0,0)}+\omega_{\D (0,1)}+\omega_{\D (1,0)}+\omega_{\D (1,1)}]\, .
		\end{align*}
		\item{\textbf{\indent Step 3:}} It can be proved by induction that \eqref{eq:remainder_estimate_k} holds, by following the same procedure as in the last step of the proof of Theorem \ref{th:remainder_interate}.
	\end{proof}

	\subsection{The Wong-Zakai convergence }
	\paragraph{The first order Wong-Zakai convergence}
	The aim of this paragraph is to prove that the solution is a continuous map of  the noise and of the initial condition with respect to the $L^\infty(L^2)\cap L^2(H^1)$ topology, as explained in the following statement
	\begin{proposition}\label{prop:cont_wz_1}
		(Continuity of the It\^o-Lyons map) : The map 
		\begin{align*}
		\pi:H^1(\T ;\mathbb{S}^2)\times\RD ^p_a(H^2)  &\longrightarrow L^\infty(L^2)\cap L^2(H^1)\cap \mathcal{V}^p(H^{-1})\\
		(a_0, \;\;\mathbf{A}) \quad &\longmapsto a=\pi(a_0;\mathbf{A})\, ,
		\end{align*}
		where $\pi(a_0,\mathbf{A})$ is the unique solution to \eqref{LLG} with initial condition $a_0$ and noise $\mathbf{A}$, exists and is continuous.
	\end{proposition}
	\begin{remark}
		The existence proof ensures the $L^2(L^2)\cap C([0,T];H^{-1})$ convergence of the sequence of the approximants $(u^n)_n$ to $u$. By reshaping the above claim into a Wong-Zakai convergence result, we recover the $L^\infty(L^2)\cap L^2(H^1)\cap \mathcal{V}^p(H^{-1})$ convergence. Note that no compactness is used in the proof.
	\end{remark}
	\begin{proof} We denote by $z:=a-b$, $Z:=A-B$, $\mathbb{Z}:=\AA -\mathbb{B}$ and $\omega:=\omega_{A;H^1}+\omega_{B;H^1}+\omega_{Z;H^1}$.
		From the existence and uniqueness theorem, the map $\pi$ exists and is well posed.  Then want to prove that for every $\epsilon>0$ there exists $\delta >0$ such that if  $(a_0,\mathbf{A}), (b_0,\mathbf{B}) \in H^1(\T ;\mathbb{S}^2)\times \RD ^p_a(H^1) $ are so that 
		\begin{align*}
		\|z_0\|_{H^1}^2+\|\mathbf{A}-\mathbf{B}\|_{\RD ^p_a(H^1)}<\delta\, ,
		\end{align*}
		then $\|z\|_{L^\infty(H^1)\cap L^2(H^2)}<\epsilon$, where $a=\pi(a_0;\mathbf{A})$ and $b=\pi(b_0;\mathbf{B})$. In order to conclude, from the product formula in Proposition \ref{pro:product}, we write the equation for $z^2$ as
		\begin{align*}
		\delta (\| z\|^2_{L^2})_{s,t}=\sum_{i=1}^{3}\int_{\T }\big[(\D +I+\II+\tilde\II)_{s,t}^{i,i}(0,0) +z^{\natural,2,i,i}_{s,t}\big]d x\, ,
		\end{align*}
		in the notations of \eqref{eq:general_eq_der_wz}.
		The drift estimate follows the same steps as in the proof of uniqueness in Theorem~\ref{thm:existence}, and thus can be bounded by the right hand side of \eqref{eq:drift_uniqueness}. 
		By recalling the identities \eqref{eq:identities_convergence}, the rough terms are estimated as 
		\[
		(I+\II+\tilde\II)_{s,t}(0,0)
		\leq C_0[\beta_1(Z,\mathbb{Z})_{s,t}+\omega^{1/p}(s,t)\|z\|^2_{L^\infty(s,t;L^2)}]\, ,
		\]
		where for $k\geq 1$ we introduce
		\begin{align}\label{beta_k}
		\beta_k (Z,\mathbb{Z})_{s,t}:=C_k(a,b,\mathbf{A},\mathbf{B})[\|Z_{s,t}\|_{H^{k+1}}+\|Z_{s,t}\|_{H^{k+1}}^2+\|\mathbb{Z}_{s,t}\|_{H^{k+1}}]\, ,
		\\\nonumber
		C_k:=C_k(a,b,\mathbf{A},\mathbf{B}):=\|a\|^2_{L^\infty(H^k)}+\|b\|^2_{L^\infty(H^k)}+\|\mathbf{A}\|_{\RD _a^p(H^{k+1})}+\|\mathbf{B}\|_{\RD _a^p(H^{k+1})}\,.
		\end{align}
		We know how to estimate the remainder $z^{\natural,2,i,i}$ locally in time  by means of Theorem \ref{th:wz_remainder}, where we recall that ``locally in time'' means that there exists a covering $(P_i)_i$ of $[0,T]$ such that the inequality in Theorem \ref{th:wz_remainder} holds for every $s\leq t \in P_i$ and for every $i$. By taking into account the remainder estimate from Theorem~\ref{th:wz_remainder}, the estimate of the noise terms, the drift and by choosing a suitable $\epsilon$, we conclude locally for $s\leq t \in P_i$: 
		\begin{align*}
		\delta (\|z\|^2_{L^2})_{s,t}+\int_{s}^{t}\|\partial_x  z_r\|^2_{L^2} dr
		&\leq C_1\beta_1(Z,\mathbb{Z})_{s,t}+\omega_{\mathbf{Z};H^1}^{1/p}\omega^{2/p}(s,t)+\Gamma(s,t)\|z_r\|_{L^\infty(s,t;L^2)}^2\, ,
		\end{align*}
		where $\Gamma(s,t):=C((t-s)(1+\|a_0\|_{H^1}^2+\|b_0\|_{H^1}^2)+\omega^{1/p}(s,t))$ (we used the a priori estimates on $a,b$). Then from the rough Gronwall Lemma \ref{lem:gronwall} applied to $\bar{\omega}(s,t):=\Gamma(s,t)$ and $\varphi(s,t):=C_1\beta_1(Z,\mathbb{Z},z)_{s,t}+\omega_{\mathbf{Z};H^1}^{1/p}\omega^{2/p}(s,t)$,  
		\begin{align}\label{eq:speed_wz_1}
		\delta (\|z\|^2_{L^2})_{0,T}+\int_{0}^{T}\|\partial_x  z_r\|^2_{L^2}  d r &\lesssim\exp\big(\Gamma(0,T)\big)\big[\|z_0\|_{L^2}^2+C_1\beta_1(Z,\mathbb{Z},z)_{0,T}+\omega_{\mathbf{Z};H^1}^{1/p}\omega^{2/p}(0,T)\big] \, .
		\end{align}
		The convergence in $\mathcal{V}^p(H^{-1})$ follows from considerations analogous to Corollary \ref{L2H2} and this yields our conclusion. 
	\end{proof}
	\begin{corollary}\label{cor:wz_1}
		Let $u^0\in H^1(\T ;\mathbb{S}^2)$ and $(G^n,\GG ^n)_n$ be an approximation of  $(G,\GG )$ with respect to the $\RD ^p_a(H^2)$-norm, for $p\in [2,3)$. Consider $u^n:=\pi(u^0;\G ^n)$ and $u:=\pi(u^0;\G )$.
		The following Wong-Zakai $\omega$-wise type convergence result holds
		\begin{equation*}
		\lim_{n\rightarrow +\infty} \|u^n-u\|^2_{L^\infty (L^2)\cap L^2(H^1)\cap \mathcal{V}^p(H^{-1})}=0\, ,
		\end{equation*}
		with speed of convergence $1$, i.e.  
		\begin{align}
			\|\pi(u^0,\mathbf{G}^n)-\pi(u^0,\mathbf{G})\|_{L^\infty(L^2)\cap L^2(H^1)\cap \mathcal{V}^p(H^{-1})}\lesssim \|\mathbf{G}^n-\mathbf{G}\|_{\mathcal{RD}^p_a(H^2)}\, .
		\end{align}
	\end{corollary}
	\begin{proof}
		As a consequence of Proposition \ref{prop:cont_wz_1}, since $(G^n,\GG ^n) \rightarrow (G,\GG )$ in the rough path topology, then
		\begin{equation*}
		\lim_{n\rightarrow +\infty} \pi(u^0, (G^n,\GG ^n))=\pi(u^0, (G,\GG ))
		\end{equation*}
		in $L^\infty(L^2)\cap L^2(H^1)\cap  \mathcal{V}^p(H^{-1})$.
		Indeed the constant appearing in \eqref{eq:speed_wz_1} does not depend on $n$, because the sequence $(u^n)_n$ is uniformly bounded in $L^\infty(H^1)\cap L^2(H^2)\cap \mathcal{V}^p(H^{-1})$ and $(G^n,\GG ^n)$ is uniformly bounded with respect to the rough path topology (see \cite{FrizVictoire}). We also observe that inequality \eqref{eq:speed_wz_1} furnishes also a speed of convergence: indeed when $\mathbf{G}^n$ converges to $\mathbf{G}$, $u^n$ converges to $u$ with speed $1$ (we choose the smallest exponent appearing of the $\mathcal{RD}^p_a(H^2)$ norm). 
	\end{proof}
	\paragraph{The second order Wong-Zakai convergence}
	The continuity of the It\^o-Lyons map in Proposition \ref{prop:cont_wz_1} can be stated with respect to a finer topology. Indeed in this paragraph we show the continuity of the map
	\begin{align*}
	\Psi^1:H^1(\T ;\mathbb{S}^2)\times\RD ^p_a(H^2) &\longrightarrow L^\infty(H^1)\cap L^2(H^2)\cap \mathcal{V}^p(L^2)\\
	(a_0, \;\;\mathbf{A}) \quad &\longmapsto a=\pi(a_0;\mathbf{A})\, .
	\end{align*}
	We proceed as before: the continuity of the map $\pi$ implies a Wong-Zakai convergence result that enables us to conclude the strong convergence of the sequence of the approximants $(u^n)_n$ to the solution $u$ in $L^\infty(H^1)\cap L^2(H^2)\cap\mathcal{V}^p(L^2)$, namely
	\begin{equation}\label{WZ_conv}
	\lim_{n\rightarrow +\infty} \|u^n-u\|_{L^\infty(H^1)\cap L^2(H^2)\cap \mathcal{V}^p(L^2)}^2 =0\, .
	\end{equation}
	It is worth noticing that the compactness method employed in the existence result implies only the strong convergence of $(u^n)_n$ to $u$ in $L^2(H^1)\cap C([0,T];L^2)$.  The continuity of $\pi$ implies also the continuity of the solution with respect to the initial data in $H^1(\mathbb{T};\mathbb{S}^2)$. We state the main result of this paragraph.

	\begin{theorem}\label{th:cont_wz_2} (Continuity of the It\^o-Lyons map)
		The map 
		\begin{align*}
		\Psi^1:H^1(\T ;\mathbb{S}^2)\times\RD ^p_a(H^2) &\longrightarrow L^\infty(H^1)\cap L^2(H^2)\cap \mathcal{V}^p(L^2)\\
		(a_0, \;\;\mathbf{A}) \quad &\longmapsto a=\pi(a_0;\mathbf{A})
		\end{align*}
		is continuous.
	\end{theorem}
	\begin{proof} Define $\omega:=\omega_{ \mathbf{A}; H^2}+\omega_{ \mathbf{B}; H^2}$, $z:=a-b$, $Z:=A-B$ and $\mathbb{Z}:=\AA -\mathbb{B}$.
		We proceed as in Proposition \ref{prop:cont_wz_1} and we want to prove that for every $\epsilon>0$ there exists $\delta >0$ such that if  $(a_0,(A,\AA )), (b_0,(B,\mathbb{B})) \in H^1(\T ;\mathbb{S}^2)\times\RD ^p_a(H^2)  $ are so that 
		\begin{align*}
		\|z_0\|_{H^1}^2+\|\mathbf{A}-\mathbf{B}\|_{\RD _a^p(H^2)}<\delta\, ,
		\end{align*}
		then $\|z\|_{L^\infty(H^1)\cap L^2(H^2)}<\epsilon$, where $a=\pi(a_0;\mathbf{A})$ and $b=\pi(b_0;\mathbf{B})$. 
		Thus from the product formula in Proposition \ref{pro:product}, we can test the equation by $\mathbf{1}=(1_{i=j})$, which gives (in the notations of \eqref{eq:general_eq_der_wz}):
		\begin{align*}
		\delta (\|\partial_x  z\|^2_{L^2})_{s,t}=\int_{\T }(\D +I+\II+\tilde\II)_{s,t}(1,1)d 
		x+\langle (\partial_x ^k z)_{s,t}^{\natural,2},\mathbf{1}\rangle\,  .
		\end{align*}
		 We address first the drift term 
		\begin{align*}
		\|\D ^{i,i}(1,1)_{s,t}\|_{L^1}
		=-\int_{s}^{t}\|\partial_x ^{2} z \|^2_{L^2} d r+\int_{s}^{t}\int_{\T } \big[
		(a|\partial_x  a|^2-b|\partial_x  b|^2)\cdot\partial_x ^{2}z+ (a\times\partial_x 
		^{2}a-b\times\partial_x ^{2} b)\cdot\partial_x ^{2}z\big] d x\,dr.
		\end{align*}
		Then, from Young's inequality and $\epsilon>0$
		\begin{align*}
		\int_{s}^{t}\int_{\T } | (a|\partial_x  a|^2-b|\partial_x  b|^2)\cdot \partial_x ^{2}z|dx dr
		&\leq \int_{s}^{t}\int_{\T } \left|z|\partial_x  a|^2+b(|\partial_x  a|^2-|\partial_x  
		b|^2)\right||\partial_x ^{2}z|dxdr\\
		&\leq \frac{1}{2\epsilon}\|z\|_{L^\infty(s,t;H^1)}^2\int_{s}^{t}\|\partial_x  a\|_{L^4}^4 
		d r+\frac{\epsilon}{2}\|\partial_x ^{2}z\|_{L^2(s,t;L^2)}^2\\
		&\quad+ \frac{1}{2\epsilon}\|\partial_x  z\|^2_{L^\infty(s,t;L^2)}[\|\partial_x  
		a+\partial_x  b\|^2_{L^2(H^1)}]+\frac{\epsilon}{2}\|\partial_x 
		^{2}z\|^2_{L^2(s,t;L^2)}\,,
		\end{align*}
		where we used that the solution takes value on the sphere. Again from Young's inequality, for $\epsilon>0$
		\begin{equation*}
		\int_{s}^{t}\int_{\T } (a_r\times\partial_x ^{2}a_r-b_r\times\partial_x ^{2} 
		b_r)\cdot\partial_x ^{2}z_r  d x\,dr
		\leq \frac{\epsilon}{2}\|\partial_x ^{2} 
		z\|^2_{L^2(s,t;L^2)}+\frac{1}{2\epsilon}\|z\|^2_{L^\infty(s,t;H^1)}\|b\|^2_{L^2(H^2)}\, .
		\end{equation*}
		We estimate the remainder  $\langle (\partial_x ^k 
		z)_{s,t}^{\natural,2,i,j},\mathbf{1}\rangle$ by means of Theorem \ref{th:wz_remainder} 
		\begin{align*}
		\|(\partial_x  z)^{\natural,2}_{s,t}\|_{H^{-1}}\lesssim & \;  \omega^{1/p}(s,t)\big[\omega_{\mathbf{Z};H^2}^{1/p}(s,t)+\| z\|_{L^\infty (s,t;H^1)}^2+[\omega_{\D (1,1)}+\omega_{\D (1,0)}+\omega_{\D (0,1)}+\omega_{\D (0,0)}](s,t)\big]\, .
		\end{align*}
		We see that we need to estimate the drift terms. From Proposition \ref{proposition:drift_wz}
		\begin{align*}
		\sum_{n,m=0}^2\|\D _{s,t}(n,m)\|_{H^{-1}}\lesssim 
		\|z\|_{L^\infty(s,t;H^1)}^2+\int_{s}^{t}\|\partial_x ^{2} z_r \|_{L^2}^2  d r\, .
		\end{align*}
		Now we examine the $L^1$ norm of $|I|+|\II|+|\tilde\II|$.  By recalling the definitions of $C_1$ and $\beta_1(Z,\mathbb{Z})_{s,t}$ in \eqref{beta_k}, we can bound the noise terms by $C_1[\beta_1(Z,\mathbb{Z})_{s,t}+\omega^{1/p}(s,t)\|z\|^2_{L^\infty(s,t;H^1)}]$.
		After choosing a suitable $\epsilon$ to absorb some drift elements to the left hand side, we are left with the following inequality 
		\begin{align*}
		\delta \big(\|\partial_x  z\|^{2}_{L^2} \big)_{s,t} + \frac{1}{4}\int_{s}^{t} \|\partial_x 
		^{2} z_r\|^2_{L^2} d r&\lesssim
		C_1\beta_1(Z,\mathbb{Z})_{s,t}+\omega_{\mathbf{Z};H^2}^{1/p}\omega^{2/p}(s,t)+\Gamma_1(s,t)\|z\|_{L^\infty (s,t;H^1)}^2\, ,
		\end{align*}
		where $\Gamma_1(s,t):=C[(t-s)[1+\|a_0\|_{H^1}^2+\|b_0\|_{H^1}^2]+\omega^{1/p}(s,t)]$, where $C$ is a positive universal constant. From the Gronwall Lemma \ref{lem:gronwall} applied to $\bar{\omega}:=\Gamma_1$ and $\varphi:=C_1\beta_1(Z,\mathbb{Z})+\omega_{\mathbf{Z};H^2}^{1/p}\omega^{2/p}$, we conclude that
		\begin{align}\label{eq:speed_conv_wz_2}
		\delta \big(\|\partial_x  z\|^{2}_{L^2} \big)_{0,T} + \int_{0}^{T} \|\partial_x ^{2} 
		z_r\|^2_{L^2} d r&\leq \exp\big(  \Gamma_1(0,T)    \big)
		[\|z_0\|^2_{H^1}+\varphi(0,T)]\, ,
		\end{align}
		from which we also get a speed of convergence. By coupling Theorem \ref{prop:cont_wz_1} and the last estimate, the claim follows.
	\end{proof}
	As a consequence of Theorem \ref{th:cont_wz_2} we can show the strong convergence  \eqref{WZ_conv}.
	\begin{corollary}[Wong-Zakai convergence]\label{cor:wz:2}
		Let $u^0\in H^1(\mathbb{T};\mathbb{S}^2)$ and $(G^n,\GG ^n)_n$ be an approximation of  $(G,\GG )$ with respect to the $\RD ^p_a(H^2)$-norm, for $p\in [2,3)$  Consider $u^n:=\pi(u^0;\G ^n)$ and $u:=\pi(u^0;\G )$. We obtain $\omega$-wise the following Wong-Zakai type convergence result 
		\begin{equation}
		\lim_{n\rightarrow +\infty}\|u^n-u\|_{L^\infty(H^1)\cap L^2(H^2)\cap\mathcal{V}^p(L^2)}^2=0\, .
		\end{equation}
		with speed of convergence	with speed of convergence $1$, i.e.  
		\begin{align}
		\|\Psi^1(u^0,\mathbf{G}^n)-\Psi^1(u^0,\mathbf{G})\|_{L^\infty(H^1)\cap L^2(H^2)\cap\mathcal{V}^p(L^2)}\lesssim \|\mathbf{G}^n-\mathbf{G}\|_{\mathcal{RD}^p_a(H^2)}\, .
		\end{align}
	\end{corollary}
	\begin{proof}
		A direct consequence of Theorem \ref{th:cont_wz_2}.
	\end{proof}
	
	The continuity of the It\^o-Lyons map provides also the $\omega$-wise convergence of \eqref{rLLG} driven by a vanishing noise $(\Lambda_{\sqrt{\epsilon}} \mathbf{G})_{\epsilon>0}$ to the solution to the deterministic LLG equation. 
	In the sequel we will investigate the deviations of the solutions to \eqref{rLLG} driven by vanishing noise from the solution to the deterministic equation, also known as large deviation principle.
	
	\subsection{The k-th order Wong-Zakai convergence}
	
	If the initial condition $a_0$ belongs to $H^k(\mathbb{T};\mathbb{S}^2)$ for $k\geq 2$ and the noise $\G \in\mathcal{RD}^p_a(H^{k+1})$, then from Theorem \ref{teo:higher_order_regularity} there exists a unique solution $a=\pi(a_0;\mathbf{A})$ to \eqref{rLLG} such that $a\in L^\infty(H^k)\cap L^2(H^{k+1})\cap\mathcal{V}^p(H^{k-1})$. 
	By an immediate generalization of Lemma \ref{lem:aubin}, the compactness method used in the proof of existence also implies the strong convergence of the approximants to the solution $a$ in $L^2(H^{k})\cap C([0,T];H^{k-1})$.
	We will see in this paragraph that convergence happens in a stronger sense, namely in $L^\infty(H^{k})\cap L^2(H^{k+1})\cap \mathcal{V}^p(H^{k-1})$ as a consequence of the continuity of the It\^o-Lyons map, stated in Theorem \ref{th:cont_wz_k}:
	\begin{theorem}\label{th:cont_wz_k} (Continuity of the It\^o-Lyons map)
		The map  
		\begin{align*}
		\Psi^k:H^k(\T ;\mathbb{S}^2)\times\RD ^p_a(H^{k+1})   &\longrightarrow L^\infty(H^k)\cap L^2(H^{k+1})\cap \mathcal{V}^p(H^{k-1})\\
		(a_0, \mathbf{A})  &\longmapsto a=\pi(a_0;\mathbf{A})
		\end{align*}
		is continuous.
	\end{theorem}
	The proof of Theorem \ref{th:cont_wz_k} makes use of the following Proposition \ref{proposition:drift_wz}, whose proof is postponed at the end of this subsection.
	\begin{proposition}\label{proposition:drift_wz}
		Let  $k\in\mathbb{N}$ , $a_0,b_0\in H^k(\mathbb{T};\mathbb{S}^2)$, $a,b\in L^\infty(H^k)\cap L^2(H^{k+1})\cap\mathcal{V}^p(H^{k-1})$ and $a_t(x),b_t(x)\in\mathbb{S}^2$ for a.e. $(t,x)\in   [0,T] \times\T $.
		Denote by $z:=a-b$. Then for all $s\le t \in [0,T]$
		\begin{align*}
		\|\mathcal{D}(k,k)\|_{\mathcal{V}^1(s,t;H^{-1})} 
		&\lesssim_{k,\|z_0\|_{H^k}}\|z\|_{L^\infty(s,t;H^k)}^2+\|z\|_{L^2(s,t;H^{k})}^2+\|z\|_{L^2(s,t;H^{k+1})}^2\, .
		\end{align*}
		The constant depends on some power of the $H^k$-norm of the initial conditions $a_0,b_0$ .
	\end{proposition}
	
	\begin{proof}[Proof of Theorem \ref{th:cont_wz_k}]
		Define $\omega:=\omega_{\mathbf{A};H^{k+1}}+\omega_{\mathbf{B};H^{k+1}}$, $z:=a-b$, $Z:=A-B$ and $\mathbb{Z}:=\AA -\mathbb{B}$, where $a=\pi(a_0;\mathbf{A})$ and $b=\pi(b_0;\mathbf{B})$.
		We can prove the claim by induction, assuming as the initial step Theorem \ref{th:cont_wz_2} and as the inductive step that the map
		\begin{align*}
		\Psi^{k-1}:H^{k-1}(\T ,\mathbb{S}^2)\times\RD ^p_a(H^k)  &\longrightarrow L^\infty(H^{k-1})\cap L^2(H^{k})\\
		(a_0, \mathbf{A})  &\longmapsto a=\pi(a_0;\mathbf{A})
		\end{align*}
		is continuous and that the a priori bound holds:
		\begin{align*}
		\delta(	\|z\|^2_{H^{k-1}})_{s,t}+\int_{s}^{t}\|z_r\|^2_{H^{k}} d r\lesssim
		\exp(\Gamma_{k-1}(s,t))[C_{k-1}\beta_{k-1}(Z,\mathbb{Z})_{s,t}+\omega_{\mathbf{Z};H^{k}}^{1/p}\omega^{2/p}(s,t)
		+\|\partial_x^{k-1} z_0\|_{L^2}^2]\, ,
		\end{align*}
		where $\Gamma_{k-1}(s,t):=C[(t-s)(1+\|a_0\|_{H^{k-1}}^2+\|b_0\|_{H^{k-1}}^2)+\omega^{1/p}(s,t)]$, where $C$ is a positive universal constant. Recall the definition of $C_k$ and $\beta_{k-1}(Z,\mathbb{Z})$ in \eqref{beta_k}.
		As in Theorem \ref{prop:cont_wz_1} and in Theorem \ref{th:cont_wz_2}, we want to prove that for every $\epsilon>0$ there exists $\delta >0$ such that if  $(a_0,\mathbf{A}), (b_0,\mathbf{B}) \in H^k(\T ;\mathbb{S}^2)\times\RD ^p_a(H^k)  $ are so that 
		\begin{align*}
		\|z_0\|_{H^k}^2+\|\mathbf{A}-\mathbf{B}\|_{\RD ^p_a(H^{k+1})}<\delta\, ,
		\end{align*}
		then $\|a-b\|_{L^\infty(H^k)\cap L^2(H^{k+1})}<\epsilon$. From the product formula in Proposition \ref{pro:product}, we can test the equation by $\mathbf{1}=(1,1,1)\in \mathbb{R}^3$
		\begin{align}\label{eq:eq1}
		\delta (\|\partial_x ^{k} z\|^2_{L^2})_{s,t}=\int_{\T }(\D + I+\II+\tilde\II)_{s,t}(k,k) d 
		x+\langle (\partial_x ^k z)_{s,t}^{\natural,2},\mathbf{1}\rangle\, ,
		\end{align}
		in the notations of \eqref{eq:general_eq_der_wz}. 
		First we analyse the drift, 
		\begin{align*}
		\int_{\T }\D _{s,t}(k,k) d x&=-\int_{s}^{t}\|\partial_x ^{k+1}z_r\|^2_{L^2} dr+ 
		\int_{s}^{t} \int_{\T } \partial_x ^{k-1} (a_r|\partial_x  a_r|^2-b_r|\partial_x  
		b_r|^2)\cdot\partial_x ^{k+1} z_r  d x \,d r\\
		&\quad+\int_{s}^{t} \int_{\T }\partial_x ^{k-1}(a_r\times\partial_x 
		^{2}a_r-b_r\times\partial_x ^{2} b_r)\cdot \partial_x ^{k+1} z_r  d x \, dr\, .
		\end{align*}
		We note that the $L^2(L^2)$ norm of $\partial_x ^{k+1}z$ has a negative sign, thus we pass 
		it to the left hand side of \eqref{eq:eq1}: it will give regularity in the a priori 
		estimate. The other elements can be estimated with similar computations as in Proposition 
		\ref{proposition:drift_wz}.
		As a consequence of Theorem \ref{th:wz_remainder} and of Proposition \ref{proposition:drift_wz}, we can estimate the remainder term.
		By analogous computations as in Theorem \ref{th:cont_wz_2} and by the Gronwall's lemma, we conclude that
		\begin{equation}\label{eq:speed_conv_k}
		\delta(	\|z\|^2_{H^{k}})_{0,T}+\int_{0}^{T}\|z_r\|^2_{H^{k+1}} d r\leq 
		\exp(\Gamma_{k}(0,T))[C_{k}\beta_{k}(Z,\mathbb{Z})_{0,T}
		+\omega_{\mathbf{Z};H^{k+1}}^{1/p}\omega^{2/p}(0,T)+\|\partial_x^k z_0\|_{L^2}^2]\, ,
		\end{equation}
		where $\Gamma_{k}(s,t):=C((t-s)(1+\|a_0\|_{H^{k}}^2+\|b_0\|_{H^{k}}^2)+\omega^{1/p}(s,t))$.
		From the inductive step and the hypothesis, the continuity of the map follows. Note that we get again a speed of convergence $1$. Consider the approximations $(\mathbf{G}_n)_n$ to a rough path $\mathbf{G}$. Then we read by \eqref{eq:speed_conv_k} that
		\begin{align*}
			\|\Psi^k(u^0,\mathbf{G}^n)-\Psi^k(u^0,\mathbf{G})\|_{L^\infty(H^k)\cap L^2(H^{k+1})}\lesssim\|\mathbf{G}^n-\mathbf{G}\|_{\mathcal{RD}^p_a(H^{k+1})}\,,
		\end{align*}
		which means that the speed of convergence is again $1$.
	\end{proof}
	The conclusions in Corollary \ref{cor:wz:2} hold also for the improved topology obtained in Theorem \ref{th:cont_wz_k}.
	\begin{corollary}\label{cor:wz_k}
	Let $u^0\in H^k$ and $(G^n,\GG ^n)_n$ be an approximation of  $(G,\GG )$ with respect to the $\RD ^p_a(H^{k+1})$-norm, for $p\in [2,3)$.  Consider $u^n:=\pi(u^0;\G ^n)$ and $u:=\pi(u^0;\G )$. We obtain $\omega$-wise the following Wong-Zakai type convergence result 
	\begin{equation}
	\lim_{n\rightarrow +\infty}\|u^n-u\|_{L^\infty(H^k)\cap L^2(H^{k+1})\cap  \mathcal{V}^p(H^{k-1})}^2=0\, ,
	\end{equation}
	with speed of convergence	with speed of convergence $1$, i.e.  
	\begin{align}
	\|\Psi^k(u^0,\mathbf{G}^n)-\Psi^k(u^0,\mathbf{G})\|_{L^\infty(H^{k})\cap L^2(H^{k+1})\cap \mathcal{V}^p(H^{k-1})}\lesssim \|\mathbf{G}^n-\mathbf{G}\|_{\mathcal{RD}^p_a(H^{k+1})}\, .
	\end{align}
	\end{corollary}
	
		\begin{proof}[Proof of Proposition \ref{proposition:drift_wz}]
		Inspired by the computations on the drift in \cite{Melcher}, we consider 
		\begin{align*}
		\|\partial_x ^{k} ((a |\partial_x  a|^2)^i&-\partial_x ^{k} (b|\partial_x  
		b|^2)^i)\|_{L^2}\leq  \|\partial_x ^{k} ((a-b)^i |\partial_x  a|^2)\|_{L^2}+\|\partial_x 
		^{k} (b^i (|\partial_x  a|^2-|\partial_x  b|^2))\|_{L^2}\, .
		\end{align*}
		We recall that for $ 0 \leq|\alpha|\leq k $
		\begin{align*}
		\partial_x ^{\alpha} (f g)=\sum_{j=1}^{|\alpha|}\binom{|\alpha|}{j}\partial_x 
		^{|\alpha|-j}f\partial_x ^{j} g\, .
		\end{align*}
		From the interpolation inequality \eqref{eq:interp_iterative_1} applied to $f:=z^i \partial_x  a$ and $g:=\partial_x  a\cdot$, it follows that
		\begin{align*}
		\|\partial_x ^{k} (z^i |\partial_x  a|^2)\|_{L^2}&\lesssim_k\|z^i\|_{L^\infty} 
		\|\partial_x  a\|_{L^\infty}\|\partial_x  a\|_{H^k}+\| z^i \partial_x  
		a\cdot\|_{H^k}\|\partial_x  a\|_{L^\infty}\\
		&\lesssim_k\|z^i\|_{L^\infty} \|\partial_x  a\|_{L^\infty}\|\partial_x  a\|_{H^k}+\| z^i \|_{H^k}\| \partial_x  a\|_{H^k}\|\partial_x  a\|_{L^\infty}\, .
		\end{align*}
		By employing the Morrey's embedding theorem in one dimension, 
		\begin{align*}
		\|\partial_x ^{k} (z^i |\partial_x  a|^2)\|_{L^2}&\lesssim_k \|z^i\|_{H^1} \|\partial_x  
		a\|_{H^1}\|\partial_x  a\|_{H^k}+\| z^i \|_{H^k}\| \partial_x  a\|_{H^k}\|\partial_x  
		a\|_{H^1}\, .
		\end{align*}
		From the interpolation inequality \eqref{eq:interp_iterative_1} applied to $f:=b^i (\partial_x  a+\partial_x  b)$ and $g:=(\partial_x  a-\partial_x  b)\cdot$, it follows 
		\begin{align*}
		\|\partial_x ^{k} (b^i (|\partial_x  a|^2-|\partial_x  b|^2))\|_{L^2}&\lesssim_k \|b^i 
		(\partial_x  a+\partial_x  b)\|_{L^\infty}\|\partial_x  a-\partial_x  b\|_{H^k}\\
		&\quad+\| b^i (\partial_x  a+\partial_x  b)\|_{H^k}\|\partial_x  a-\partial_x  b\|_{L^\infty}\\
		&\lesssim_k  \|b\|_{L^\infty} \|\partial_x  a+\partial_x  b\|_{L^\infty}\|\partial_x  a-\partial_x  b\|_{H^k}\\
		&\quad+\| b\|_{H^k} \|\partial_x  a+\partial_x  b\|_{H^k}\|\partial_x  a-\partial_x  b\|_{L^\infty}\, ,
		\end{align*}
		where in the last inequality we have employed again the interpolation inequality \eqref{eq:interp_iterative_1} to $f:=b^i $ and $g:=(\partial_x  a+\partial_x  b)\cdot$.
		Hence we can estimate the following
		\begin{align*}
		\int_{s}^{t}\|(\partial_x ^{k} (a_r |\partial_x  a_r|^2)-\partial_x ^{k} (b_r|\partial_x  
		b_r|^2))  \|_{L^2}  d r&\lesssim_k\|z\|_{L^\infty(s,t;H^1)} \int_{s}^{t} \|\partial_x  
		a_r\|_{H^1}\|\partial_x  a_r\|_{H^k}  d r\\
		&\quad +\|z\|_{L^\infty(s,t;H^k)} \int_{s}^{t} \|\partial_x  a_r\|_{H^k}\|\partial_x  a_r\|_{H^1}  d r\\
		&\quad +\|b\|_{L^\infty(L^\infty)}\int_{s}^{t} \|\partial_x  a_r+\partial_x  b_r\|_{H^1}\|\partial_x  a_r-\partial_x  b_r\|_{H^k}  d r\\
		&\quad +\|b\|_{L^\infty (H^k)} \int_{s}^{t} \|\partial_x  a_r+\partial_x  b_r\|_{H^k}\|\partial_x  a_r-\partial_x  b_r\|_{H^1}  d r\, .
		\end{align*}
		Notice that if $k=1$, the above is bounded; if $k>1$, then we can consider the boundedness of $\|a\|_{L^\infty(H^2)}$, $\|b\|_{L^\infty(H^2)}$.
		In conclusion for every $\Phi\in H^1$ such that $\|\Phi\|_{H^1}\leq 1$, from the H\"older's inequality in time and in space and Young's inequality for all $\epsilon>0$
		\begin{align*}\label{eq:wz_k_difference_estimate}
			\int_{s}^{t}\langle\partial_x ^{k}( a |\partial_x  a|^2-b|\partial_x  
			b|^2)^i\partial_x^k z^j,\Phi \rangle d r 
		&\lesssim_k\|\partial_x ^{k} (a |\partial_x  a|^2)^i-\partial_x ^{k} (b|\partial_x  
		b|^2)^i  \|_{L^1(s,t;L^2)}\|\partial_x^k z\|_{L^\infty(s,t;L^2)}\|\Phi\|_{H^1} \\
		&\lesssim_k \|z\|_{L^2(s,t;H^k)}^2+\|z\|_{L^2(s,t;H^{k+1})}^2+\|z\|_{L^\infty(s,t;H^k)}^2\, ,
		\end{align*}
		where the constant depends on some powers of the initial condition $\|a^0\|_{H^k(\mathbb{T};\mathbb{S}^2)}$ and $\|b^0\|_{H^k(\mathbb{T};\mathbb{S}^2)}$ (the bound are a consequence of the a priori estimates in the existence proof).
		Consider the explicit  derivatives of $\partial_x ^{k} (z\times \partial_x  a)$,
		\begin{align*}
		\partial_x ^{k} (z\times  \partial_x  a)=\sum_{j=0}^{k}\binom{k}{j}\partial_x ^{k-j}z\times 
		\partial_x ^{j}  \partial_x  a,\quad \partial_x ^{k} (b\times  \partial_x  
		z)=\sum_{j=0}^{k}\binom{k}{j}\partial_x ^{k-j}b\times\partial_x ^{j}  \partial_x  z\, .
		\end{align*}
		First we estimate the $L^2$-norm of the $0$-th and the $k$-th terms of the first sum above,
		\begin{align*}
		\|\partial_x ^{k}z\times \partial_x  a\|_{L^2}+\|z\times \partial_x 
		^{k+1}a\|_{L^2}&\lesssim_k \|\partial_x ^{k}z\|_{L^2}\|\partial_x  a\|_{L^\infty}+\| 
		\partial_x ^{k+1}a\|_{L^2}\|z\|_{L^
			\infty}\\
		&\lesssim_k \|\partial_x ^{k}z\|_{L^2}\|\partial_x  a\|_{H^1}+\| \partial_x 
		^{k+1}a\|_{L^2}\|z\|_{H^1}\, .
		\end{align*}
		Then the other elements of the sum for $j=1,\dots,k-1$ are bounded by
		\begin{align*}
		\|\partial_x ^{k-j}z\times \partial_x ^{j}  \partial_x  a\|_{L^2} \lesssim_k \|\partial_x 
		^{k-j}z\|_{L^\infty}\|\partial_x ^{j}  \partial_x  a\|_{L^2}\, ,
		\end{align*}
		and the same computations to the second sum $ \partial_x ^{k} (b\times  \partial_x  z)$ and 
		conclude the following bound:
		\begin{align*}
		\|\partial_x ^{k-j} (b\times  \partial_x ^j z)\|_{L^2}\lesssim_k \|\partial_x 
		^{k-j}b\|_{L^\infty}\|\partial_x ^{j}  \partial_x  z\|_{L^2}\, .
		\end{align*}
		In conclusion, from H\"older's inequality, weighted Young's inequality and the above estimates, for all $\Phi\in H^1$
		\begin{align*}
		&\int_{s}^{t}\langle(\partial_x ^{k} (a_r \times \partial_x  a_r)-\partial_x ^{k} 
		(b_r\times \partial_x  b_r))^i\partial_x ^{k+1} z_r^j,\Phi\rangle d 
		r+\int_{s}^{t}\langle(\partial_x ^{k} (a_r \times \partial_x  a_r)-\partial_x ^{k} 
		(b_r\times \partial_x  b_r))^i\partial_x ^{k} z^j_r,\partial_x\Phi\rangle d r\\
		&\quad	\lesssim_k \big[\|z\|^2_{L^\infty(s,t;H^k)}+ 
		\|\partial_x ^{k+1} 
		z^j\|^2_{L^2(s,t;L^2)}\big]\|\Phi\|_{H^1}\,,
		\end{align*}	
		where the constant depends again on some power of the initial conditions. This estimate  concludes our proof.
	\end{proof}
	
	\section{Support theorem and a large deviations principle}
	\label{sec:LD}
	
	In Section \ref{sec:wong-zakai} we stated for every $k\in \mathbb{N}$ the continuity of the maps
	\begin{align*}
	\Psi^k:H^k(\T ;\mathbb{S}^2)\times\RD ^p_a(H^{k+1})   &\longrightarrow L^\infty(H^k)\cap L^2(H^{k+1})\cap \mathcal{V}^p(H^{k-1})\\
	(u^0, \;\;\G ) \quad &\longmapsto u=\pi(u^0;\G )\, ,
	\end{align*}
	where $\pi(u^0,\G )$ is the unique solution to \eqref{rLLG} and we established that the solution $u$ is a random variable on a probability space $(\Omega, \mathcal{G}, \mathbb{P})$ adapted with respect to the natural filtration. As straightforward consequences of the continuity of the maps $\Psi^k$ we can derive a support theorem for the law of the solution $u$ and a large deviation principle, both in the case of Brownian motion. In order to simplify the presentation, we restrict ourselves to a noise term of the form considered in Example \ref{exa:B}.

	\subsection{A  Stroock-Varadhan support theorem.}
	Fix $k\in \mathbb{N}$ and $p\in (2,3)$. We aim to describe the support of the law of the unique solution to \eqref{rLLG}. Consider a filtered probability space $(\Omega,\mathcal{G},(\mathcal{G}_t)_t,\mathbb{P})$, let 
	 $\omega\rightarrow\rp(\omega)\equiv(\beta(\omega),\rpp(\omega))$  be a random pathwise rough driver and $\omega\rightarrow \beta(\omega)$ is a Brownian motion with respect to $(\mathcal{G}_t)_t$, $g\in H^{k+1}$ and $\G(\omega)=\Lambda_{g(x)}\mathcal F(\rp(\omega))$ be as in \eqref{simple_RD}.
	 Given $u^0\in H^k(\T ;\mathbb{S}^2)$, let $u\colon \Omega\times [0,T]\to L^2$ be the pathwise solution to \eqref{rLLG} as built in Theorem \ref{thm:existence}.
	We recall that $u=\pi(u^0,\G )$, $\mathbb P$-a.s., where $\pi$ denotes the solution map and is deterministic.
	 Let  $\mu_u:=\mathbb{P}\circ u^{-1}$ be the law of the previous solution on the measurable space $(\mathcal Y,\mathcal B_{\mathcal Y})$, where $\mathcal{Y}:=L^\infty(H^k)\cap L^2(H^{k+1})\cap \mathcal{V}^p(H^{k-1})$ and $\mathcal{B}_\mathcal{Y}$ denotes its Borel $\sigma$-algebra. 
	 We want to describe the support of $\mu_u$ on $(\mathcal{Y}, \mathcal{B}_\mathcal{Y})$, that is we want to characterise
	\begin{align*}
	\mathrm{supp}(\mu_u)=\{y\in \mathcal{Y}: \forall  \text{ open neighbourhood } A_y\in \mathcal{B}_{\mathcal{Y}}\text{ of } y\, \text{ with }  \mu_u(A_y)>0 \}\, .
	\end{align*}

	For a continuous path $w\colon [0,T]\rightarrow \R ^3$,  $n\in \mathbb{N}$, we define $D_n:=\{Ti/2^n:i\in \{0,\cdots, 2^n-1\}\}$ and $w^n\colon [0,T]\rightarrow \R^3$ be the continuous linear path defined as
	\begin{equation}  \label{w_n_def}
	w^n_t:=\sum_{i\in D_n} \mathbf 1_{[t_i,t_{i+1} )}(t)\bigg[w_{t_i}+\frac{t-t_{i}}{t_{i+1}-t_i}(w_{t_{i+1}}-w_{t_i})\bigg]\, .
	\end{equation}
	The paths $(w^n)_n$ are referred to as (dyadic) piecewise linear approximations of the path $w$. 
	For the Brownian motion defined in Example \ref{exa:B}, the sequence $(\boldsymbol\beta^n)_n=(\beta^n,\rpp^n)$, where the second component is defined canonically through $\rpp^n_{s,t}:=\int_s^t\delta \beta_{s,r}^n\otimes d\beta^n _r$, can be used to approximate the geometric rough path $\rp$, in the $p$-variation topology.
	For $p\in (2,3)$ we denote by $\D \subset \mathcal{V}^p([0,T]; \mathbb{R}^3)\times\mathcal{V}_2^{p/2}([0,T]; \mathbb{R}^{3\times3})$ the space of (canonical enhancements of) piecewise linear approximations of continuous paths as in \eqref{w_n_def}.
	
	We recall a version of support theorem for the rough Brownian motion (see \cite{Ledoux_Qian_Zhang_p_var}, Theorem 5):
	\begin{proposition}\label{pro:support_p_var}
		Let $\rp=(\delta\beta,\rpp)$ be the enhanced Brownian motion defined in Example \ref{exa:B} with $H=\frac12$, then the support of its distribution is the closure of the space of dyadic rough paths $\mathcal{D}$ with respect to the $p$-variation norm, for $p\in (2,3)$. 
	\end{proposition}
	As a consequence of Proposition \ref{pro:support_p_var} and the continuity of the It\^o-Lyons map we conclude a support theorem for the solution to \eqref{rLLG}.
	\begin{theorem}
		Let $k\in \mathbb{N}$, $u^0\in H^k(\T ;\mathbb{S}^2)$ and $u=\pi(u^0,\G )\in \mathcal{Y}$ be the solution to \eqref{rLLG}, where in the notations of Section \ref{sec:RD}, 
		\[
		\G(\omega;x) =(g(x)\mathcal F(\delta\beta(\omega)),g(x)^2\mathcal F^{\otimes 2}\rpp(\omega))\in \RD ^p_a(H^{k+1})
		\]
		 is the Brownian rough driver as defined in \eqref{simple_RD} with $H=\frac12$.
		 Then the support of its distribution is the closure $\overline{\Psi^k(u^0;\Lambda_{g(\cdot)}\circ\mathcal F(\mathcal D))}^{p\mathrm{-var};\mathcal{Y}}$ of the image through the solution map $\Psi^k(u^0;\Lambda_{g(\cdot)}\circ\mathcal F(\cdot))  $ of the space of dyadic rough paths with respect to the $\mathcal{Y}$-norm, for $p\in (2,3)$. 
	\end{theorem}
	\begin{proof}
		From Proposition \ref{pro:support_p_var},  $\mathrm{supp} \; \mathbb{P} \circ \rp^{-1}=\overline{\D }^{\,p\mathrm{-var};\R^3\times \R^{3\times 3}}$.
		In the notations of Section \ref{sec:RD}, define $\phi(\cdot):=\Psi^k(u^0,\Lambda_{g(x)}\mathcal{F}(\cdot))$, so that $u=\phi(\rp)$.
		We recall that the map $\rp\rightarrow \Lambda_{g(\cdot)} \mathcal{F}(\rp)=\G$ is continuous and that the solution map is also continuous with respect to the rough driver. 
		  We aim to show that $\mathrm{supp} \, \mu_u=\overline{\phi(\D )}^{\, p\mathrm{-var};\mathcal{Y}}\,.$

		We prove first that $\mathrm{supp}\;\mu_u\subset\overline{\phi(\D )}^{\, p\mathrm{-var};\mathcal{Y}}$. By our construction of the solution, $u=\lim_{n\rightarrow +\infty}\phi(\rp^n)$ $\mathbb{P}-$a.s. and therefore $\mathbb{P}(u\in\overline{\phi(\D )}^{\, p\mathrm{-var};\mathcal{Y}})=1$. This implies that $\mathrm{supp}\;\mu_u\subset\overline{\phi(\D )}^{\, p\mathrm{-var};\mathcal{Y}}$, indeed by an equivalent definition 
		\begin{align*}
		\mathrm{supp}\;\mu_u=\cap_{A\in \mathcal{B}_\mathcal{Y}:A=\bar{A}} \{\mathbb{P}(u\in A^c)=0\}=\cap_{A\in \mathcal{B}_\mathcal{Y}:A=\bar{A}} \{\mathbb{P}(u\in A)=1\},
		\end{align*}
		we observe that the support is the smallest set of full probability. Since also the set $\overline{\phi(\D )}^{\, p\mathrm{-var};\mathcal{Y}}$ is of full probability, then $\mathrm{supp}\;\mu_u\subset\overline{\phi(\D )}^{\, p\mathrm{-var};\mathcal{Y}}$.
		
		To show the opposite direction,
		let $v\in \overline{\phi(\D )}^{\, p\mathrm{-var};\mathcal{Y}}=\phi(\overline{\D }^{\, p\mathrm{-var};\R^3\times \R^{3\times 3}})=\phi(\mathrm{supp} \, \mathbb{P}\circ \rp^{-1})$ (where we used the continuity of $\phi$ and Proposition \ref{pro:support_p_var} ), then there exists a sequence $(\gamma_n)_n\subset \mathcal{D}$ such that $\lim_{n\rightarrow +\infty}\gamma_n=\gamma$ and $v=\phi(\gamma)$. Let  $A_v\in \mathcal{B}_{\mathcal{Y}}$ be an open neighbourhood of $v$, then
		\begin{align*}
			\mathbb{P}(u\in A_v)=\mathbb{P}(\phi(\rp)\in A_v)=\mathbb{P}(\rp\in\phi^{-1}( A_v)).
		\end{align*}
		We note that $\phi^{-1}( A_v)$ is a neighbourhood of $\gamma$, indeed $v=\phi(\gamma)\in A_v$ and $\phi^{-1}(A_v)$ is an open set, from the continuity of $\phi$. Since $\gamma\in \mathrm{supp} \; \mathbb{P} \circ \rp^{-1}$, from the definition of support $\mathbb{P}(\gamma\in\phi^{-1}( A_v))>0.$  Thus also $\mathbb{P}(u\in A_v)>0$ and $v\in \mathrm{supp}\;\mu_u$ and this concludes the proof.
	\end{proof}
		
	\subsection{A large deviations principle}
	Informally speaking, we expect that if the noise approaches $0$, then the stochastic solution $u^\epsilon=\pi(u^0,\Lambda_{\sqrt{\epsilon}}\mathbf{G})$ to \eqref{rLLG} driven by the vanishing noise $\Lambda_{\sqrt{\epsilon }}\mathbf{G}$ approaches the deterministic solution $u=\pi(u^0,0)$ to the deterministic LLG equation for $\epsilon\rightarrow 0$. In order to measure the deviation of the stochastic solutions from the deterministic solution, we investigate the asymptotic behaviour of the the trajectories $u^\epsilon-u$.  We restrict the exposition to the case in which the solution is driven by the enhancement of a Brownian motion.
	
	In the classical theory, the large deviations of Brownian motion were proved in the Schilder's theorem (see for instance Theorem 13.38 in \cite{FrizVictoire}). 
	Let $\mathcal{X}$ be a topological space endowed with the Borel $\sigma$-algebra $\mathcal{B}_{\mathcal{X}}$.
	Recall that a map $\mathcal{I}\colon \mathcal{X}\rightarrow [0,+\infty] $ is a \textit{a good rate function} if the level sets $\{x\in \mathcal{X}: \mathcal{I}(x)\leq c\}$ for  $c\in\R $ are compact.
	Moreover, given such $\mathcal I,$ a family of probability measures $(\mu_\epsilon)_{\epsilon>0}$ is said to satisfy a \textit{large deviation principle (LDP)} on $(\mathcal{X}, \mathcal{B}_{\mathcal{X}})$ with good rate function $\mathcal{I}$ if for every $A\in \mathcal{B}_{\mathcal{X}}$ 
	\begin{equation*}
	-\mathcal{I}(A^o)\leq \liminf_{\epsilon\rightarrow 0}\epsilon\log\mu_\epsilon(A)\leq\limsup_{\epsilon\rightarrow 0}\epsilon\log\mu_\epsilon(A) \leq -\mathcal{I}(\bar{A})\, ,
	\end{equation*}
	where $A^o$ is the interior of the set $A$, $\bar{A}$ is the closure of $A$. 
	
	The following theorem is the rough counterpart of Schilder's theorem for enhanced Brownian motion (see \cite[Theorem 1]{Ledoux_Qian_Zhang_p_var} or \cite[Theorem 13.43]{FrizVictoire} in the $\alpha$-H\"older topology) and it states a large deviation principle for the vanishing enhanced Brownian motion $\Lambda_{\sqrt{\epsilon}} \rp\equiv (\sqrt{\epsilon} \delta\beta,\epsilon \rpp)$, where $\epsilon>0$. 
	We introduce the \textit{Cameron-Martin space} which is the subset of $W^{1,2}([0,T];\R ^3)$ whose paths are null at $0$, i.e.
	\begin{equation*}
	\mathcal{H}:=\Big\{\int_{0}^{\cdot}\dot{h}_t dt: \dot{h}\in L^2([0,T];\R ^3), \,h_0=0\Big\}\, .
	\end{equation*}
	
	\begin{theorem}[Schilder theorem for rough paths]
		The family $(\Lambda_{\sqrt{\epsilon} }\rp)_{\epsilon>0}$ satisfies a large deviation principle in  $\mathcal{V}^p(\R ^3)\times\mathcal{V}^{p/2}_2(\R ^3)$ for $p\in (2,3)$ with good rate function
		\begin{equation}\label{good_rate_fc_BM}
		\mathcal{I}(h):=\left\{
		\begin{aligned}
		&\int_{0}^{T}|\dot{h}_t|^2 dt \quad\text{if}\quad  h\in \mathcal{H}
		\\[0.3em]
		&+\infty\quad \quad \quad \quad\mathrm{otherwise}\, .
		\end{aligned}
		\right.
		\end{equation} 
	\end{theorem}
	We aim to derive a large deviation principle for the family of probability measures $\nu_{\epsilon}:=\mathbb{P}\circ (u^{\epsilon})^{-1}$, $\epsilon>0$, where   $u^\epsilon=\Psi^k(u^0,\Lambda_{\sqrt{\epsilon} g(\cdot)}\mathcal{F}(\rp) )$ and $\rp$ the Stratonovich enhancement of a $\mathbb{R}^3$ valued Brownian motion (recall that we take in consideration rough drivers of the form introduced in Example \ref{exa:B}).
	The large deviation principle for $(\nu_{\epsilon})_{\epsilon>0}$ is a consequence of the following contraction principle (Appendix C, Theorem C.6 in \cite{FrizVictoire}).
	\begin{theorem}\label{teo:contr-principle}
		(Contraction principle) Let $\mathcal{X}$ and $\mathcal{Y}$ be Hausdorff topological spaces. 
		Suppose that $f:\mathcal{X}\rightarrow\mathcal{Y}$ is a continuous measurable map. 
		If a set of probability measures $(\mu_\epsilon)_{\epsilon>0}$ on $\mathcal{X}$ satisfies a large deviation principle on $\mathcal{X}$ with good rate function $\mathcal{I}$, then the image measures $(\mu_\epsilon\circ f^{-1})_{\epsilon>0}$ satisfy a large deviation principle on $\mathcal{Y}$ with good rate function
		\begin{align*}
		\mathcal{J}(y):= \inf \big(\mathcal{I}(x) :\,x\in \mathcal{X}\;\mathrm{and}\;f(x)=y\big)\, .
		\end{align*}
	\end{theorem}
	The next theorem states the large deviation principle for vanishing noise for the solution to \eqref{rLLG} on $\T $. 
	\begin{theorem} (Freidlin-Wentzell large deviations)\label{teo:Schilder-RP}
		Fix $k\in \mathbb{N}$, $p\in (2,3)$ and $u^0\in H^k(\T ;\mathbb{S}^2)$. 
		Consider $\mathbf{G}:=\Lambda_{g(\cdot)} \mathcal{F}(\rp)$, where $g\in H^{k+1}$ and $\rp$ is  the enhanced Brownian motion.
		Let $\epsilon>0$ and $u^\epsilon=\Psi^k(u^0,\Lambda_{\sqrt{\epsilon}}\mathbf{G} )$ be a solution to the \eqref{rLLG} and $u=\Psi^k(u^0, 0)$  be the deterministic solution to the (LLG) on $\T\times [0,T]$. 
		Then $(\nu_{\epsilon})_{\epsilon>0}$ satisfies a large deviations principle in $L^\infty(H^k)\cap L^2(H^{k+1})\cap \mathcal{V}^p(H^{k-1})$  with good rate function
		\begin{align}\label{rate_fct_on_Y}
		\mathcal{J}(y):= \inf _{h\in\mathcal{H}}\big(\mathcal{I}(h) : \Psi^k(u^0,\Lambda_{g(\cdot)}\mathcal{F}(h))=y\big)\, ,
		\end{align}
		where $\mathcal{I}$ is defined in \eqref{good_rate_fc_BM}.
	\end{theorem}
	\begin{proof}
		As a consequence of Theorem \ref{th:cont_wz_2}, the map that associates the noise $\Lambda _{\sqrt{\epsilon}} \G $ to the solution $\pi(u^0;\Lambda_{\sqrt{\epsilon}} \G )$ is continuous. Since $\G=\Lambda_{g(\cdot)} \mathcal{F}(\rp)$ and $\mathcal{F}$ is continuous, also the map that associates $\rp$ to $\Psi^k(u^0, \Lambda_{\sqrt{\epsilon}g(\cdot)} \mathcal{F}(\rp))$ is continuous. 
		
		We state the problem in the framework of Theorem \ref{teo:contr-principle}.
		Since every metric space is a Hausdorff space, we can set  $\mathcal{X}:= \mathcal{V}^p(\R ^3)\times\mathcal{V}^{p/2}_2(\R ^3)$ and $\mathcal{Y}:=L^\infty(H^k)\cap L^2(H^{k+1})\cap \mathcal{V}^p(H^{k-1})$. Define $f(\rp):=\Psi^k(u^0,\Lambda_{g(\cdot)}\mathcal{F}(\rp))$ for $\rp \in \mathcal{X}$ and the family of measures $\mu_{\epsilon}(\cdot):=\mathbb{P}\circ(\Lambda_{\sqrt{\epsilon}}\boldsymbol\beta )^{-1}$ on $\mathcal{X}$. In addition, by Theorem \ref{teo:Schilder-RP} the sequence of measures $(\mu_{\epsilon})_{\epsilon>0}$ satisfies a large deviation principle on $\mathcal{X}$ with good rate function $\mathcal{I}$ as in \eqref{good_rate_fc_BM}. We can therefore apply the contraction principle in Theorem \ref{teo:contr-principle} and conclude that the sequence of probability measures $( \mu_{\epsilon}\circ f^{-1})_\epsilon$ satisfies a LDP on $\mathcal{Y}$ with good rate functions defined in \eqref{rate_fct_on_Y}.
	\end{proof}
		
	\appendix
	\section{Combinatorial Lemmas}
	\label{app:combinatorial}

	Here we prove some combinatorial results that are used both in the proof of higher order regularity or in the wong-zakai convergence.
	We introduce some algebraic identities that will be useful in the following: for any real maps $(g_t)$, $(G_{s,t})$
	and every $s\le u\le t\in[0,T]$:
	\begin{align}\label{der_1}
	\delta[(s,t)\mapsto G_{s,t}g_s]_{s,u,t}
	&=-G_{u,t}\delta g_{s,u}+\delta G_{s,u,t} g_s\, ,
	\intertext{and if $H=H_{s,t}$ is another such two-index map}
	\label{der2_gene}
	\delta [(s,t)\mapsto G_{s,t}H_{s,t}]_{s,u,t}
	&=G_{s,u}H_{u,t}+ G_{u,t}H_{s,u} + \delta G_{s,u,t}H_{s,t} + G_{s,t}\delta G_{s,u,t} - \delta G_{s,u,t}\delta H_{s,u,t}
	\end{align}
	 If  $G$ and $H$ are increments, the formula \eqref{der2_gene} simplifies and we have
	\begin{equation}\label{der_2}
	\delta [(s,t)\mapsto G_{s,t}H_{s,t}]_{s,u,t}=G_{s,u}H_{u,t}+G_{u,t}H_{s,u}\,.
	\end{equation}

	\subsection{The first Lemma}
	Recall the notations \eqref{tensor_vector}-\eqref{tensor_matrix}.
	We need the following preliminary result.
	\begin{lemma}\label{Lemma_1_combinatorics}
		Let $a,b,c,d:[0,T]\rightarrow \R^n $ and $A,B,C,D:[0,T]\times [0,T]\rightarrow \mathcal L(\R ^n)$.
		 Then for $s <u<t\in[0,T]$ we have 
		\begin{multline}
		\label{eq:combinatorics_1}
		\delta [(s,t)\mapsto(A_{s,t}a_s-B_{s,t}b_s)\otimes c_s]_{s,u,t}
		\\
		=-\big((A-B)_{u,t}\otimes \mathbf 1_{n\times n}\big)\delta [a\otimes c]_{s,u}-B_{u,t}\otimes \mathbf 1_{n\times n}\delta [(a-b)\otimes c]_{s,u}
		+(\delta A_{s,u,t}a_s-\delta B_{s,u,t}b_s)\otimes c_s\,.
		\end{multline}
		If $A,B,C,D$ are increments, then
		\begin{multline}
		\label{eq:combinatorics_2}
		\delta [(s,t)\mapsto(A_{s,t}a_s-B_{s,t}b_s)\otimes (C_{s,t}c_s-D_{s,t}d_s)]_{s,u,t}
		\\=(A_{u,t}-B_{u,t})\Big(\mathbf 1_{n\times n}\otimes (D_{u,t}-C_{u,t})\delta(a\otimes d)_{s,u}-\big(\mathbf 1_{n\times n}\otimes C_{u,t}\big)\delta[a\otimes (c-d)]_{s,u}\Big)
		\\+B_{u,t}\otimes \mathbf 1_{n\times n}\Big(\mathbf 1_{n\times n}\otimes (D_{u,t}-C_{u,t})\delta[(a-b)\otimes d]_{s,u}-\mathbf 1_{n\times n}\otimes C_{u,t}(\delta[(a-b)\otimes (c-d)]_{s,u})\Big)
		\\+(C_{u,t}c_s-D_{u,t}d_s)\otimes (A_{s,u}a_s-B_{s,u}b_s)+(C_{s,u}c_s-D_{s,u}d_s)\otimes (A_{u,t}a_s-B_{u,t}b_s)\,.
		\end{multline}
	\end{lemma}
	
		\begin{remark}\label{rem:combinatorics}
		In coordinates, the relations \eqref{eq:combinatorics_1}-\eqref{eq:combinatorics_2} translate as
		\begin{equation}
		\begin{aligned}
		\delta [(Aa-Bb)^ic^j]_{s,u,t}=&-\sum_{l} ((A^{i,l}-B^{i,l})_{u,t}\delta [a^lc^j]_{s,u}-B^{i,l}_{u,t}\delta [(a^l-b^l)c^j]_{s,u})\\
		&+(\delta A_{s,u,t}a_s-\delta B_{s,u,t} b_s)^ic^j_s\,,
		\end{aligned}
		\end{equation}
		\begin{equation}
		\begin{aligned}
		\delta [(Aa-Bb&)^i(Cc-Dd)^j]_{s,u,t}=\\&-\sum_{l,m} (A^{i,l}_{u,t}-B^{i,l}_{u,t})((D^{j,m}_{u,t}-C^{j,m}_{u,t})\delta[a^ld^m]_{s,u}-C^{j,m}_{u,t}\delta[a^l(c^m-d^m)]_{s,u})\\
		&-\sum_{l,m}B^{i,l}_{u,t}((D^{j,m}_{u,t}-C^{j,m}_{u,t})\delta[(a^l-b^l)d^m]_{s,u}-C^{j,m}_{u,t}\delta[(a^l-b^l)(c^m-d^m)]_{s,u}))\\
		&+(A_{u,t}a_s-B_{u,t}b_s)^i(C_{s,u}c_s-D_{s,u}d_s)^j+(A_{s,u}a_s-B_{s,u}a_s)^i(C_{u,t}c_s-D_{u,t}d_s)^j.
		\end{aligned}
		\end{equation}
	\end{remark}	
	
	\begin{proof}
		Herein, we always assume that $s \leq u \leq t\in[0, T]$. 
		We only deal with $n=1,$ the general case being merely a matter of notations.
		
	For the first identity, we observe that
		\begin{align*}
		(A_{s,t}a_s-B_{s,t}b_s)c_s
		=(A-B)_{s,t}a_sc_s+B_{s,t}(a-b)_sc_s\,,
		\end{align*} 
		The claim follows by applying $\delta[\cdot]_{s,u,t}$ to both sides and using the relation \eqref{der_1}.
		
		In order to prove the second equality, we notice that the following algebraic relation holds:
		\begin{align*}
		(A_{s,t}a_s-B_{s,t}b_s)(C_{s,t}c_s-D_{s,t}d_s)=A_{s,t}C_{s,t}a_sc_s-A_{s,t}D_{s,t}a_sd_s-B_{s,t}C_{s,t}b_sc_s+B_{s,t}D_{s,t}b_sd_s\, .
		\end{align*}
		We apply $\delta[\cdot]_{s,u,t}$ to both sides, yielding
		\begin{align*}
		\delta[(Aa-Bb)(Cc-Dd)]_{s,u,t}=\delta (ACac-ADad)_{s,u,t}-\delta(BCbc-BDbd)_{s,u,t}\, .
		\end{align*}
		We compute only $\delta [ACac-ADad]_{s,u,t}$, since the second term has the same behaviour. We use the relation \eqref{der_1}
		\begin{align*}
		\delta [ACac-ADad]_{s,u,t}=&\delta (AC)_{s,u,t}a_sc_s-\delta (AD)_{s,u,t}a_sd_s\\
		&-(AC)_{u,t}\delta(ac)_{s,u}+(AD)_{u,t}\delta(ad)_{s,u}\, .
		\end{align*}
		From the relation \eqref{der_2},
		\begin{align*}
		\delta (AC)_{s,u,t}a_sc_s-\delta (AD)_{s,u,t}a_sd_s= (A_{s,u}C_{u,t}+A_{u,t}C_{s,u})a_sc_s-(A_{s,u}D_{u,t}+A_{u,t}D_{s,u})a_sd_s\, .
		\end{align*}
		It also holds that
		\begin{align*}
		-A_{u,t}C_{u,t}\delta(ac)_{s,u}+A_{u,t}D_{u,t}\delta(ad)_{s,u}&=A_{u,t}(D_{u,t}\delta(ad)_{s,u}-C_{u,t}\delta(ac)_{s,u})
		\\
		&=A_{u,t}((D_{u,t}-C_{u,t})\delta(ad)_{s,u}-C_{u,t}(\delta(ac)_{s,u}-\delta(ad)_{s,u}))
		\\&=A_{u,t}((D_{u,t}-C_{u,t})\delta(ad)_{s,u}-C_{u,t}(\delta(a(c-d))_{s,u}))\, .
		\end{align*}
		Hence
		\begin{align*}
		\delta((Aa-Bb)(Cc-Dd))_{s,u,t}&=(A_{s,u}C_{u,t}+A_{u,t}C_{s,u})a_sc_s-(A_{s,u}D_{u,t}+A_{u,t}D_{s,u})a_sd_s\\
		&\quad-(B_{s,u}C_{u,t}+B_{u,t}C_{s,u})b_sc_s+(B_{s,u}D_{u,t}+B_{u,t}D_{s,u})b_sd_s\\
		&\quad+A_{u,t}((D_{u,t}-C_{u,t})\delta(ad)_{s,u}-C_{u,t}(\delta(a(c-d))_{s,u}))\\
		&\quad-B_{u,t}((D_{u,t}-C_{u,t})\delta(bd)_{s,u}-C_{u,t}(\delta(b(c-d))_{s,u}))\, .
		\end{align*}
		This shows our claim, since
		\[
		\begin{aligned}
		&A_{u,t}((D_{u,t}-C_{u,t})\delta(ad)_{s,u}-C_{u,t}(\delta(a(c-d))_{s,u}))\\
		&\quad-B_{u,t}((D_{u,t}-C_{u,t})\delta(bd)_{s,u}-C_{u,t}(\delta[b(c-d)]_{s,u}))\\
		&=(A_{u,t}-B_{u,t})((D_{u,t}-C_{u,t})\delta(ad)_{s,u}-C_{u,t}(\delta[a(c-d)]_{s,u}))\\
		&\quad+B_{u,t}\big((D_{u,t}-C_{u,t})\delta((a-b)d)_{s,u}-C_{u,t}(\delta[(a-b)(c-d)]_{s,u})\big)\,.
		\quad \quad \quad \quad \quad \quad \quad \qedhere		
		\end{aligned}
		\]
	\end{proof}

	\subsection{An iterative formula for the remainder}
	\label{Section_rem_est_WZ}
	This is a technical section where we introduce some estimates that we have used in the proof existence and are needed to show the continuity of the It\^o-Lyons map.  We may look at this section as similar in spirit and objectives as Section \ref{Section_remainder_existence}. Indeed also in what follows we use the rough standard machinery \ref{RM1}--\ref{RM3} to estimate a remainder. Since the proof does not rely on the geometric properties of LLG equation, we expose it in an abstract framework.

	Let $f\in L^2(L^2)$,  $a_0\in H^k(\T ;\R ^3)$ and $\mathbf{A}=(A,\AA )\in \RD _a^p(H^{k+1})$ for $p\in [2,3)$ and $k\in\mathbb{N}$. Assume that there exists a solution $a\in L^\infty(H^k)\cap L^2(H^{k+1})$ to the abstract equation
	\begin{equation}\label{eq;abstract_eq}
	\delta a_{s,t}=\int_{s}^{t}f_r d r+A_{s,t}a_s+\AA _{s,t}a_s+a^\natural_{s,t}\, ,
	\end{equation}
	where $a^\natural$ is defined implicitly by the above equation and $a^\natural\in \mathcal{V}^{1-}_2([0,T];L^2)$. We denote a solution to \eqref{eq;abstract_eq} starting in $a_0\in H^k$ and driven by the rough driver $\mathbf{A}=(A,\AA )\in \RD ^p_a(H^{k+1})$ for $p\in [2,3)$, by $a=\pi(a_0;\mathbf{A})$. 
	
	Consider an other rough driver $\mathbf{B}=(B,\mathbb{B})\in \RD ^p_a(H^{k+1})$  and let $b=\pi(b_0;\mathbf{B})$ be an existing solutions to equation \eqref{eq;abstract_eq}.  i.e. there exists a solution $b\in L^\infty(H^k)\cap L^2(H^{k+1})$ to the abstract equation
	\begin{equation*}
	\delta b_{s,t}=\int_{s}^{t}g _rd r+B_{s,t}b_s+\mathbb{B}_{s,t}b_s+b^\natural_{s,t}\, ,
	\end{equation*}
	
	In the sequel, we will need to estimate the difference $a-b$ in the $L^\infty(H^k)\cap 
	L^2(H^{k+1})$ norm for any $k\ge0$, and for that purpose we aim to use the equation satisfied 
	by $\partial_x ^N (a-b)^{\otimes 2}$ for $N\in\mathbb{N}_0$.
 The aim of this paragraph is to find a recursive estimate for the remainder associated with this equation. 
 As in Section \ref{Section_remainder_existence}, the remainder is defined locally on $s\leq t\in P_k$, where $(P_k)_k$ is a partition of $[0,T]$: in order to simplify the notations, the times $s\leq t$ are dropped where not needed. In particular $(Ab)c$ would correspond to $(A_{s,t}b_s)c_s$, whereas $(Ab)(Cd) $ stands for $(A_{s,t}b_s)(C_{s,t}d_s)$.
	
	For the reasons alluded in Section \ref{Section_remainder_existence}, we need to do the 
	computations for the more general term $\partial_x ^N (a-b)\otimes \partial_x ^M (a-b)$, for 
	any $N,M\in\mathbb{N}_0$.
	Specifically, we will compute an estimate for the remainder term
	\begin{align}\label{eq:general_eq_der}
	(\partial_x ^N (a-b) \partial_x ^M (a-b))^{\natural,i,j}:=\delta (\partial_x ^N (a-b)^i & 
	\partial_x ^M (a-b)^j)-(\D +I+\II+\tilde\II)^{i,j}(N,M)\, ,
	\end{align}
where
	\begin{equation}\label{Drift_NM}
	\begin{aligned}
	\D ^{i,j}(N,M)_{s,t}
	&:=\int_s^t (\partial_x ^N(f_r-g_r)^i\partial_x ^M(a_r-b_r)^j + \partial_x 
	^N(a_r-b_r)^i\partial_x ^M(f_r-g_r)^j ) d r\, ,
	\end{aligned}
	\end{equation}
and
	\begin{align}
	I&:={\sum_{n=0}^{N} \binom{N}{n} (\partial_x ^{N-n}A\partial_x ^{n} a-\partial_x 
	^{N-n}B\partial_x ^{n} b)^i \partial_x ^M(a-b)^j}
	\label{I_1}
	\\
	&\quad 
	+\sum_{m=0}^{M} \binom{M}{m}\partial_x ^N(a-b)^i(\partial_x ^{M-m}A\partial_x ^{m} a-\partial_x 
	^{M-m}B\partial_x ^{m} b)^j
	\label{I_2}
	\\&
	=I_1+I_2\,,
	\nonumber
	\\
	\II&:=\sum_{n=0}^{N} \binom{N}{n}(\partial_x ^{N-n}\AA \partial_x ^{n} a-\partial_x 
	^{N-n}\mathbb{B}\partial_x ^{n} b)^i \partial_x ^M(a-b)^j,
	\label{I_3}
	\\
	&\quad +\sum_{m=0}^{M} \binom{M}{m}\partial_x ^N(a-b)^i(\partial_x ^{M-m}\AA \partial_x ^{m} 
	a-\partial_x ^{M-m}\mathbb{B}\partial_x ^{m} b)^j 
	\label{I_4}
	\\&
	=\II_3 + \II_4\, ,
	\nonumber
	\\
	\tilde\II&:=\sum_{n=0}^{N} \binom{N}{n} (\partial_x ^{N-n}A\partial_x ^{n} a-\partial_x 
	^{N-n}B\partial_x ^{n} b)^i\sum_{m=0}^{M} \binom{M}{m}(\partial_x ^{M-m}A\partial_x ^{m} 
	a-\partial_x ^{M-m}B\partial_x ^{m} b)^j\, .
	\label{I_5}
	\end{align}
	Note that the term $I$ has bounded $p$-variation, whereas $\II,\tilde\II$ have bounded $p/2$-variation. 
	We denote these elements of only $\mathcal{V}^p$-variation of the equation $\partial_x ^N (a-b) 
	\partial_x ^M (a-b)$ by $T_{\partial_x ^N (a-b) \partial_x ^M (a-b)}$. More generally, consider 
	two solutions $a,b$ as above and consider the product equation: the elements of only 
	$\mathcal{V}^p$-variation are denoted by $T_{ab}$.
	
    Let $c,d,e,f:[0,T]\rightarrow \mathbb{R}^3$ and $C,D,E,F:[0,T]\rightarrow \mathcal{L}(\mathbb{R}^3)$, we introduce the following notation:
    \begin{equation}
    \begin{aligned}
    \boldsymbol\gamma (C,D,E,F,c,d,e,f)
    &:=-\sum_{l,m} (C^{i,l}_{u,t}-D^{i,l}_{u,t})((F^{j,m}_{u,t}-E^{j,m}_{u,t})\delta(c^lf^m)_{s,u}-E^{j,m}_{u,t}(\delta(c^l(e^m-f^m))_{s,u}))\nonumber\\
    &\quad-\sum_{l,m}D^{i,l}_{u,t}(F^{j,m}_{u,t}-E^{j,m}_{u,t})\delta((c^l-d^l)f^m)_{s,u}\, .
    \end{aligned}
    \end{equation}
	The objective is to write a general formula for $\delta (\partial_x ^N (a-b) \partial_x ^M 
	(a-b))^{\natural,i,j}_{s,u,t}$ and we make this computation now in Theorem 
	\ref{teo:rem_derivatives}. The proof is technical like the result, but we can no more postpone 
	it.
	\begin{theorem}\label{teo:rem_derivatives}
		Let $z:=a-b$, $Z:=A-B$ and $\mathbb{Z}:=\AA -\mathbb{B}$. In the notations of this section, 
		the quantity $\delta (\partial_x ^N (a-b) \partial_x ^M (a-b))^{\natural,i,j}_{s,u,t}$ 
		corresponding to equation \eqref{eq:general_eq_der} can be rewritten as 
		\begin{align}
		&+\sum_{n=0}^{N} \binom{N}{n} \sum_{l} \partial_x ^{N-n}Z^{i,l}_{u,t}(\delta (\partial_x 
		^{n} a^l\partial_x ^Mz ^j)_{s,u}-T_{\partial_x ^{n} a^l\partial_x 
		^Mz^j})\label{Final_I_1_a}\\
		&-\sum_{n=0}^{N} \binom{N}{n} \sum_{l}\partial_x ^{N-n}B^{i,l}_{u,t}(\delta(\partial_x 
		^{n}z^l \partial_x ^Mz^j)_{s,u}-T_{\partial_x ^{n}z^l  \partial_x 
		^Mz^j})\label{Final_I_1_b}\\
		&+\sum_{m=0}^{M} \binom{M}{m} \sum_{l} \partial_x ^{M-m}Z^{j,l}_{u,t}(\delta (\partial_x 
		^{m} a^l\partial_x ^Nz^i)_{s,u}-T_{\partial_x ^{m} a^l\partial_x 
		^Nz^i})\label{Final_I_2_a}\\
		&-\sum_{m=0}^{M} \binom{M}{m} \sum_{l}\partial_x ^{M-m}B^{j,l}_{u,t}(\delta(\partial_x 
		^{m}z^l \partial_x ^Nz^i)_{s,u}-T_{\partial_x ^{m}z^l \partial_x 
		^Nz^i})\label{Final_I_2_b}\\
		&+\sum_{n=0}^{N} \binom{N}{n}\sum_{l}\big( \partial_x ^{N-n}\mathbb{B}^{i,l}_{u,t}\delta 
		(\partial_x ^{n}z^l\partial_x ^M z^j)_{s,u}-\partial_x^{N-n} \mathbb{Z}^{i,l}_{u,t}\delta 
		(\partial_x^n a^l\partial_x^M z^j)_{s,u} \big)
		 \label{Final_I_3}\\
		&+\sum_{m=0}^{M} \binom{M}{m}\sum_{l} \big(\partial_x ^{M-m}\mathbb{B}^{j,l}_{u,t}\delta 
		(\partial_x ^{m}z^l\partial_x ^N z^i)_{s,u}- \partial_x^{M-m} \mathbb{Z}^{j,l}_{u,t}\delta 
		(\partial_x^m a^l\partial_x^N z^i)_{s,u}\big)
		 \label{Final_I_4}\\
		&+\sum_{n=0}^{N} \sum_{m=0}^{M}\binom{N}{n}\binom{M}{m} \boldsymbol\gamma (\partial_x ^{N-n} (A- 
		B)^i, \partial_x ^{M-m}(A-B)^j) \label{Final_I_5_a}\\
		&-\sum_{n=0}^{N} \sum_{m=0}^{M}\sum_{l,\ell}\binom{N}{n}\binom{M}{m}\partial_x ^{N-n} 
		B^{i,l}_{u,t}\partial_x ^{M-m}A^{j,\ell}_{u,t}\delta(\partial_x ^{n}z^l\partial_x 
		^{m}z^\ell)_{s,u}\label{Final_I_5_b}\, ,
		\end{align}
		where
		$\boldsymbol\gamma (\partial_x ^{N-n} (A- B)^i,\partial_x ^{M-m}(A-B)^j)$ is a shorthand notation for
		\[
		\boldsymbol\gamma (\partial_x 
		^{N-n}A,\partial_x ^{N-n}B,\partial_x ^{M-m} A,\partial_x ^{M-m} B,\dd^n a,\dd^n b,\dd^m 
		a,\dd^m b) \, .
		\]
	\end{theorem}
\begin{remark}
	Note that $\delta (\partial_x ^N (a-b) \partial_x ^M (a-b))^{\natural,i,j}_{s,u,t}$ can be 
	written as sum of terms are all of $\mathcal V^{p/3}$-variation or higher regularity.
\end{remark}
	\begin{proof}
		Note that the Chen's relation implies that
		\begin{align} \label{eq: Chen_multidim}
		\delta (\partial_x ^{m} \AA )_{s,u,t}=\partial_x ^{m} \delta (\AA )_{s,u,t}=\partial_x 
		^{m}(A_{u,t}A_{s,u})=\sum_{l=0}^{m}  \binom{m}{l} \partial_x ^{m-l} A_{u,t}\partial_x ^{l} 
		A_{s,u}\, ,
		\end{align}
		and the same relation holds between $B$ and $\mathbb{B}$.
		By applying \eqref{eq:combinatorics_1} to $ \delta (\II_3)_{s,u,t}$, the elements that are on only $\mathcal V^{p/2}$- variation are
		\begin{align*}
		\sum_{n=0}^{N} \binom{N}{n}(\delta \partial_x ^{N-n}\AA _{s,u,t}\partial_x ^{n} a_s-\delta 
		\partial_x ^{N-n}\mathbb{B}_{s,u,t}\partial_x ^{n}b_s)^i(\partial_x ^Ma^j_s-\partial_x 
		^Mb^j_s)\, ,
		\end{align*}
		where by substituting the explicit form of $\delta \partial_x ^{N-n}\AA _{s,u,t}$ and 
		$\delta \partial_x ^{N-n}\mathbb{B}_{s,u,t}$ as in \eqref{eq: Chen_multidim}  and by 
		recalling that  $\binom{N}{n}=\binom{N}{N-n}$ the last line becomes
		\begin{equation}
		\begin{aligned}\label{Compensation_I_3}
		{\sum_{n=0}^{N} \binom{N}{N-n}\sum_{l=0}^{N-n} \binom{N-n}{l} (\partial_x ^{N-n-l} 
		A_{u,t}\partial_x ^{l} A_{s,u}\partial_x ^{n}a_s-\partial_x ^{N-n-l} B_{u,t}\partial_x ^{l} 
		B_{s,u}\partial_x ^{n}b_s)^i(\partial_x ^M a^j_s-\partial_x ^M b^j_s)}\, . 
		\end{aligned}
		\end{equation}
		The terms of $ \delta (\II_3)_{s,u,t}$ that contribute to the final estimate are in line \eqref{Final_I_3}. From $\delta (\II_4)_{s,u,t}$ in the last line we get similarly
		\begin{align}\label{Compensation_I_4}
		{\sum_{m=0}^{M} \binom{M}{m}\sum_{l=0}^{m} \binom{m}{l} (\partial_x ^{m-l} 
		A_{u,t}\partial_x ^{l} A_{s,u}\partial_x ^{m}a_s-\partial_x ^{m-l} B_{u,t}\partial_x ^{l} 
		B_{s,u}\partial_x ^{m}b_s)^j(\partial_x ^N a^i_s-\partial_x ^N b^i_s)}\, ,
		\end{align}
		and the contributions to the final estimate are in line \eqref{Final_I_4}.
		
		By applying \eqref{eq:combinatorics_2} to $\delta (\tilde\II)_{s,u,t}$, what needs to be compensated is
		\begin{align}
		&{+\sum_{n=0}^{N} \binom{N}{n} (\partial_x ^{N-n}A_{u,t}\partial_x ^{n} a_s-\partial_x 
		^{N-n}B_{u,t}\partial_x ^{n} b_s)^i  \sum_{m=0}^{M} \binom{M}{m} (\partial_x 
		^{M-m}A_{s,u}\partial_x ^{m} a_s-\partial_x ^{M-m}B_{s,u}\partial_x ^{m} b_s)^j} 
		\label{Compensation_I_5_a}\\
		&{+\sum_{n=0}^{N} \binom{N}{n} (\partial_x ^{N-n}A_{s,u}\partial_x ^{n} a_s-\partial_x 
		^{N-n}B_{s,u}\partial_x ^{n} b_s)^i  \sum_{m=0}^{M} \binom{M}{m} (\partial_x 
		^{M-m}A_{u,t}\partial_x ^{m} a_s-\partial_x ^{M-m}B_{u,t}\partial_x ^{m} b_s)^j}\, , 
		\label{Compensation_I_5_b}
		\end{align}
		and the contributions to the final estimates are in \eqref{Final_I_5_a} and  \eqref{Final_I_5_b}.
		
		Now we need to address \eqref{I_1} and \eqref{I_2}: we apply to both \eqref{eq:combinatorics_1} and we notice that no element is of $\mathcal V^{p/3}$-variation. Indeed, for what concerns  \eqref{I_1}, 
		\begin{align}
		\delta (I_1)_{s,u,t}&=
		-\sum_{n=0}^{N} \binom{N}{n} \sum_{l} (\partial_x ^{N-n}A^{i,l}_{u,t}-\partial_x 
		^{N-n}B^{i,l}_{u,t})\delta (\partial_x ^{n} a^l\partial_x ^M(a-b)^j)_{s,u} 
		\label{eq_1_I_1}\\
		&\quad+\sum_{n=0}^{N} \binom{N}{n} \sum_{l}\partial_x ^{N-n}B^{i,l}_{u,t}\delta(\partial_x 
		^{n}(a^l-b^l)  \partial_x ^M(a-b)^j)_{s,u} \label{eq_2_I_1} \, .
		\end{align}
		Analogously for  \eqref{I_2}. For this reason we need to find a compensation so that there appears a term of $\mathcal V^{p/3}$-variation that contributes to the final equality and so that the other elements appearing cancel with \eqref{Compensation_I_3}, \eqref{Compensation_I_4}, \eqref{Compensation_I_5_a} and \eqref{Compensation_I_5_b}. 
		
		We continue to work on $\delta (I_1)_{s,u,t}$ and we notice that we can act on $\delta 
		(\partial_x ^{n} a^l\partial_x ^M(a-b)^j)_{s,u}$ in \eqref{eq_1_I_1} and subtract the term 
		of only $\mathcal V^{p}$-variation, namely
		\begin{align*}
		T_1:=\sum_{p=0}^{n} \binom{n}{p}(\partial_x ^{n-p}A_{s,u}\partial_x ^{p} a_s)^l \partial_x 
		^M (a-b)^j_s+\sum_{m=0}^{M} \binom{M}{m}(\partial_x ^{M-m}A_{s,u}\partial_x ^{m} 
		a_s-\partial_x ^{M-m}B_{s,u}\partial_x ^{m} b_s)^j \partial_x ^{n} a^l_s.
		\end{align*}
		In this way, $\delta (\partial_x ^{n} a^l\partial_x ^M(a-b)^j)_{s,u}-T_1$ presents only 
		terms of $\mathcal V^{p/2}$-variation and therefore 
		\begin{align*}
		-\sum_{n=0}^{N} \binom{N}{n} \sum_{l} (\partial_x ^{N-n}A^{i,l}_{u,t}-\partial_x 
		^{N-n}B^{i,l}_{u,t})(\delta (\partial_x ^{n} a^l\partial_x ^M(a-b)^j)_{s,u}-T_1)
		\end{align*}
		is of $\mathcal V^{p/3}$-variation and appears in the final estimate in \eqref{Final_I_1_a}.
		We substitute $T_1$ and we notice that 
		\begin{align}
		&\quad-\sum_{n=0}^{N} \binom{N}{n} \sum_{l} (\partial_x ^{N-n}A^{i,l}_{u,t}-\partial_x 
		^{N-n}B^{i,l}_{u,t})	\sum_{p=0}^{n} \binom{n}{p}(\partial_x ^{n-p}A_{s,u}\partial_x ^{p} 
		a_s)^l \partial_x ^M (a-b)^j_s\nonumber\\
		&\quad-\sum_{n=0}^{N} \binom{N}{n} \sum_{l}(\partial_x ^{N-n}A^{i,l}_{u,t}-\partial_x 
		^{N-n}B^{i,l}_{u,t})\sum_{m=0}^{M} \binom{M}{m}(\partial_x ^{M-m}A_{s,u}\partial_x ^{m} 
		a_s-\partial_x ^{M-m}B_{s,u}\partial_x ^{m} b_s)^j \partial_x ^{n} a^l_s\nonumber\\
		&={-\sum_{n=0}^{N} \binom{N}{n} \sum_{p=0}^{n} \binom{n}{p}((\partial_x 
		^{N-n}A_{u,t}-\partial_x ^{N-n}B_{u,t})\partial_x ^{n-p}A_{s,u}\partial_x ^{p} a_s)^i 
		\partial_x ^M (a-b)^j_s}\label{Compensation_I_1_a}\\
		&\quad{-\sum_{n=0}^{N} \binom{N}{n} ((\partial_x ^{N-n}A_{u,t}-\partial_x 
		^{N-n}B_{u,t})\partial_x ^{n} a_s)^i\sum_{m=0}^{M} \binom{M}{m}(\partial_x 
		^{M-m}A_{s,u}\partial_x ^{m} a_s-\partial_x ^{M-m}B_{s,u}\partial_x ^{m} b_s)^j 
		}\label{Compensation_I_1_b}.
		\end{align}
		Now consider $\delta(\partial_x ^{n}(a^l-b^l)  \partial_x ^M(a-b)^j)_{s,u}$ in 
		\eqref{eq_2_I_1} and by subtracting $T_2$ defined as
		\begin{align*}
		T_2&:=\sum_{p=0}^{n} \binom{n}{p}(\partial_x ^{n-p}A_{s,u}\partial_x ^{p} a_s-\partial_x 
		^{n-p}B_{s,u}\partial_x ^{p} b_s)^l \partial_x ^M (a-b)^j_s\\
		&\quad+\sum_{m=0}^{M} \binom{M}{m}(\partial_x ^{M-m}A_{s,u}\partial_x ^{m} a_s-\partial_x 
		^{M-m}B_{s,u}\partial_x ^{m} b_s)^j (\partial_x ^{n} a^l_s-\partial_x ^{n} b^l_s)\, ,
		\end{align*}
		we can conclude that $\delta(\partial_x ^{n}(a^l-b^l)  \partial_x ^M(a-b)^j)_{s,u}-T_2$ 
		presents terms that are all of $\mathcal V^{p/2}$-variation and therefore 
		\eqref{Final_I_1_b} contributes to the final equality.
		The new term appearing from adding and subtracting $T_2$ is
		\begin{align}
		&\quad-\sum_{n=0}^{N} \binom{N}{n} \sum_{l} \partial_x ^{N-n}B^{i,l}_{u,t}\sum_{p=0}^{n} 
		\binom{n}{p}(\partial_x ^{n-p}A_{s,u}\partial_x ^{p} a_s-\partial_x ^{n-p}B_{s,u}\partial_x 
		^{p} b_s)^l \partial_x ^M (a-b)^j_s\nonumber\\
		&\quad-\sum_{n=0}^{N} \binom{N}{n}  \sum_{l} \partial_x ^{N-n}B^{i,l}_{u,t}\sum_{m=0}^{M} 
		\binom{M}{m}(\partial_x ^{M-m}A_{s,u}\partial_x ^{m} a_s-\partial_x ^{M-m}B_{s,u}\partial_x 
		^{m} b_s)^j (\partial_x ^{n} a^l_s-\partial_x ^{n} b^l_s)\nonumber\\
		&={-\sum_{n=0}^{N} \binom{N}{n}  \sum_{p=0}^{n} \binom{n}{p}(\partial_x 
		^{N-n}B_{u,t}\partial_x ^{n-p}A_{s,u}\partial_x ^{p} a_s-\partial_x ^{N-n}B_{u,t}\partial_x 
		^{n-p}B_{s,u}\partial_x ^{p} b_s)^i \partial_x ^M (a-b)^j_s} \label{Compensation_I_1_c}\\
		&\quad{-\sum_{n=0}^{N} \binom{N}{n} \partial_x ^{N-n}B_{u,t}(\partial_x ^{n} a_s-\partial_x 
		^{n} b_s)^i\sum_{m=0}^{M} \binom{M}{m}(\partial_x ^{M-m}A_{s,u}\partial_x ^{m} 
		a_s-\partial_x ^{M-m}B_{s,u}\partial_x ^{m} b_s)^j ) } \label{Compensation_I_1_d}\, .
		\end{align}
		The sum of \eqref{Compensation_I_1_b} and \eqref{Compensation_I_1_d} cancels with \eqref{Compensation_I_5_a}. The sum of \eqref{Compensation_I_1_a} and \eqref{Compensation_I_1_c} gives
		\begin{align}
		&\quad {-\sum_{p=0}^{N} \sum_{n=p}^{N} \binom{N}{n}   \binom{n}{p}(\partial_x 
		^{N-n}A_{u,t}\partial_x ^{n-p}A_{s,u}\partial_x ^{p} a_s-\partial_x ^{N-n}B_{u,t}\partial_x 
		^{n-p}B_{s,u}\partial_x ^{p} b_s)^i \partial_x ^M (a-b)^j_s}\nonumber\\
		&={-\sum_{p=0}^{N} \sum_{l=0}^{N-p} \binom{N}{l+p}   \binom{l+p}{p}(\partial_x 
		^{N-n}A_{u,t}\partial_x ^{n-p}A_{s,u}\partial_x ^{p} a_s-\partial_x ^{N-n}B_{u,t}\partial_x 
		^{n-p}B_{s,u}\partial_x ^{p} b_s)^i \partial_x ^M (a-b)^j_s}\label{Comepnsation_I_1_Fin}\, ,
		\end{align}
		where in the last equality we applied the change on indices $l:=n-p$.
		Now note that, if $p=n$, then
		\begin{align*}
		\binom{N}{N-n}\binom{N-n}{l}=\frac{N!}{n!l!(N-n-l)!}=\binom{N}{l+n}\binom{l+n}{l}\, ,
		\end{align*}
		and therefore \eqref{Comepnsation_I_1_Fin} cancels with \eqref{Compensation_I_3}.
		We can proceed similarly with \eqref{I_2}. First consider
		\begin{align*}
		\delta ({I_2})_{s,u,t}&={-\sum_{m=0}^{M} \binom{M}{m}\sum_{l} ((\partial_x 
		^{M-m}A^{j,l}-\partial_x ^{M-m}B^{j,l})_{u,t}\delta (\partial_x ^{m}a^l\partial_x 
		^N(a^i-b^i))_{s,u})} \\
		&\quad{+\sum_{m=0}^{M} \binom{M}{m}\sum_{l}\partial_x ^{M-m}B^{j,l}_{u,t}\delta (\partial_x 
		^{m}(a^l-b^l)\partial_{N}(a^i-b^i))_{s,u}} 
		\end{align*}
		and analogous steps as above lead to
		\begin{align}
		&{-\sum_{p=0}^{M}\sum_{l=0}^{M-p} \binom{M}{l+p}   \binom{l+p}{p}(\partial_x 
		^{M-m}A_{u,t}\partial_x ^{m-p}A_{s,u}\partial_x ^{p} a_s-\partial_x ^{M-m}B_{u,t}\partial_x 
		^{m-p}B_{s,u}\partial_x ^{p} b_s)^i \partial_x ^N (a-b)^i_s} \label{Comepnsation_I_2_a}\\
		&{-\sum_{m=0}^{M} \binom{M}{m} (\partial_x ^{M-m}A_{u,t}\partial_x ^{m} a_s-\partial_x 
		^{M-m}B_{u,t}\partial_x ^{m} b_s)^j\sum_{n=0}^{N} \binom{N}{n}(\partial_x 
		^{N-n}A_{s,u}\partial_x ^{n} a_s-\partial_x ^{N-n}B_{s,u}\partial_x ^{n} b_s)^i 
		}\label{Compensation_I_2_b}.
		\end{align}
		Now we notice that \eqref{Comepnsation_I_2_a} cancels with \eqref{Compensation_I_4}. This concludes the proof of Theorem \ref{teo:rem_derivatives}.
	\end{proof}

	\section{The perturbed equation with a regular input}
	\label{app:perturbed}
	
	In this section, we consider the problem with (sufficiently) regular input $\xi_t(x)$ taking values in $\mathcal L_a(\R^3)$, and we prove that the associated solution $u^\xi$ exists and is unique among a certain class.
	The following result states that for every such $\xi$ there exists a unique solution to
	\begin{equation} \label{rLLGxi}
	\left\{\begin{aligned}
	&du= [\partial_x ^{2} u+ u|\partial _x  u|^2 + u\times \partial_x ^{2} u]dt +  \xi u\,,
	\quad \text{on}\enskip [0,T]\times \T ,
	\\
	&u(0)=u^{0}\in H^1(\T ;\R^3)\,.
	\end{aligned}\right.
	\end{equation}

	\begin{proposition}\label{pro:B1}
	Let $\xi\in L^\infty(H^1)$ and $u^0\in H^1(\T ;\mathbb{S}^2)$, then there exists a unique solution $u^\xi\in L^\infty(H^1)\cap L^2(H^2)\cap H^1(L^2) $ to the deterministic equation \eqref{rLLGxi} such that $|u^\xi|_{\R^3}=1$ a.e. Moreover, if $u^0\in H^k(\T ;\mathbb{S}^2)$ and $\xi\in W^{1,\infty}(H^{k+1})$, then $u^\xi\in L^\infty(H^k)\cap L^2(H^{k+1})\cap H^1(H^{k-1})$ for $k\geq 2$.
	\end{proposition}
	\begin{proof}
		The  proof is based on the implicit mapping theorem (see for instance \cite{lang2012fundamentals} for a Banach space version). Let $u^0\in H^1(\T ,\mathbb S^2)$ and consider the spaces 
		\[ \begin{aligned}
		&X:= L^\infty (H^1), \quad &Y:= L^\infty(H^1) \cap L^2(H^{2})\cap H^{1}(L^{2}) , \quad &Z:= L^2(L^2)\times H^1.
		\end{aligned}
		\]
		We define a map $F:X\times Y \to Z$ by
		\[ \begin{aligned}
		F(\xi,y):= 
		\begin{bmatrix}
		\partial_ty - \tt_y - \partial_x(y\times \partial _x y) - \xi y
		\\
		y_0-u^0
		\end{bmatrix}
		\end{aligned} \, .
		\]
		
		Note that $F(0,y)=0$ if and only if $y$ is a solution to the unperturbed LLG equation, with initial datum $u^0$. For this equation, it follows from the classical theory (see e.g.\ \cite{Zhu_Guo_Tan}) that there exists a unique smooth solution and we denote it by $u$. More precisely, if the initial condition belongs to $H^k$, then the solution lies in $L^\infty(H^k)\cap L^2(H^{k+1})$.
		
		Let us show that $F$ is continuously Fr\'echet differentiable in an open neighbourhood of $(0,u)$, denoted by $U\subset X\times Y$. 
		Recall that $F$ is continuously Fr\'echet differentiable  at a point $v=(v_1,v_2)\in U$ if there exists a continuous linear map $\lambda :X\times Y \rightarrow Z$ and  functions $\phi:U\rightarrow Z$, $\psi:U\to \R ^{+}$ so that $|\phi (v)|_{Z}\leq |v|_{X\times Y}\psi(v)$ and $\lim_{|v|_{X\times Y}\rightarrow 0}\psi(v)=0$, and satisfying the property 
		\begin{align*}
		F(v+h)=F(v)+\lambda h+\phi(h)
		\end{align*}
		for all $h=(h_1,h_2)\in X\times Y$ so that $|h|\rightarrow 0$.  In order to define $\lambda, \phi$, we compute the first component 
		\begin{align*}
		F^1(v+h)-F^1(v)&= \partial_t h_2 -\partial_x ^{2} h_2-(v_2+h_2)|\partial_x 
		(v_2+h_2)|^2+v_2|\partial_x v_2|^2-(\xi+h_1)(v_2+h_2) +\xi v_2\,.
		\end{align*}
		We can define the continuous and linear functional $\lambda :X\times Y \rightarrow Z$ (continuous because it is a linear and bounded functional between Banach spaces) 
		\begin{align*}
		\lambda h:= \begin{bmatrix}
		\partial_t h_2 - \partial_x ^{2} h_2 - 2 v_2 (\partial _x v_2,\partial_x h_2 ) 
		-h_2|\partial_x v_2|^2-\partial_x (v_2\times \partial _x h_2)-\partial_x (h_2\times 
		\partial_x v_2)-h_1v_2- \xi h_2
		\\
		h_0
		\end{bmatrix}\, ,
		\end{align*}
		 and the auxiliary function $\phi$ by 
		\begin{align*}
		\phi (v,h):=-v_2|\partial_x h_2|^2-2 h_2 \partial_x h_2\cdot\partial_x v_2-h_2|\partial_x 
		h_2|^2-\partial_x (h_2\times\partial_x h_2)-h_1\times h_2\, .
		\end{align*}
		Note that, with the above definitions, we can say that $F$ is of class $C^{1}$ in any  open neighbourhood $U$ of $(0,u)$.
		
		Moreover, at $(0,u)$ its partial derivative in the direction $(0,h)$ for $h\in Y$ is given by 
		\[ 
		D_2F(0,u)[h]= \begin{bmatrix}
		\mathscr Lh
		\\
		h_0
		\end{bmatrix}\, ,
		\]
		where $D_2$ is the second derivative of $F$ with respect to the second component and $\mathscr L$ is the linear operator
		\[ \mathscr L_th:= \partial _th - \partial_x ^{2} h -\partial_x (u_t \times \partial _x h) 
		-  \partial_x (h\times \partial_x u)
		-2u\partial_x u\cdot \partial_x h - |\partial_x u|^2 h \,. \]
		We aim to show that $D_2F(0,u)\colon Y\to Z$ is a continuous isomorphism and for that purpose, let us 
		introduce the family of elliptic operators $(L_t)_{t\in[0,T]}$ on $L^2(\T ;\R^3)$ defined as
		\[ 
		L_th:= -a(t,x)\partial_x ^2h(x) \,,
		\quad
		a(t,x):= \mathbf 1_{3\times3}- (u(t,x)\times\cdot) 
		\]
		(with constant domain $D(L_t)\equiv H^2(\T ;\R^3)$).
		The matrix $a$ being uniformly elliptic, one sees from standard parabolic theory (see, e.g., \cite[Theorem 6.6]{amann1985global}) that  $L_t$ is for each $t\in[0,T]$ the generator of an analytic semigroup.
		Now, the same holds true for $\mathscr L$ because the difference $K:=\mathscr L-L$ is a differential operator of lower order (with smooth coefficients), and as such, is a Kato perturbation of $\mathscr L$. Hence our assertion by \cite[Corollary 6.9]{goldstein}.
		Hence, for every $(f,z)\in Z$, classical PDE results (\cite{grisvard1969equations}) imply that the linear equation
		\[ \left \{\begin{aligned}
		&\mathscr Lh =f
		\\
		& h_0=z\enskip,
		\end{aligned}\right .\]
		with unknown $h\in Y$, is well-posed.
		
		By the implicit mapping Theorem, there is an open neighbourhood $U$ of $0$ in $X$ and a map $\varphi \in C^1(U\to Y)$ so that 
		\begin{equation}\label{equation_F}
		F(\xi,y)=0\enskip\text{for}\enskip(\xi,y)\in U\times Y
		\quad \Leftrightarrow \quad 
		y=\varphi(\xi)\, .
		\end{equation}	
		To see why $U$ can be taken as the full space $X,$ let us consider $(U,\varphi)$ to be maximal among the class of all such pairs satisfying \eqref{equation_F} (that such a maximal element exists follows from Zorn's Lemma for the order relation $(U,\varphi)\leqslant (V,\psi)$ $\Leftrightarrow$ $U\subset V$ and $\psi|_U=\varphi$).
		In particular, $U$ is an open set.
		It is enough to show that $U$ is closed, since being closed and open for a set $S\subset X$ implies $S=X$.
		But if $\xi^n\to \xi \in X$, an argument similar to those used in Section \ref{sec:existence_uniqueness} shows that the limit $y=\lim y^{\xi^n}$ exists in $Y$ and solves \eqref{rLLGxi}. This means that $\varphi(\xi)=y$, consequently $U$ is closed and therefore $U=X.$

		Finally, if the initial condition belongs to $H^k$, the above argument can be carried out by simply shifting the indices of the corresponding Sobolev spaces, leading to existence and uniqueness for the perturbed equation in $L^\infty(H^k)\cap L^2(H^{k+1})$
		(in this case we assume also the noise to belong to $L^\infty(H^k)$). 
		Since the proof is so similar, we leave the details the reader.
	\end{proof}


	\bibliographystyle{plain}

\end{document}